\documentclass[11pt,reqno]{amsart}

\usepackage{esint}
\usepackage{comment}
\usepackage{amsmath}
\usepackage{amsthm}
\usepackage{mathrsfs}
\usepackage{graphicx}
\usepackage{mathtools}
\usepackage{amsfonts}
\usepackage{amssymb}
\usepackage[T1]{fontenc}
\usepackage{hyperref}
\usepackage[margin=1 in]{geometry}
\usepackage{enumerate}
\usepackage{cite}
\usepackage[normalem]{ulem}
\usepackage{tikz}
\usepackage{tikz-cd}
\usepackage{xcolor}
\usetikzlibrary{matrix,arrows,positioning,automata}
  \usepackage{cancel}
\usepackage{dsfont}

  \usepackage{scalerel,stackengine}
\stackMath
\newcommand\reallywidehat[1]{%
\savestack{\tmpbox}{\stretchto{%
  \scaleto{%
    \scalerel*[\widthof{\ensuremath{#1}}]{\kern-.6pt\bigwedge\kern-.6pt}%
    {\rule[-\textheight/2]{1ex}{\textheight}}
  }{\textheight}%
}{0.5ex}}%
\stackon[1pt]{#1}{\tmpbox}%
}

\newcommand{\N}{\mathbb N}

\newcommand{\R}{\mathbb R}
\def\E{\mathbb E}

\newcommand{\bx}{\bm{x}}
\newcommand{\bz}{\bm{z}}
\newcommand{\bX}{\boldsymbol X}

\DeclareMathOperator{\dist}{dist}

\def\XXint#1#2#3{{\setbox0=\hbox{$#1{#2#3}{\int}$}
\vcenter{\hbox{$#2#3$}}\kern-.5\wd0}}

\numberwithin{equation}{section}
\newtheorem{thm}{Theorem}
\newtheorem{lem}[thm]{Lemma}
\newtheorem{cor}[thm]{Corollary}
\newtheorem{prop}[thm]{Proposition}

\theoremstyle{definition}
\newtheorem{defn}[thm]{Definition}
\newtheorem{rmk}[thm]{Remark}
\newtheorem{example}[thm]{Example}

\numberwithin{thm}{section}

\def\smallnegint{\mathop{\int\mkern-13mu
        \raise.5ex\hbox{${\scriptscriptstyle\diagup}$}}\nolimits}
\def\ds{\displaystyle}

\def\bx{{\boldsymbol x}}

\newcommand{\be}{\begin{equation}}
\newcommand{\ee}{\end{equation}}
\newcommand{\bc}{\begin{case}}
\newcommand{\ec}{\end{cases}}
\newcommand{\bs}{\begin{split}}
\newcommand{\es}{\end{split}}

\newcommand{\norm}[1]{\left\Vert#1\right\Vert}

\newcommand{\bm}[1]{\boldsymbol #1}

\renewcommand{\tilde}{\widetilde}
\renewcommand{\hat}{\widehat}

\DeclareMathOperator{\supp}{supp}

\def \be {\begin{equation}}
\def \ee {\end{equation}}

\def \E {\mathbb{E}}

\def \R {\mathbb{R}}

\renewcommand{\tilde}{\widetilde}

\newcommand{\cC}{\mathcal{C}}

\newcommand{\cP}{\mathcal{P}}
\newcommand{\ov}{\overline}
\newcommand{\cL}{\mathcal{L}}
\newcommand{\bd}{\bm d}

\newcommand{\wt}{\widetilde}
\newcommand{\eps}{\epsilon}

\newcommand{\sub}{\cP_{\text{sub}}}

\begin{document}

\title[Uniform in time weak convergence for Fleming-Viot]{Uniform in time weak convergence for a Fleming-Viot particle system with hard killing}

\author[P. Cardaliaguet]{Pierre Cardaliaguet
\address{(P. Cardaliaguet) Universit\'e Paris Dauphine-PSL, Place du Mar\'echal de Lattre de Tassigny, 75016 Paris, France
}\email{cardaliaguet@ceremade.dauphine.fr
}}

\author[M.\ Cirant]{Marco Cirant\address{(M. Cirant) Dipartimento di Matematica ``T.\ Levi-Civita'', Università degli Studi di Padova, Via Trieste 63, 35121 Padova, Italy}
\email{cirant@math.unipd.it}}

\author[J. Jackson]{Joe Jackson\address{(J. Jackson) The University of Chicago, Eckhart Hall, 5734 S University Ave, Chicago, IL 60637, USA}
\email{jsjackson@uchicago.edu}}
 
\author[P.E. Souganidis]{Panagiotis E. Souganidis\address{(P.E. Souganidis) The University of Chicago, Eckhart Hall, 5734 S University Ave, Chicago, IL 60637, USA}
\email{souganidis@uchicago.edu}}

\maketitle

\begin{abstract}
   This paper is concerned with the Fleming-Viot particle system introduced in \cite{burdzy2020}. In this model, $N$ Brownian particles evolve independently in a bounded domain $D$ until one of the particles reaches the boundary. Then, the particle which has hit the boundary instantaneously jumps to the location of one of the other particles, chosen uniformly at random. In \cite{burdzy2020}, Burdzy, Ho{\l}yst and March showed that when $N$ tends to infinity, the empirical measure of the system converges to a solution to the heat equation on $D$ with Dirichlet boundary conditions, renormalized to have total mass $1$. Our main result is a sharp, uniform-in-time, quantitative version of this statement. We employ the method of weak propagation of chaos, which necessitates a careful study of the (backward) Kolmogorov equations associated to the $N$-particle systems, and the infinite-dimensional transport equation whose characteristics are given by renormalized solutions of the Dirichlet heat equation. The proofs are entirely analytical, and most of the technical effort is devoted to building barrier functions which are used to control the singular behavior of the system when most of the particles approach the boundary. 
\end{abstract}

 \setcounter{tocdepth}{1}

\tableofcontents

\section{Introduction}

In this paper we study the Fleming-Viot particle system introduced by Burdzy, Ho\l yst and March in \cite{burdzy2020}. In this model, $N$ Brownian particles evolve independently in a bounded domain $D$ until one of the particles reaches the boundary. Then, the particle which has hit the boundary instantaneously jumps to the location of one of the other particles, chosen uniformly at random. In \cite{burdzy2020}, it was shown that when $N$ tends to infinity, the empirical measure of the system converges to a solution to the heat equation on $D$ with Dirichlet boundary conditions, renormalized to have total mass $1$. 

Our main result is a quantitative, uniform-in-time, weak error estimate for this model. We employ the method of weak propagation of chaos, which necessitates a careful study of the (backward) Kolmogorov equations associated to the $N$-particle systems, and the infinite-dimensional transport equation whose characteristics are given by renormalized solutions of the Dirichlet heat equation. The proofs are entirely analytical, and most of the technical effort is devoted to building barrier functions which are used to control the singular behavior of the system when most of the particle approach the boundary. 

\subsection{The Fleming-Viot particle system}
Let $D \subset \R^d$ be a smooth domain.
In \cite{burdzy2020}, Burdzy, Ho\l yst and March considered a $\ov{D}^N$-valued process $\bX^N_t = (X^{N,1}_t,...,X^{N,N}_t)$, which evolves as follows:
\begin{enumerate}
    \item The system is initialized at $\bX_0^N = \bx^N \in D^N$.
    \item While $\bX_t^N \in D^N$, the particles do not interact, and are driven by independent Brownian motions, i.e. $d X_t^{i} = dB_t^i$, $i = 1,...,N$, where $(B^i)_{i =1,...,N}$ are independent Brownian motions. 
    \item When one of the particles, say $X^{i}_t$, hits the boundary, it immediately jumps to the location of one of the $N-1$ other particles, chosen uniformly at random, and independently of everything else.
\end{enumerate}
The main results of \cite{burdzy2020} show that this particle system is well-defined for all $t \geq 0$, and that its mean field limit is given by $\ov p_t = \frac{p_t}{p_t(D)}$, where $p = (p_t)_{t \geq 0} \in C\big(\R_+ ; \cP(D)\big)$ solves
\begin{align}
     \label{peqn.intro}
    \partial_t p - \frac{1}{2} \Delta p = 0, \quad p\big|_{\R_+ \times \partial D} = 0, \quad p_0 = m.
\end{align}
More precisely, \cite{burdzy2020} shows that the empirical measure
\begin{align}
    m_t^N = m_{\bX_t^N}^N = \frac{1}{N} \sum_{i = 1}^N \delta_{X_t^{N,i}},
\end{align}
converges towards $\ov p_t$, provided that the initial conditions converge:
\begin{align} \label{burdzymeanfield}
    m_{\bx^N}^N \xrightarrow{N \to \infty} m  \implies m_t^N \xrightarrow{N \to \infty} \ov{p}_t \text{ weakly in $\cP(D)$, for all $t \geq 0$.}
\end{align}
The significance of this result is due in part to the fact that, by basic spectral theory, 
\begin{align*}
    \ov p_t \xrightarrow{t \to \infty} \rho \, dx, 
\end{align*}
where $\rho$ is a version of the principal Dirichlet eigenfunction of the operator $- \frac{1}{2} \Delta$ on $D$, normalized so that $\rho(x) > 0$ for $x \in D$, and $\int_D \rho dx = 1$. Thus, \eqref{burdzymeanfield} shows that if we choose $t$ and then $N$ large enough, $\bX_t^N$ provides a ``particle approximation" of the principal eigenfunction $\rho$. Probabilistically, the convergence of $\ov p_t$ to $\rho$ is related to the fact that $\rho$ is the unique quasi-stationary distribution (QSD) for the Brownian motion killed upon exiting $D$; we refer to the survey \cite{meleardvillemonais} or the book \cite{collet} for background on the general theory of killed Markov processes and their QSDs.

In addition, \cite{burdzy2020} shows that for each fixed $N$, the particle system has an invariant measure $P^N \in \cP( D^N)$ such that $\cL(\bX_t) \xrightarrow{t \to \infty} P^N$, and that the large time and large $N$ limits commute; that is, if $\bX_{\infty} \sim P^N$, then
\begin{align} \label{invariant.conv}
 m_{\bX_{\infty}}^N \xrightarrow{N \to \infty} \rho \, dx.
\end{align}
We refer the reader to \cite{grigorescukang} and \cite{lobus} for refinements of the results in \cite{burdzy2020}, and \cite{toughnolen} for a more recent extension to McKean-Vlasov dynamics.

\subsection{Our perspective and main results}
Our main result is a quantitative, uniform-in-time version of \eqref{burdzymeanfield}, which in particular implies a quantitative version of \eqref{invariant.conv}.
Rather than studying directly the asymptotics of the empirical measure $m_t^N$ of the particle system, we are going to fix a nice "test function" $G : \cP(D) \to \R$, and study the function $V^N : \R_+ \times D^N \to \R$ given by
\begin{align} \label{def.VN}
     V^N(t,\bx) = \E\Big[ G\big( m_{\bX_t^N}^N \big) \, |\,  \bX_{0} = \bx \Big], 
\end{align}
where $\bX$ evolves according to the $N$-particle Fleming-Viot dynamics discussed above.
In other words, we are studying the backward Kolmogorov equation associated to the Markov process $\bX$. 

Formally, $V^N$ should satisfy the heat equation in $\R_+ \times D^N$, with "Fleming-Viot boundary conditions" on the lateral boundary. More precisely, we expect $V^N$ to be characterized by the equation
\begin{align} \label{eqn.VN}
    \begin{cases}
        \ds \partial_t V^N(t,\bx) - \frac{1}{2} \sum_{i = 1}^N \Delta_{x^i} V^N(t,\bx) = 0, \quad (t,\bx) \in \R_+ \times D^N,
       \vspace{.1cm} \\
        \ds \quad  V^N(t,\bx) = \frac{1}{N-1} \sum_{j \neq i} V^N\big(t,\bx^{i,j}\big), \quad t \in \R_+, \,\, x^i \in \partial D, \,\, x^j \in  D \text{ for $j \neq i$},
       \vspace{.1cm} \\
        \ds V^N(0,\bx) = G^N(\bx), \,\, \bx \in D^N, 
    \end{cases}
\end{align}
where $G^N(\bx) = G\big(m_{\bx}^N\big)$.
Here we use the notation
\begin{align*}
    \bx^{i,j} = (x^1,...,x^{i-1},x^j,x^{i+1},...,x^N), \text{  for  } \bx \in \ov D^N.
\end{align*}
In other words, $\bx^{i,j}$ is obtained from $\bx$ by replacing the value of the $i^{\text{th}}$ coordinate with the value of the $j^{\text{th}}$ coordinate. 
We explain in Section \ref{sec.vn} how to precisely interpret this equation, and we establish in Theorem \ref{thm.vneqn} its well-posedness. 

We aim to show that for $N$ large and $m \in \cP(D)$,
\begin{align*}
    \int_{D^N} V^N(t,\bx) m^{\otimes N}(d\bx) \approx U(t,m), 
\end{align*}
where $U : \R_+ \times \cP(D) \to \R$ is defined by 
\begin{align} \label{def.U.intro}
     U(t,m) = G\big( \ov p_t^m\big), \,\, \text{ where } \,\, \ov p_t^m = p_t^m / p_t^m(D)
\end{align}
and $p^m$ denotes the solution to \eqref{peqn.intro}. We will show in Proposition \ref{prop.Ueqn} that $U$ satisfies the infinite-dimensional transport equation
     \begin{align} \label{eqn.U}
      \begin{cases} \ds    \partial_t U(t,m) - \frac{1}{2} \int_{D} \Delta_x \frac{\delta U}{\delta m}(t,m,x) m(dx) = 0, \quad (t,m) \in \R_+ \times \cP^c(D), 
     \vspace{.2cm}   \\
        \ds \quad   \frac{\delta U}{\delta m}(t,m,x) = 0, \quad (t,m) \in \R_+ \times \cP^+(\ov D), \quad x \in \partial D, 
     \vspace{.2cm}     \\
       \ds    U(0,m) = G(m), \quad m \in \cP(D).
        \end{cases}
    \end{align}
Here $\frac{\delta U}{\delta m}$ denotes the "linear derivative", and the spaces $\cP(D$)$, \cP^c(D)$, and $\cP^+(\ov D)$ of probability measures will be introduced in Section \ref{sec.prelim}.


Our main result is as follows. The definition of $\cC^{2+\alpha}\big(\cP(\ov D)\big)$ will be given in the next section. Given two sets $A, B \subset \R^d$, $A \Subset B$ means that $\ov{A} \subset B$. 
\begin{thm} \label{thm.conv.finitetime}
    Suppose that $G \in \cC^{2+\alpha}\big( \cP(\ov D)\big)$ for some $\alpha \in (0,1)$, and let $G^N(\bx) = G(m_{\bx}^N)$. Let $V^N$ be the unique classical solution of \eqref{eqn.VN} in the sense of Definition \ref{def.VN.classicalsoln}. Then, for any $m \in \cP(D)$ with $\text{supp}(m) \Subset D$, we have the estimate 
    \begin{align*}
      \sup_{t \in \R_+}  \Big| U(t,m) - \int_{D^N} V^N(t,\bx) m^{\otimes N}(d\bx) \Big| \lesssim 1/N, 
    \end{align*}
    where $U$ is defined by \eqref{def.U.intro}, and the implied constant depends only on $D$, $\| G\|_{\cC^{2}}$, and $\text{supp}(m)$. 
\end{thm}

\begin{rmk}
    Throughout the paper, we work at the level of the equations \eqref{eqn.VN} and \eqref{eqn.U}, and in particular our proof of Theorem \ref{thm.conv.finitetime} makes no use of the particle system. This is why we prefer to state the result purely in terms of $V^N$ and $U$. However, by expanding $V^N$ along the dynamic of $\bX$ (in particular, this can be done via a straightforward extension of Proposition 1 in \cite{grigorescukang}), we find that \eqref{def.VN} holds,
    and hence
    \begin{align*}
        \int_{D^N} V^N(t,\bx) m^{\otimes N}(d\bx) = \int_{D^N} \E\Big[ G\big(m_{\bX_t}^N\big) | \bX_0 = \bx \Big] m^{\otimes N}(d\bx) =   \E\Big[ G\big( m_{\bX_t}^N\big) | \bX_0 \sim m^{\otimes N} \Big].
    \end{align*}
    Moreover, by definition $U(t,m) = G\big(\ov p_t^m\big)$. Thus, the main estimate can equivalently be written 
    \begin{align*}
        \sup_{t \in \R_+} \Big| G\big(\ov p_t^m \big) -  \E\Big[ G\big(m_{\bX_t}^N\big) | \bX_0 \sim m^{\otimes N} \Big] \Big| \lesssim 1/N.
    \end{align*}
    This explains why this type of result is sometimes referred to as "weak propagation of chaos", see e.g. \cite{ChassegneuxSzpruchTse, DelarueTse}. 
\end{rmk}

\begin{rmk}
    The rate $1/N$ in Theorem \ref{thm.conv.finitetime} is likely optimal, since it is optimal in the case of weakly interacting particle systems without the Fleming-Viot resampling mechanism, as studied in \cite{ChassegneuxSzpruchTse, DelarueTse}. However, at present we do not have an example confirming optimality in the Fleming-Viot case. The condition that $\supp(m) \Subset D$ is likely overkill, but it seems necessary to impose some condition on $m$ which avoids too much mass being near the boundary, since, as will be discussed below, the derivatives of $U$ blow up when $m$ concentrates near the boundary. 
\end{rmk}

\begin{rmk}
    It is natural to wonder whether Theorem \ref{thm.conv.finitetime} could be generalized by replacing the Brownian motion with a different diffusion process, or even the solution of a McKean-Vlasov SDE as in \cite{toughnolen}. We think that on a finite time horizon, this generalization should be possible. However, the uniform-in-time estimate is more delicate, and likely needs to be handled on a case-by-case basis.
\end{rmk}

Since $m \mapsto \|m - n\|_{H^{-s}}^2$ is $\cC^{2+\alpha}$ for each fixed $n \in \cP(D)$ and $s > d/2 + 2$, we get the following corollary at the level of the particle system.

\begin{cor} \label{cor.particlesystem.hminuss}
    Suppose that $m$ satisfies $\text{supp}(m) \Subset D$, and that $\bX$ denotes the $N$-dimensional Fleming-Viot particle system with initial condition $\bX_0 \sim m^{\otimes N}$. Then, for any $s > d/2 + 2$,
    \begin{align} \label{hminuss.est}
       \sup_{t \in \R_+}  \E \Big[ \norm{ m_t^N - \ov{p}_t^{m} }_{H^{-s}}^2 \Big] \lesssim 1/N,   
    \end{align}
    where $m_t^N = m_{\bX_t}^N = \frac{1}{N}\sum_{i = 1}^N \delta_{X_t^{i}}$, and the implied constants depend on $D$, $s$, and $\text{supp}(m)$. In addition, for any $\theta >0$, we have
     \begin{align} \label{d1.est}
        \sup_{t \in \R_+} \E \Big[\bd_1\big( m_t^N, \ov{p}_t^{m} \big)\Big] \lesssim N^{-1/(d + 4 + \theta)}, 
    \end{align}
    where the implied constant can depend on $\theta$, $D$, and $\text{supp}(m)$. 
\end{cor}


Finally, by sending $t \to \infty$ in the above estimates, we obtain the following.

\begin{cor} \label{cor.invariantdist}
   Denote by $P^N \in \cP(D^N)$ the invariant measure for the $N$-particle Fleming-Viot particle system. Then, for $G \in \cC^{2+\alpha}\big(\cP(\ov D)\big)$, we have the estimate
   \begin{align} \label{invariantest1}
      \Big| \int_{D^N} G\big(m_{\bx}^N\big) P^N(d\bx) - G\big( \rho \, dx \big) \Big| \lesssim 1/N, 
   \end{align}
   where the implied constant depends only on $\| G\|_{\cC^{2}}$ and $D$. 
   
   For $s > d/2 + 2$, we have the estimate
    \begin{align} \label{invariantest2}
        \E_{P^N}\Big[ \norm{m_{\bX}^N - \rho \, dx}_{H^{-s}}^2 \Big] = \int_{D^N} \norm{m_{\bx}^N - \rho}_{H^{-s}}^2 dP^N(\bx) \lesssim 1/N, 
    \end{align}
    where the implied constant can depend on $D$ and $s$. 
    
   Finally, for $\theta > 0$, we have the estimate
    \begin{align} \label{invariantest3}
        \E_{P^N}\Big[ \bd_1\big(m_{\bX}^N, \rho \, dx\big) \Big] = \int_{D^N} \bd_1\big(m_{\bx}^N, \rho \, dx\big) dP^N(\bx) \lesssim N^{- 1/(d+4 + \theta)},
    \end{align}
    where the implied constant can depend on $D$, and $\theta$. 
\end{cor}

\subsection{Related literature}


Finite-time quantitative convergence results for Fleming-Viot particle systems have been obtained in significant generality. For example, when applied in our setting, Theorem 2.2 of \cite{villemonais} yields an estimate of the form
\begin{align} \label{est.vill}
   \E\bigg| \int_D \, f \, dm_t^N - \int_D \, f \, d \ov p_t^{m_0^N} \bigg| \lesssim_{t,m_0^N} \frac{\| f\|_{\infty}}{\sqrt{N}}, 
\end{align}
where $f : D \to \R$ is an arbitrary bounded test function, and the implied constant depends in an exponential way on $t$. This estimate was obtained for a much more general class of models, in which the Brownian motions are replaced by independent Markov processes satisfying appropriate conditions. Meanwhile, Theorem 2 of \cite{DelMoralVillemonais} established a uniform-in-time analogue of \eqref{est.vill}, without a stated convergence rate. More recently, \cite{JournelMonmarche} obtains a uniform-in-time result in a related setting, where the Brownian motion is replaced by a Langevin diffusion with a small noise with dynamics of the form 
\begin{align*}
    dX_t = - \nabla U(X_t) dt + \sqrt{2\eps} dB_t. 
\end{align*}
In this work, the function $U$ and the domain $D$ are supposed to satisfy certain technical conditions, and the parameter $\eps$ is required to be sufficiently small; this guarantees that the domain $D$ is "metastable" for the diffusion. Under these conditions, \cite[Theorem 2]{JournelMonmarche} establishes a uniform-in-time analogue of \eqref{est.vill}, but with a non-explicit rate $N^{-\eta_{\eps}}$, for some $\eta_{\eps} > 0$.

Finally, Theorem 4 in the very recent \cite{cox2025linear} establishes, in the setting of the generalized Fleming-Viot-type model introduced in \cite{CoxHortonVillemonais1}, a uniform-in-time $L^2$ version of \eqref{est.vill}. In this result, the underlying diffusion is a Brownian motion with a bounded drift. When applied in our setting, it yields an estimate of the form 
\begin{align} \label{est.vill2}
    \Big\| \int_D \, f \, dm_t^N - \int_D \, f \, d \ov p_t^{m_0^N} \Big\|_{L^2} \lesssim_{m_0^N} \frac{\| f\|_{\infty}}{\sqrt{N}}.
\end{align}
In particular, this implies a version of our Theorem \ref{thm.conv.finitetime}, in the case that $G$ is linear, i.e. $G(m) = \int_D f \, dm$ for some bounded function $f$, and with the rate $\frac{1}{\sqrt{N}}$. Compared to \eqref{est.vill2}, the novelty of our Theorem \ref{thm.conv.finitetime} is that our analytical approach allows us to establish uniform-in-time weak error bounds for non-linear $G$, and to obtain the faster rate $1/N$ when $G$ is smooth enough. These weak error rates in turn imply estimates in genuine metrics, as explained in Corollary \ref{cor.particlesystem.hminuss}.

We note that in addition to the results discussed above concerning "hard killing" (meaning that the particles evolve in a bounded domain, and resampling occurs when one of them reaches the boundary), there are much more precise estimates available in the "soft killing" case. Soft killing means that particles are instead resampled at random times, with rates which depend on the positions of the particles. We refer to e.g. \cite{MoralMiclo, Rousset, JournelRousset2024} for some quantitative results in this setting, and to the recent \cite{JournalMonmarcheSoft} for a probabilistic approach, which establishes sharp, uniform in time convergence results in the Wasserstein distance. We refer also to \cite{CloezNoemie, CloezCorujo, FerrariMaric} for some related results in discrete state spaces, \cite{toughnolen} for a McKean-Vlasov version of the model studied here, and \cite{CDGR2020} for a central limit theorem for a general class of Fleming-Viot particle systems.

\subsection{Proof strategy}

We start by following the method which is known as ``weak propagation of chaos'', and has been developed in the setting of weakly interacting diffusions (without any Fleming-Viot boundary conditions) in \cite{ChassegneuxSzpruchTse, DelarueTse}. We refer also to \cite{BayraktarEkrenZhou, CaoRenTan} for further applications of this idea. The general strategy of projecting the solution of an infinite-dimensional PDE down to finite dimensions in order to obtain mean field convergence has also been used frequently in the theory of mean field games and mean field control, ever since the seminal work \cite{CDLL_2019}. A similar strategy is also developed in \cite{mischlermouhot} in the setting of kinetic theory.

In this work, we begin by projecting $U$ down to finite dimensions, defining 
\begin{align*}
    U^N : \R_+ \times \big( D^N \setminus \partial^N(D^N)\big) \to \R, \quad U^N(t,\bx) = U(t,m_{\bx}^N).
\end{align*}
In order to estimate $V^N - U^N$ (which implies an estimate on $\int_{D^N} V^N d m^{\otimes N} - U$ by standard arguments), we ask the key question:
\begin{align} \label{question}
    \text{By how much does $U^N$ fail to solve the equation \eqref{eqn.VN}?}
\end{align}
On the one hand, using the equation for $U$ and following a standard computation, we get
\begin{align} \label{un.error1}
    \partial_t U^N(t,\bx) - \frac{1}{2} \sum_{i = 1}^N \Delta_{x^i} U^N(t,\bx) =  - \frac{1}{2N^2} \sum_{i = 1}^N  \text{tr}\big(D_{mm} U(t,m_{\bx}^N,x^i,x^i) \big), \quad (t,\bx) \in \R_+ \times D^N.
\end{align}
On the other hand, by using the boundary condition $\frac{\delta U}{\delta m}(t,m,x) = 0$ for $x \in \partial D$, we can show (see the proof of Proposition \ref{prop.UN.properties}) that 
\begin{align} \label{un.boundaryerror1}
    \Big| U^N(t,\bx) - \frac{1}{N-1} \sum_{j \neq i} U^N(t,\bx^{i,j}) \Big| \lesssim \norm{ \frac{\delta^2 U}{\delta m^2} (t,m_{\bx}^N,\cdot,\cdot)}_{\infty} N^{-2}, \quad t \in \R_+, \,\, x^i \in \partial D, \,\, x^j \in D \text{ for } j \neq i.
\end{align}
So, in order to understand the answer to the question \eqref{question}, we need to estimate both $D_{mm} U$ and $\frac{\delta^2 U}{\delta m^2}$. 

In our setting, because the map $(t,m) \mapsto p_t^m$ becomes less regular when more of the mass of $m$ concentrates near the boundary, we cannot expect global estimates on the derivatives of $U$. Nevertheless, in Subsection \ref{subsec.Uderivest}, we obtain precise estimates which explain how the derivatives $D_{mm} U$ and $\frac{\delta^2 U}{\delta m^2}$ can blow up. When combined with \eqref{un.error1} and \eqref{un.boundaryerror1}, this yields estimates of the form 
\begin{align} \label{un.error2}
    \Big| \partial_t &U^N(t,\bx) - \frac{1}{2} \sum_{i = 1}^N \Delta_{x^i} U^N(t,\bx) \Big|
   \nonumber  \\
    &\lesssim \exp( - c t) \frac{1}{N} \Big(\frac{1}{N} \sum_{i = 1}^N \dist(x^i, \partial D) \Big)^{-2} \Big(1 + \frac{1}{N} \sum_{i = 1}^N w^2(t,x^i) 1_{t \leq 1} \Big), \quad (t,\bx) \in \R_+ \times D^N,
\end{align}
and
\begin{align} \label{un.boundaryerror2}
    \Big| U^N(t,\bx) - \frac{1}{N-1} \sum_{j \neq i} U^N(t,\bx^{i,j}) \Big| \lesssim  \exp(-ct) \frac{1}{N^2} \Big(\frac{1}{N} \sum_{i = 1}^N \dist(x^i, \partial D) \Big)^{-2} , \quad  x^i \in \partial D, \,\, x^j \in D \text{ for } j \neq i, 
\end{align}
where $c = \lambda_2 - \lambda_1 > 0$, with $0 < \lambda_1 < \lambda_2$ the first two Dirichlet eigenvalues of $D$, and $w(t,x) : (0,1) \times D \to \R_+$ is a tailor-made barrier function which blows up when $t \to 0$ and $x \to \partial D$; see \eqref{def.w} for the definition, which makes use of the transition kernel for a certain 1-dimensional diffusion with drift. 

We note that the left-hand side of \eqref{un.boundaryerror2} is a priori of size $1/N$ as soon as $U^N$ is e.g. Lipschitz with respect to $\bd_1$, so there is a gain of $1/N$. Thus, we can view the estimates \eqref{un.error2} and \eqref{un.boundaryerror2} as saying that $U^N$ solves \eqref{eqn.VN} up to an error of order $1/N$, at least \textit{locally}. But the estimate blows up when either (i) most of the mass is near the boundary, or (ii) $t$ is close to $0$ and \textit{any} of the mass is near the boundary (this is quantified by the function $w$).

This leads us to the second main novelty of our proofs. To the authors' knowledge, all previous works involving weak propagation of chaos use a "trajectorial approach". However, because our estimates are not uniform, executing this strategy would require a delicate analysis of the particle system, controlling in a quantitative way the amount of time the particle system spends in places where $D_{mm} U$ and $\frac{\delta^2 U}{\delta m^2}$ are large, the number of times that the boundary is reached, etc.

We completely avoid this trajectorial approach; instead, our argument is based on the maximum principle. The initial idea is to try to build a function $\Phi^N$ which is bounded (in some appropriate sense) uniformly in $t$ and $N$, and such that $U^N + \frac{1}{N} \Phi^N$ is a super-solution of the equation \eqref{eqn.VN} satisfied by $V^N$, while $U^N - \frac{1}{N} \Phi^N$ is a subsolution. If we can build such functions, then the comparison principle in Proposition \ref{prop.comparison} gives $$ U^N - \frac{1}{N} \Phi^N \leq V^N \leq U^N + \frac{1}{N} \Phi^N.$$ 
In light of the estimates \eqref{un.error2}, \eqref{un.boundaryerror2}, this means that $\Phi^N$ should satisfy
\begin{align} \label{un.error2.comp}
    \partial_t &\Phi^N(t,\bx) - \frac{1}{2} \sum_{i = 1}^N \Delta_{x^i} \Phi^N(t,\bx) 
   \gtrsim \exp( - c t) \frac{1}{N} \Big(\frac{1}{N} \sum_{i = 1}^N \dist(x^i, \partial D) \Big)^{-2} \Big(1 + \frac{1}{N} \sum_{i = 1}^N w^2(t,x^i) 1_{t \leq 1} \Big),
\end{align}
and
\begin{align} \label{un.boundaryerror2.comp}
    \Phi^N(t,\bx) - \frac{1}{N-1} \sum_{j \neq i} \Phi^N(t,\bx^{i,j})    \gtrsim\exp(-ct) \frac{1}{N^2} \Big(\frac{1}{N} \sum_{i = 1}^N \dist(x^i, \partial D) \Big)^{-2} , \,\, x^i \in \partial D, \,\, x^j \in D \text{ for } j \neq i,
\end{align}
for some large $C$. 

We do not succeed in building functions $\Phi^N$ which are uniformly bounded and satisfy \eqref{un.error2.comp} and \eqref{un.boundaryerror2.comp}. However, we do succeed in building functions $\Phi^N$ which satisfy similar bounds on the set where $\Phi^N$ is bounded by $CN$ for large $C$, which, in light of the boundedness of $U$ and $V^N$, is sufficient.

The main challenge of the paper is thus to build super-solutions to \eqref{eqn.VN} with the correct blow-up profiles. This turns out to be quite challenging. We in fact build several functions, which play different roles. In a first step, in Subsection \ref{subsec.unittime} we obtain convergence on a unit time horizon by working with a function of the form
\begin{align} \label{phi.intro}
    \Phi^N = \Phi_1^N + \Phi_2^N, \quad \Phi_1^N(t,\bx) = \exp\big(9 t\big) \Big(\frac{1}{N} \sum_{i = 1}^N \rho(x^i) \Big)^{-7}, \quad \Phi_2^N(t,\bx) = \frac{t^{\beta}}{N} \sum_{i = 1}^N w(t,x^i), 
\end{align}
with $w$ the same barrier function appearing in \eqref{un.boundaryerror2}, and for an appropriate value of $\beta \in (0,1)$. Then, in a second step, we obtain convergence for all $t \geq 1$, by working instead with the function $ \Psi^N = \Psi_1^N + \Psi_2^N$, with
\begin{align} \label{psi.intro}
  \nonumber   &\Psi_1^N(t,\bx) = \exp\big(- ct \big) \Big(\frac{1}{N} \sum_{i = 1}^N \rho(x^i) \Big)^{-7},
    \\
    &\Psi_2^N(t,\bx) = \frac{1}{N(N-1)...(N-9)} \sum_{\substack{i_1,...,i_{10} = 1,...,N \\ i_1,...,i_{10} \text{ distinct}}} \Psi\big(t,x^{i_1},..x^{i_{10}}\big) + C \int_0^t \exp\big(- c'  s \big)ds, 
\end{align}
and $\Psi : \R_+ \times D^N \to \R$ a barrier function which satisfies certain useful properties, which we build in Proposition \ref{prop.supersolexists}.

\subsection{Outline of the paper}

In Section \ref{sec.prelim}, we fix notation and discuss some preliminaries. In Section \ref{sec.vn}, we establish the well-posedness of the finite-dimensional PDE \eqref{eqn.VN}. In Section \ref{sec.Uprops}, we obtain various estimates on the function $U$, and show that it solves \eqref{eqn.U} in a classical sense. Finally, the appendices collect various facts about the (renormalized) heat flow on $D$, and build the key barrier function $\Psi$ appearing in \eqref{psi.intro}. 

\subsection{Acknowledgments} P. Cardaliaguet was partially supported by the Agence Nationale de la Recherche (ANR), project ANR-22-CE40-0010 COSS. The work was completed during the period he was hosted by the INRIA project team Martingale. M. Cirant is member of the Gruppo Nazionale per l’Analisi Matematica, la
Probabilità e le loro Applicazioni (GNAMPA) of the Istituto Nazionale di Alta Matematica (INdAM).
 J. Jackson was supported by the NSF under Grant No. DMS-2302703. P. E. Souganidis was partially supported by the National Science Foundation grants DMS-2153822 and DMS-2452972.

\section{Preliminaries and notation}  \label{sec.prelim}

We discuss preliminaries and notation related to the domain $D$, spaces of functions, and calculus on spaces of measures on $D$.

\subsection{The domain $D$}

Recall that we work on a bounded domain $D \subset \R^d$ with $C^{\infty}$ boundary $\partial D$. We denote by $\dist(x,\partial D)$ the distance of a point $x \in D$ to the boundary $\partial D$, and for $\eps > 0$, we use the notation 
\begin{align*}
    (\partial D)_{\eps} = \big\{x \in D : \dist(x,\partial D) < \eps \big\} \subset D.
\end{align*}
We denote by $\dist : \ov D \to \R_+$ a $C^{\infty}$ function with the property $\dist(x) > 0$ for $x \in D$, and for some $\eps > 0$,  
\begin{align*}
    \dist(x) = \dist(x,\partial D) \text{ for } x \in (\partial D)_{\eps}.
\end{align*}
The spectrum of the operator $- \frac{1}{2} \Delta$ on $D$ (with homogeneous Dirichlet boundary conditions) is denoted by
\begin{align*}
    0 < \lambda_1 < \lambda_2 \leq \lambda_3 \leq ... \to \infty. 
\end{align*}
We denote by $\rho_1,\rho_2,...$ a corresponding orthonormal basis of $L^2(D)$, i.e. for $n,m \in \N$, we have
\begin{align*}
    \rho_n \in C^{\infty}(\ov D), \quad \rho_n|_{\partial D} = 0, \quad - \frac{1}{2} \Delta \rho_n = \lambda_n \rho_n \text{ in } D, \quad \langle \rho_n, \rho_m \rangle_{L^2(D)} = 1_{n = m}.
\end{align*}
It is known that the first eigenfunction $\rho_1$ is of constant sign in $D$, and so we assume without loss of generality that $\rho_1(x) > 0$ for $x \in D$. We then denote by 
\begin{align*}
    \rho = \frac{\rho_1}{\int_D \rho_1 dx}
\end{align*}
a renormalized version, which we will often view as the density of a probability measure on $D$. We note that by the Hopf lemma, $\dist$ and $\rho$ are comparable, i.e. there is a constant $C > 0$ such that 
\begin{align*}
    \frac{1}{C} \dist(x) \leq \rho(x) \leq C \dist(x), \quad x \in \ov D. 
\end{align*}

\subsection{Spaces of functions and measures on $D$}
We denote by $L^p(D)$ the standard space of $p$-integrable functions on $D$, endowed with the norm $\|\cdot\|_{L^p}$. Given $s > 0$, we denote by $H^s(D)$ the Sobolev space of functions on $D$ with $s$ square-integrable derivatives. More precisely, for $k \in \N$, 
\begin{align*}
    H^k(D) = \big\{f \in L^2(D) : \| f \|_{H^k} < \infty \big\}, \, \text{ where } \norm{f}_{H^k}^2 \coloneqq  \sum_{0 \leq |\alpha| \leq k} \int_D |D^{\alpha} f|^2 dx, 
\end{align*}
and with the sum in the definition of $\|f\|_{H^k}$ taken over the set of multi-indices $\alpha$ of order at most $k$. Then, for $s = k + \theta$, $k \in \N \cup \{0\}$ and $\theta \in (0,1)$, we have 
\begin{align*}
    H^s(D) = \big\{ f\in L^2(D) : \| f \|_{H^s} < \infty \big\}, \quad \norm{f}_{H^s}^2 \coloneqq \|f\|_{H^k}^2 + \sum_{|\alpha| = k} \int_D \int_D \frac{|D^{\alpha} f(x) - D^{\alpha} f(y)|^2}{|x-y|^{d + 2\theta}} dx dy. 
\end{align*}
We denote by ${H^{-s}} = {H^{-s}}(D)$ the topological dual of $H^s$, and denote by $\langle p, q \rangle_{H^{-s}}$ and $\|p\|_{H^{-s}}$ the inner product and norm on $H^{-s}$ inherited from duality with $H^s$. 

Given a metric space $X$, we denote by $\cP(X)$ the space of Borel probability measures on $X$, and $\sub(X)$ the space of non-negative Borel measures of total mass at most $1$. For technical reasons, we will work with various spaces of measures, but all of these spaces can be viewed as subsets of $\sub(\ov D)$, the space of non-negative Borel measures $m$ on $\ov D$ with $0 \leq m( \ov D) \leq 1$. For $m,n \in \sub(\ov D)$, we work with the distance $\bd$ inherited from duality with the space $W^{1,\infty}(\ov D)$ of bounded and Lipschitz functions on $D$: 
\begin{align*}
    \bd(m,n) = \|m - n\|_{(W^{1,\infty})^*} = \sup_{\phi : \ov D \to \R, \, \|\phi\|_{W^{1,\infty}} \leq 1} \int_{\ov D} \phi \, d(m - n), \quad \|\phi\|_{W^{1,\infty}} = \|\phi\|_{\infty} + \|D \phi\|_{\infty}.
\end{align*}
We denote by $\cP^+(\ov D)$ the set of all $m \in \cP(\ov D)$ such that $m(D) > 0$, and by $\cP^c(D)$ the set of all $m \in \cP(D)$ such that $m(K) = 1$ for some $K \Subset D$.

We note that for probability measures (i.e. when we work with $\cP(\ov D)$, $\cP^+(\ov D)$, $\cP^c(D)$,...), the metric $\bd$ is equivalent to the usual $1$-Wasserstein distance $\bd_1$, given by
\begin{align*}
    \bd_1(m,n) = \sup_{\phi \in E} \int \phi \, d(m - n), \quad E \coloneqq \big\{ \phi : \ov D \to \R \, \big| \, \phi \text{ is 1-Lipschitz}\big\}.
\end{align*}

\subsection{Calculus on spaces of measures}

We are going to be dealing specifically with derivatives of functions $\Phi : \cP^+(\ov D) \to \R$ or $\Phi : \cP(\ov D) \to \R$. We say that $\Phi : \cP^+(\ov D) \to \R$ has a continuous linear derivative if there exists a continuous function 
\begin{align*}
    \frac{\delta \Phi}{\delta m} : \cP^+(\ov D) \times \ov D \to \R
\end{align*}
which satisfies 
\begin{align} \label{def.linearderiv}
    \Phi(m') - \Phi(m) = \int_0^1 \int_{\ov D} \frac{\delta \Phi}{\delta m}\big( [m,m']_s, x \big)\, (m' - m)(dx)ds, \quad \forall \,\, m,m' \in \cP^+(\ov D), 
\end{align}
and 
\begin{align} \label{linderiv.normalization}
    \int_{\ov D} \frac{\delta \Phi}{\delta m}(m,x) m(dx) = 0, \quad \forall \, m \in \cP^+(\ov D).
\end{align}
Here and throughout the paper, we use the notation 
\begin{align}
    [m,m']_s = (1-s) m + sm', \quad 0 \leq s \leq 1.
\end{align}
Higher derivatives are defined similarly. In particular, we say that $\Phi$ has a continuous second linear derivative if $m \mapsto \frac{\delta \Phi}{\delta m}(m,x)$ has a linear derivative for each $x$, and the map 
\begin{align*}
    \frac{\delta^2 \Phi}{\delta m^2}  : \cP^+(\ov D) \times \ov D \times \ov D \to \R, \quad \frac{\delta^2 \Phi}{\delta m^2} (m,x,x') = \frac{\delta }{\delta m} \Big[ \frac{\delta \Phi}{\delta m}(\cdot,x) \Big](m,x')
\end{align*}
is jointly continuous.
If $\frac{\delta \Phi}{\delta m}$ exists and $\frac{\delta \Phi}{\delta m}(m,\cdot) \in C^1(\ov D)$ for each fixed $m \in \cP^+(\ov D)$, then we define 
\begin{align*}
    D_m \Phi : \cP^+(\ov D) \times \ov D \to \R^d, \quad D_m \Phi(m,x) = D_x \frac{\delta \Phi}{\delta m}(m,x).
\end{align*}
We say that $\Phi$ is $\cC^1$, written $\Phi \in \cC^1\big(\cP^+(\ov D)\big)$, if $D_m \Phi$ exists and is jointly continuous. Similarly, we say that $\Phi$ is $\cC^{2}$, written $\Phi \in \cC^{2}\big( \cP^+(\ov D)\big)$, if $\Phi$ admits two continuous linear functional derivatives $\frac{\delta \Phi}{\delta m}$ and $\frac{\delta^2 \Phi}{\delta m^2}$, we have $\frac{\delta \Phi}{\delta m}(m,\cdot) \in C^2(\ov D)$, $\frac{\delta^2 \Phi}{\delta m^2}(m,\cdot,\cdot) \in C^2(\ov D \times \ov D)$, and the maps
\begin{align*}
    &D_m \Phi : \cP^+(\ov D) \times D \to R^d, \quad D_x D_m \Phi = D_{xx} \frac{\delta \Phi}{\delta m} : \cP^+(\ov D) \times D \to \R^{d \times d}, \\
    &D_{mm} \Phi :  \cP^+(\ov D) \times \ov D \times \ov D \to \R, \quad  D_{mm} \Phi(m,x,x') \coloneqq D_m \big[ D_m \Phi(\cdot,x) \big](m,x') = D_{x'} D_x \frac{\delta^2 \Phi}{\delta m^2}(m,x,x')
\end{align*}
are jointly continuous. We note that if $\Phi$ is $\cC^{2}$, then in fact we have
\begin{align*}
    D_{mm} \Phi(m,x,x') = D_m \big[ D_m \Phi(\cdot,x) \big](m,x').
\end{align*}
Finally, for $\alpha \in (0,1)$, we say that $\Phi$ is $\cC^{2 + \alpha}$, written $\Phi \in \cC^{2+\alpha}\big(\cP^+(\ov D)\big)$, if $\Phi$ is $\cC^2$ and in addition the derivatives $D_{mm} \Phi$ and $D_x D_m \Phi$ are (globally) $\alpha$-H\"older continuous in their arguments. 

We refer to \cite{CDLL_2019} or \cite{CarmonaDelarue_book_I} for details on the derivatives discussed above and their relationships to each other.
We note that we use very similar definitions for $\cP(\ov D)$ instead of $\cP^+(\ov D)$. Finally, we also use similar definitions and notation if $\Phi$ also depends on an additional finite-dimensional parameter.

\section{The equation for $V^N$} \label{sec.vn}

This section is devoted to the well-posedness of \eqref{eqn.VN}. We will primarily be interested in the case that 
\begin{align*}
    G^N(\bx) = G(m_{\bx}^N), \text{ for some } G : \cP(\ov D) \to \R, 
\end{align*}
but none of the results in this section require this structural condition.

There are two subtleties in studying \eqref{eqn.VN}. First, the non-local lateral boundary condition is only specified on $\partial^1(D^N)$, rather than all of $\partial(D^N)$. This means that for \eqref{eqn.VN} to be well-posed, the values of $V^N$ on $\partial (D^N) \setminus \partial^1(D^N)$ must somehow be irrelevant. This makes sense because the probability that two Brownian particles will reach the boundary at the same time is zero, but it is not clear a priori how to see this from an analytical viewpoint. Second, we do not assume that $G^N$ is compatible with the lateral boundary conditions, i.e. we do \textit{not} assume that 
\begin{align*}
    G^N(\bx) = \frac{1}{N-1} \sum_{j \neq i} G^N(\bx^{i,j}) \text{ for } \bx \in \partial^1(D^N), \, x^i \in \partial D. 
\end{align*}
As a consequence, the solution will not be continuous at the "corner" $\{0\} \times \partial(D^N)$. In light of these points, it is reasonable to work with solutions which are continuous on the domain 
\begin{align*}
      \text{Dom}_N \coloneqq \Big( \R_+ \times \big(D^N \cup \partial^1(D^N)\big) \Big) \bigcup \Big( \{0\} \times D^N \Big). 
\end{align*}
Here is the precise definition we will use.
\begin{defn} \label{def.VN.classicalsoln}
    We say that $V^-$ is a classical subsolution of \eqref{eqn.VN} if $V^- : \text{Dom}_N \to \R$ is continuous, bounded from above, smooth on $\R_+ \times D^N$, and satisfies
    \begin{align} \label{eqn.VN.subsol}
    \begin{cases}
        \ds \partial_t V^-(t,\bx) - \frac{1}{2} \sum_{i = 1}^N \Delta_{x^i} V^-(t,\bx) \leq  0, \quad (t,\bx) \in \R_+ \times D^N,
       \vspace{.1cm} \\
        \ds \quad  V^-(t,\bx) \leq \frac{1}{N-1} \sum_{j \neq i} V^-\big(t,\bx^{i,j}\big), \quad t \in \R_+, \,\, \bx \in \partial^{1,i}(D^N),
       \vspace{.1cm} \\
        \ds V^-(0,\bx) \leq  G^N(\bx), \,\, \bx \in D^N.
    \end{cases}
\end{align}
in a pointwise fashion.
Classical supersolutions are defined analogously. Finally, we say that $V$ is a classical solution if it is both a classical subsolution and a classical supersolution. 
\end{defn}

The main result of this section is as follows.

\begin{thm} \label{thm.vneqn}
    Suppose that $G^N : \ov{D}^N \to \R$ is of class $C^{2+\alpha}(\overline D^N)$. Then there is a unique classical solution to \eqref{eqn.VN}.
\end{thm}

\begin{proof}
  Existence is proved in Proposition \ref{prop.vn.eu}, and uniqueness is a consequence of the comparison principle in Proposition \ref{prop.comparison}. 
\end{proof}

\subsection{Comparison}

The uniqueness part of Theorem \ref{thm.vneqn} will come from the following comparison principle.

\begin{prop} \label{prop.comparison}
    Let $V^-$ be a classical subsolution of \eqref{eqn.VN} and $V^+$ be a classical supersolution. Then $V^- \leq V^+$ on $\text{Dom}_N$. 
\end{prop}

In order to prove Proposition \ref{prop.comparison}, we will build two useful supersolutions of the equation. We start by defining
\begin{align*}
    f : [0,\infty)^2 \setminus \{(0,0)\} \to \R, \quad f(z,z') = \frac{z + z'}{(z)^2 + (z')^2}.
\end{align*}
Then, $f$ is smooth and harmonic, i.e. 
\begin{align*}
    \partial_{zz} f + \partial_{z'z'} f = 0 \text{ in } \R_+^2.
\end{align*}
We then define 
\begin{align*}
    g : [0,\infty)^2 \setminus \{(0,0)\} \to \R, \quad g(z,z') = f^{1/2}(z,z') = \sqrt{\frac{z + z'}{(z)^2 + (z')^2}}.
\end{align*}
Now, define 
\begin{align*}
    \xi_1 : (\ov D)^2 \setminus \partial^2(D^2) \to \R, \quad \xi_1(x,x') = g\big(\dist(x), \dist(x')\big).
\end{align*}

The main point is that the function $\xi_1$ blows up on $\partial^2(D^2)$ (in the sense that it is ${\rm LSC}(\overline{D^2})$ once we set $\xi_1 = +\infty$ on $\partial^2(D^2)$), and $\partial_{xx} \xi_1 + \partial_{x'x'} \xi_1$ is bounded from above. In particular, we have the following.

\begin{lem} \label{lem.phi.supersol}
    There is a constant $C$ such that the function $\xi_1$ satisfies 
    \begin{align} \label{vderiv.formula}
        - \frac{1}{2} \Delta_{x} \xi_1 - \frac{1}{2} \Delta_{x'} \xi_1 \geq -C \xi_1  \text{ in } D^2.
    \end{align}
\end{lem}
\begin{proof}
By explicit computation,
we have 
\begin{align*}
   &\partial_{z} g = \frac{1}{2} f^{-1/2} \partial_{z} f, 
   \qquad  \partial_{zz} g = \frac{1}{2} f^{-1/2} \partial_{zz} f - \frac{1}{4} f^{-3/2} |\partial_{z} f|^2,
\end{align*}
and likewise for $\partial_{z'} g$ and $\partial_{z'z'} g$. As a consequence, we can compute
\begin{align*}
    &\Delta_{x} \xi_1(x,x') = \partial_{zz} g\big(\dist(x), \dist(x')\big) |D \dist (x)|^2 + \partial_{z} g\big(\dist(x),\dist(x')\big) \Delta \dist(x), 
\end{align*}
and similarly for $\Delta_{x'} \xi_1$.
Setting $z = \dist(x)$, $z' = \dist(x')$ for simplicity, we can thus compute
\begin{align*}
    - \frac{1}{2} &\Delta_{x} \xi_1(x,x') - \frac{1}{2} \Delta_{x'} \xi_1(x,x')  = - \frac{1}{2} \partial_{zz} g(z,z') |D\dist(x)|^2  - \frac{1}{2} \partial_{z'z'} g(z,z') |D\dist(x')|^2
    \\
    &\qquad \qquad \qquad \qquad \qquad \qquad \qquad  - \frac{1}{2} \partial_{z} g(z,z') \Delta \dist(x) - \frac{1}{2} \partial_{z'} g(z,z') \Delta \dist(x').
\end{align*}
Plugging in the formulas for $\partial_z g$ and $\partial_{zz} g$, we find
\begin{align} \label{phi.laplacian}
  - \frac{1}{2} \Delta_{x} \xi_1(x,x') & - \frac{1}{2} \Delta_{x'} \xi_1(x,x')  = - \frac{1}{4} f^{-1/2}(z,z')  \Big( \partial_{zz} f(z,z') |D\dist(x)|^2 + \partial_{z'z'} f(z,z') |D\dist(x')|^2 \Big)
  \nonumber   \\
    &\qquad \qquad + \frac{1}{8} f^{-3/2}(z,z') \Big(|\partial_{z}f(z,z')|^2 |D\dist(x)|^2 + |\partial_{z'} f(z,z')|^2 |D\dist(x')|^2 \Big) 
   \nonumber  \\
    &\qquad \qquad - \frac{1}{4} f^{-1/2}(z,z') \Big( \partial_{z} f(z,z') \Delta \dist(x) + \partial_{z'} f(z,z') \Delta \dist(x') \Big).
\end{align}
Next, for some $\eps > 0$, we have $|D\dist(x)| = 1$ for $x \in \big(\partial D\big)_{\eps}$. Together with the fact that $f$ is harmonic, this means \eqref{phi.laplacian} simplifies to 
\begin{align} \label{phi.laplacian.nearboundary}
    - \frac{1}{2} \Delta_{x} \xi_1(x,x') - &\frac{1}{2} \Delta_{x'} \xi_1(x,x') = \frac{1}{8} f^{-3/2}(z,z') \big(|\partial_{z} f(z,z')|^2 + |\partial_{z'}f(z,z')|^2 \big) 
  \nonumber   \\
    & - \frac{1}{4} f^{-1/2}(z,z') \big(\partial_{z} f(z,z') \Delta \dist(x) + \partial_{z'} f(z,z') \Delta \dist(x')\big) =: I - II \text{ in } \big(\partial D\big)_{\eps}^2.
\end{align}
We now estimate the term $II$ via Young's inequality; for $\delta > 0$, we have
\begin{align*}
   |II| &\leq \frac{\|\Delta \dist\|_{\infty}}{4} f^{-1/2}(z,z') \Big( |\partial_z f(z,z')| + |\partial_{z'} f(z,z')| \Big)
    \\
    &\leq \frac{\|\Delta \dist\|_{\infty}}{4} \Big[ \frac{\delta}{2} f^{-3/2}(z,z')  \big(|\partial_z f(z,z')|^2 + |\partial_{z'}f(z,z')|^2 \big) + \frac{1}{\delta} f^{1/2}(z,z') \Big].
\end{align*}
Choosing $\delta = \frac{1}{\|\Delta \dist\|_{\infty}}$, we get
\begin{align*}
    |II| \leq I + \frac{\|\Delta \dist\|_{\infty}^2}{4} f^{1/2}(z,z') = I + \frac{\|\Delta \dist\|_{\infty}^2}{4} \xi_1(x,x'), 
\end{align*}
and so coming back to \eqref{phi.laplacian.nearboundary}, we have
\begin{align} \label{phi.supersol.nearboundary}
     - \frac{1}{2} \Delta_{x} \xi_1(x,x') & - \frac{1}{2} \Delta_{x'} \xi_1(x,x') \geq -\frac{\|\Delta \dist\|_{\infty}^2}{4} \xi_1(x,x') \text{ in } \big(\partial D\big)_{\eps}^2.
\end{align}
On the other hand, $\xi_1$ is smooth and bounded from below and above on $D^2 \setminus \big(\partial D\big)_{\eps}^2$, it is clear that there exists a constant $C' > 0$ such that
\begin{align*}
      - \frac{1}{2} \Delta_{x} \xi_1(x,x') & - \frac{1}{2} \Delta_{x'} \xi_1(x,x') \geq -C' \xi_1(x,x') \text{ in } D^2 \setminus \big(\partial D\big)_{\eps}^2.
\end{align*}
Taking $C = \frac{\|\Delta \dist\|_{\infty}^2}{4} \vee C'$ proves the claim.
\end{proof}

We next define a $\Xi_1^N : \text{Dom}_N \to \R$ via 
\begin{align} \label{def.Xi1N}
    \Xi_1^N(t,\bx) = \frac{1}{N(N-1)} \sum_{i \neq j} \exp(Ct) \xi_1(x^i,x^j), 
\end{align}
where $C$ is as in the statement of Lemma \ref{lem.phi.supersol}. The following lemma will be used to verify that $\Xi_1^N$ is a supersolution of the Fleming-Viot boundary condition.

\begin{lem} \label{lem.algebra}
For $\bz \in [0,\infty)^N$ set
    \begin{align*}
       h(\bz) =  \sum_{k \neq l} g(z^k,z^l).
    \end{align*}
Then, if $z^i = 0$ and $z^j \neq 0$ for $j \neq i$, we have 
\begin{align} \label{h.ineq}
    h(\bz) \geq \frac{1}{N-1} \sum_{j \neq i}  h\big(\bz^{i,j}\big).
\end{align}
\end{lem}

\begin{proof}
    The proof will follow from the identity 
    \begin{align} \label{algebra}
        g(z,0) + g(0,z') = \frac{1}{\sqrt{z}} + \frac{1}{\sqrt{z'}} \geq 2 \sqrt{\frac{z + z'}{z^2 + (z')^2}} = 2 g(z,z'),  \quad z,z' > 0.
    \end{align}
    which can be checked through straightforward algebra.
    To use \eqref{algebra} to verify \eqref{h.ineq}, suppose without loss of generality that $z^1 = 0$, and $z^j > 0$ for $j > 1$. So, our goal is to prove 
    \begin{align} \label{goal}
         h(\bz) \geq \frac{1}{N-1} \sum_{j >1}  h\big(\bz^{1,j}\big).
    \end{align}
    First, we note that
    \begin{align} \label{hz.comp1}
        h(\bz) = 2 \sum_{k > 1} g(z^k,0) + \sum_{\substack{k,l > 1 \\ k \neq l}} g(z^k,z^l).
    \end{align}
    Meanwhile, for $j > 1$, 
    \begin{align*}
       h(\bz^{1,j}) = \sum_{\substack{k,l > 1 \\ k \neq l}} g(z^k,z^l) + 2 \sum_{k > 1} g(z^j,z^k), 
    \end{align*}
    so that 
    \begin{align} \label{hz.comp2}
        \frac{1}{N-1} \sum_{j > 1} h(\bz^{1,j}) = \sum_{\substack{k,l > 1 \\ k \neq l}} g(z^k,z^l) + \frac{2}{N-1} \sum_{j,k > 1} g(z^j,z^k).
    \end{align}
    In light of \eqref{hz.comp1} and \eqref{hz.comp2}, we see that in order to prove \eqref{goal}, it suffices to show that
    \begin{align} \label{goal2}
        \frac{1}{N-1} \sum_{j,k > 1} g(z^j,z^k) \leq \sum_{j > 1} g(z^j,0).
    \end{align}
    To this end, we use \eqref{algebra} to find 
    \begin{align*}
        \frac{1}{N-1} \sum_{j,k > 1} g(z^j,z^k) &= \frac{1}{N-1} \sum_{j > 1} g(z^j,z^j) + \frac{2}{N-1} \sum_{1 < j < k} g(z^j,z^k)
        \\
        &\leq \frac{1}{N-1} \sum_{j > 1} g(z^j,0)  + \frac{1}{N-1} \sum_{1 < j < k} \Big( g(z^j,0) + g(z^k,0) \Big) 
        \\
        &= \frac{1}{N-1} \sum_{j > 1} g(z^j,0)  + \frac{1}{2(N-1)} \sum_{1 < j} \sum_{\substack{k > 1 \\ k \neq j}} \Big( g(z^j,0) + g(z^k,0) \Big)
        \\
        &=  \frac{1}{N-1} \sum_{j > 1} g(z^j,0)  + \frac{(N-2)}{(N-1)} \sum_{1 < j} g(z^j,0) = \sum_{j > 1} g(z^j,0). 
    \end{align*}
    This establishes \eqref{goal2}, and completes the proof.
\end{proof}

We now show that $\Xi_1^N$ is a supersolution of \eqref{eqn.VN}.

\begin{lem} \label{lem.PhiN.comp}
    The function $\Xi_1^N$ defined above satisfies 
    \begin{align} \label{PhiN.1}
        \partial_t \Xi_1^N - \frac{1}{2} \sum_{i = 1}^N \Delta_{x^i} \Xi_1^N \geq 0 \text{ on } \R_+ \times D^N, 
    \end{align}
    in addition to 
    \begin{align} \label{PhiN.2}
        \Xi_1^N(t,\bx) \geq \frac{1}{N-1} \sum_{j \neq i} \Xi_1^N\big(t,\bx^{i,j}\big) \text{ for } t \in \R_+, \, \bx \in \partial^{1,i}(D^N).
    \end{align}
\end{lem}

\begin{proof}
    By explicit computation,
    \begin{align*}
        \partial_t \Xi_1^N - \frac{1}{2} \sum_{i = 1}^N \Delta_{x^i} \Xi_1^N = \frac{\exp(Ct)}{N(N-1)} \sum_{i \neq j} \Big( C \xi_1(x^i,x^j) - \frac{1}{2} \Delta_x \xi_1(x^i,x^j) - \frac{1}{2} \Delta_{x'} \xi_1(x^i,x^j) \Big), 
    \end{align*}
    and so \eqref{PhiN.1} follows from Lemma \ref{lem.phi.supersol}.
    
    Meanwhile, the equality \eqref{PhiN.2} comes from Lemma \ref{lem.algebra}, with $\bz = (\dist(x^1),...,\dist(x^N))$. 
\end{proof}

We have one more barrier function to introduce before we prove uniqueness. We start with $h$, a shifted version of the fundamental solution of the 1-D heat equation, defining
\begin{align*}
    h : \R_+ \times \R \to \R, \quad h(t,z) = 1 + \frac{1}{\sqrt{t}} e^{-\frac{z^2}{2t}}.
\end{align*}
Then, we define
\begin{align*}
  \xi_2 : \R_+ \times D \to \R, \quad   \xi_2(t,x) = \exp(C_1 t) h\big(t,\dist(x)\big)^{1/2} + \exp(C_2 t).
\end{align*}
where $C_1,C_2 > 0$ will be chosen later.

\begin{lem} \label{lem.psi.comp}
    For appropriate values of $C_1$ and $C_2$, the function $\xi_2$ satisfies 
    \begin{align} \label{psi.heat}
         \partial_t \xi_2 - \frac{1}{2} \Delta \xi_2 > 0 \text{ in } \R_+ \times D.
    \end{align}
    Moreover, for each fixed $t > 0$, $\xi_2(t,\cdot)$ is constant on $\partial D$ and achieves its strict global maximum there.
\end{lem}
\begin{proof} Setting $z = \dist(x)$ for simplicity, we compute
\begin{align*}
    &\partial_t \xi_2(t,x) =   \frac{1}{2} e^{C_1 t}  h^{-1/2}(t,z) \partial_t h(t,z)  + C_1 e^{C_1 t} h^{1/2}(t,z) + C_2 \exp(C_2 t),  
    \\
    &\Delta \xi_2(t,x) = \frac{1}{2} e^{C_1 t} h^{-1/2}(t,z) \partial_{zz} h(t,z) |D\dist(x)|^2 
    \\
    &\qquad \qquad \quad   - \frac{1}{4} e^{C_1 t}  h^{-3/2}(t,z) |\partial_z h(t,z)|^2 |D\dist(x)|^2+ \frac{1}{2} e^{C_1 t}  h^{-1/2}(t,z) \partial_z h(t,z) \Delta \dist(x). 
\end{align*}
We now choose $\eps$ small enough that we have $|D\dist(x)|^2 = 1$ in $\big(\partial D\big)_{\eps}$, and thus (omitting the arguments $z$ and $x$ for simplicity)
\begin{align} \label{psi.supersolcomp.boundary}
     \partial_t \xi_2 - \frac{1}{2} \Delta \xi_2 &= C_2 \exp(C_2 t) + C_1 e^{C_1 t} h^{1/2} + \frac{1}{8} e^{C_1 t} h^{-3/2} |\partial_z h|^2 - \frac{1}{4} e^{C_1 t}  h^{-1/2} \partial_z h \Delta \dist 
    \nonumber  \\
     &= C_2\exp\big(C_2 t\big) + h^{-1/2} e^{C_1 t} \Big( C_1  h + \frac{1}{8}  h^{-1} |\partial_z h|^2 - \frac{1}{4} \partial_z h \Delta \dist \Big) \text{ in } \R_+ \times \big(\partial D\big)_{\eps}.
\end{align}
By Young's inequality,
\begin{align*}
    \Big| \frac{1}{4} \partial_z h \Delta \dist \Big| \leq \frac{1}{8} h^{-1} |\partial_z h|^2 + \frac{h}{2} \|\Delta \dist\|_{\infty}^2, 
\end{align*}
and so if $C_1 \geq \frac{1}{2} \|\Delta \dist\|_{\infty}^2$, then we have 
\begin{align*}
    \partial_t \xi_2 - \frac{1}{2} \Delta \xi_2 \geq C_2 > 0 \text{ in } \R_+ \times \big(\partial D\big)_{\eps}.
\end{align*}
Meanwhile, on $\R_+ \times \big(D \setminus \big(\partial D\big)_{\eps})$, the function $(t,x) \mapsto h\big(t,\dist(x)\big)$ is smooth and bounded below, with bounded derivatives of all orders. From this, we deduce that with $C_1$ and $\eps$ fixed, there is a constant $C = C(\eps,C_1)$ such that
\begin{align*}
    \partial_t \xi_2 - \frac{1}{2} \Delta \xi_2 \geq C_2 \exp(C_2 t) - C \exp(C_1 t).
\end{align*}
In particular, by choosing $C_1$ and then $C_2$ large enough, we obtain \eqref{psi.heat}.

Finally, the fact that $\xi_2(t,\cdot)$ is constant on $\partial D$ and achieves its global maximum there is a consequence of the fact that $h(t,\cdot)$ is maximized at $0$. 
\end{proof}

We now define 
\begin{align} \label{def.Xi2N}
  \Xi_2^N : \text{Dom}_N \to \R, \quad   \Xi_2^N(t,\bx) = \frac{1}{N}\sum_{i = 1}^N \xi_2\big(t,x^i\big), 
\end{align}
and note that Lemma \ref{lem.psi.comp} easily implies the following. 

\begin{lem} \label{lem.PsiN.comp}
    For appropriate values of $C_1$ and $C_2$, the function $\Xi_2^N$ defined above satisfies 
    \begin{align} \label{PsiN.1}
        \partial_t \Xi_2^N - \frac{1}{2} \sum_{i = 1}^N \Delta_{x^i} \Xi_2^N > 0 \text{ on } \R_+ \times D^N, 
    \end{align}
    in addition to 
    \begin{align} \label{PsiN.2}
        \Xi_2^N(t,\bx) > \frac{1}{N-1} \sum_{j \neq i} \Xi_2^N\big(t,\bx^{i,j}\big) \text{ in } \R_+ \times \partial^{1,i}(D^N).
    \end{align}
\end{lem}

We now complete the proof of comparison. 

\begin{proof}[Proof of Proposition \ref{prop.comparison}]
   Fix $\eps, \eta \in (0,1)$, consider the optimization problem 
   \begin{align} \label{meps.comparison}
       M_{\eps, \eta} = \sup_{(t,\bx) \in \text{Dom}_N} \Big\{ V^-(t,\bx) - V^+(t,\bx) - \eps \Big( \Xi_1^N(t,\bx) +  \Xi_{2,\eta}^N(t,\bx) \Big) \Big\}, 
   \end{align}
   where we set 
   \begin{align*}
      \Xi_{2,\eta}^N(t,\bx) = \Xi_2^N(t+\eta, \bx).
   \end{align*}
   We now claim that for $\eta$ small enough (depending on $\eps$), the problem \eqref{meps.comparison} has at least one optimizer $(\hat t, \hat \bx) \in \text{Dom}_N$. To see this, let $(t_N,\bx_N)$ be an optimizing sequence in $\text{Dom}_N$. Since $V^- - V^+$ is bounded from above by assumption, there is a constant $K$ independent of $\eps$ and $\eta$ such that 
   \begin{align} \label{tnbxn}
       (t_N,\bx_N) \in \Big\{ (t,\bx) \in \text{Dom}_N : \Xi_1^N(t,\bx) +\Xi_{2,\eta}^N (t,\bx) \leq K/\eps \Big\}.
   \end{align}
   Now, by construction, both $\Xi_1^N$ and $\Xi_2^N$ are non-negative and grow exponentially in $t$, and thus we can conclude that $(t_N)$ is bounded. In particular, there is some point $(\hat t, \hat \bx) \in [0,\infty) \times \ov{D}^N$ such that, along a subsequence, 
   \begin{align*}
       (t_N, \bx_N) \to (\hat t, \hat \bx).
   \end{align*}
   Next, again by construction, 
   \begin{align*}
       \Xi_1^N(t,\bx) \to +\infty \text{ as } \bx \to \partial(D^N) \setminus \partial^1(D^N), \text{  uniformly in $t$,}
   \end{align*}
   and so we deduce that $\hat \bx \in D^N \cup \partial^1(D^N)$. Finally, we note that 
   \begin{align*}
       \Xi_{2,\eta}^N(0,\bx) \geq N^{-1} \eta^{-1/4} \text{ for } \bx \in \partial(D^N). 
   \end{align*}
   Since $\Xi_{2,\eta}^N$ is globally continuous, we get that 
   \begin{align*}
       \Xi_{2,\eta}^N(t,\bx) \geq \frac{1}{2} N^{-1} \eta^{-1/4} \text{ in a neighborhood of } \{0\} \times \partial(D^N).
   \end{align*}
   In particular, we see from \eqref{tnbxn} that for $\eta$ small enough (depending on $N$ and $\eps$), we have $(\hat t, \hat \bx) \notin \{0\} \times \partial(D^N)$. Combining the above observations, we deduce that $(\hat t, \hat \bx) \in \text{Dom}_N$, and so by the continuity of $V^-$, $V^+$, $\Xi_{2,\eta}^N$ and $\Xi_{1}^N$ on $\text{Dom}_N$, we find that $(\hat t, \hat \bx)$ is an optimizer. 

    In the rest of the argument, we fix $\eps > 0$, and $\eta > 0$ small enough (relative to $\eps$), and we fix some optimizer $(\hat t, \hat \bx) \in \text{Dom}_N$ for \eqref{meps.comparison}. If $\hat t = 0$ and $\hat \bx \in D^N$, then because $V^-(0,\hat \bx) \leq G^N(\hat \bx) \leq V^+(0,\hat \bx)$, we obtain $M_{\eps,\eta} \leq 0$, which is equivalent to the statement
   \begin{align} \label{comparison.comp}
       V^-(t,\bx) - V^+(t,\bx) \leq \eps \Big( \Xi_1^N(t,\bx) + \Xi_{2}^N(t+\eta,\bx)\Big), \quad \text{ for all } (t,\bx) \in \text{Dom}_N.
   \end{align}
   On the other hand, suppose that $\hat t > 0$. First, $\hat \bx \notin \partial(D^N) \setminus \partial^1(D^N)$, since $\Xi_1^N = +\infty$ on $\partial(D^N) \setminus \partial^1(D^N)$. If $\hat \bx \in \partial^1(D^N)$, with $\hat x^i \in \partial D$, then, using Lemmas \ref{lem.PhiN.comp} and \ref{lem.PsiN.comp} and the definition of sub/supersolutions, 
   \begin{align*}
       &V^-(\hat t, \hat \bx) - V^+(\hat t, \hat \bx) - \eps \Xi_1^N(\hat t, \hat \bx) - \eps \Xi_2^N(\hat t + \eta, \hat \bx)
       \\
       &\quad < \frac{1}{N-1} \sum_{j \neq i} \bigg( V^-\big(\hat t, \hat \bx^{i,j}\big) - V^+\big(\hat t, \hat \bx^{i,j}\big) - \eps \Xi_1^N\big(\hat t, \hat \bx^{i,j}\big) - \eps \Xi_2^N\big(\hat t + \eta, \hat \bx^{i,j}\big) \bigg).
   \end{align*}
   This contradicts the optimality of $(\hat t, \hat \bx)$, so we deduce that any optimizer with $\hat t > 0$ must satisfy $\hat \bx \in D^N$. 
   
   Finally, if $\hat t > 0$ and $\hat \bx \in D^N$, we obtain that 
   \begin{align*}
       \partial_t V^- &= \partial_tV^+ + \eps \partial_t \Xi_1^N + \eps \partial_t \Xi_{2,\eta}^N, 
       \\
       \Delta_{x^i} V^- &\leq \Delta_{x^i} V^+ + \eps \Delta_{x^i} \Xi_1^N + \eps \Delta_{x^i} \Xi_{2,\eta}^N, \quad i = 1,...,N
   \end{align*}
   at the point $(\hat t, \hat \bx)$. Thus, again at the point $(\hat t, \hat \bx)$ and using Lemmas \ref{lem.PhiN.comp} and \ref{lem.PsiN.comp}, 
   \begin{align*}
     0 &\geq   \partial_t V^- - \frac{1}{2} \sum_{i = 1}^N \Delta_{x^i}  V^-
     \\
     &\geq \partial_t V^+ - \frac{1}{2} \sum_{i = 1}^N \Delta_{x^i}  V^+ + \eps \Big( \partial_t \Xi_1^N - \frac{1}{2} \sum_{i = 1}^N \Delta_{x^i} \Xi_1^N \Big) + \eps \Big( \partial_t \Xi_{2,\eta}^N - \frac{1}{2} \sum_{i = 1}^N \Delta_{x^i} \Xi_{2,\eta}^N \Big)
     > 0.
   \end{align*}
   This is a contradiction. We conclude that any optimizer $(\hat t, \hat \bx)$ satisfies $\hat t = 0$, which means that \eqref{comparison.comp} holds for every $\eps > 0$ and $\eta$ small enough. Sending $\eta \to 0$ first and then $\eps \to 0$ completes the proof. 
\end{proof}

By a very similar (and simpler) argument, the following maximum principle for the heat equation holds. The result itself should be classical, but since we have not found an explicit statement in the literature, we provide it here below (note that the subsolution is not required to be continuous on the whole parabolic cylinder $[0,\infty)\times\overline{D^N}$, and the sign condition is prescribed only on a strict subset of the lateral boundary).

\begin{prop} \label{prop.heat.comparison}
    Let $W : \text{Dom}_N \to \R$ be a continuous, bounded from above, smooth on $\R_+ \times D^N$ function satisfying
    \begin{align*}
    \begin{cases}
        \ds \partial_t W(t,\bx) - \frac{1}{2} \sum_{i = 1}^N \Delta_{x^i} W(t,\bx) \leq  0, & (t,\bx) \in \R_+ \times D^N,
       \\
        \ds W(t,\bx) \leq 0, & (t, \bx) \in \R_+ \times \partial^{1}(D^N),
       \\
        \ds W(0,\bx) \leq  0& \bx \in D^N.
    \end{cases}
\end{align*}
Then, $W \le 0$ on $\text{Dom}_N$.
\end{prop}

\subsection{Existence}

\begin{prop}\label{prop.vn.eu}
    Suppose that $G^N : \ov{D}^N \to \R$ is of class $C^{2+\alpha}(\overline D^N)$. Then there exists a classical solution to \eqref{eqn.VN}.
\end{prop}

\begin{proof} Let $v \in C({\rm Dom}_N)$ and $u := \mathcal F(v)$ be the solution to
\[
\begin{cases}
\partial_t u(t,\bx) - \frac{1}{2} \sum_{i=1}^{N} \Delta_{x^i} u(t,\bx) = 0, 
& (t,\bx) \in \mathbb{R}_+ \times D^N, \\
u(t,\bx) = \frac{1}{N-1} \sum_{j \ne i} v\bigl(t,x^{i,j}\bigr), 
& t \in \mathbb{R}_+, \; \bx \in \partial^1(D^N), \; x^i \in \partial D, \\
u(0,\bx) = G^N(\bx), 
& \bx \in D^N.
\end{cases} 
\]
Note that $u$ is smooth on $(0,\infty) \times D^N$ and continuous on $(0,\infty) \times \partial^1(D^N)$ and $\{0\}\times D^N$ (by a standard barrier argument, since the Cauchy-Dirichlet boundary conditions are continuous at these points). Let $\{u_n\}$ be the sequence defined as follows
\[
u_0 \equiv \inf_{\bx \in D^N} G^N(\bx), \qquad u_{n} = \mathcal F(u_{n-1}) \ \ \forall n \ge 1.
\]
Since $u_1 = \inf G^N$ on $(0,\infty) \times \partial^1(D^N)$ and $u_1 \ge \inf G^N$ on $t = 0$, by the comparison principle of Proposition \ref{prop.heat.comparison} we have $u_1 \ge u_0$ on ${\rm Dom}_N$. Similarly, assuming that $u_{n} \ge u_{n-1}$ in ${\rm Dom}_N$, by definition
\[
u_{n+1}(t,\bx) = \frac{1}{N-1} \sum_{j \ne i} u_{n}\bigl(t,x^{i,j}\bigr) \ge \frac{1}{N-1} \sum_{j \ne i} u_{n-1}\bigl(t,x^{i,j}\bigr) = u_{n}(t,\bx)
\]
on $(0,\infty) \times  \partial^1(D^N)$, and $u_{n+1}=u_{n}$ at $t = 0$, hence $u_{n+1} \ge u_{n}$ on ${\rm Dom}_N$ again by Proposition \ref{prop.heat.comparison}. By induction, we find that $\{u_n\}$ is a monotone nondecreasing sequence. Since the whole sequence is bounded from above by $\sup_{\bx \in D^N} G^N$ (by an analogous induction argument), there exists a function $V^N$ defined on ${\rm Dom}_N$ such that
\[
\lim_{n \to \infty }u_{n}(\bx) = V^N(\bx) \quad \forall \bx \in {\rm Dom}_N.
\]

We now show that $V^N$ satisfies the desired properties. Clearly, $V^N(0, \bx) = G^N(\bx)$ on $D^N$. In addition, pointwise convergence guarantees that $V^N(t,\bx) = \frac{1}{N-1} \sum_{j \ne i} V^N(t,x^{i,j})$ holds on $(0,\infty) \times \partial^1(D^N)$. Let $K \subset D^N$ be any closed set; on $[0, T] \times K$, standard (interior) Schauder estimates guarantee that, up to subsequences, $u_n$ converges to $V^N$ in $C^{2,1}([0,T] \times K)$. This implies that $V^N$ is a classical solution of the heat equation in $(0,\infty) \times D^N$, and it is continuous in $[0,\infty) \times D^N$. 

Finally, let $\bx \in \partial^1(D^N)$ and $B$ be a closed ball centered at $\bx$, such that $B \cap \overline D^N \subset D^N \cup \partial^1(D^N)$.  Note that as $\bx$ varies in $B \cap \partial^1(D^N)$, the components $x^j$ are bounded away from $\partial D$ for every $j \neq i$ (for some $i$ not depending on $\bx$). Since $u_n$ is bounded in $C^{2+\alpha,1+\alpha/2}([0,T] \times K)$, by the identity
\[
u_{n+1}(t,\bx) = \frac{1}{N-1} \sum_{j \ne i} u_{n}\bigl(t,x^{i,j}\bigr)
\]
we have that $u_{n+1}$ is of class $C^{2+\alpha,1+\alpha/2}([\tau,T] \times (B \cap \partial^1(D^N)))$, and therefore the sequence is bounded in $C^{2+\alpha,1+\alpha/2}$ $([\tau, T]\times (\tilde B \cap \overline D^N))$ by (boundary) Schauder estimates, for all $\tau > 0$ and for $\tilde B \subset B$. The uniform convergence (up to subsequences) then shows that $V^N$ is continuous on $(0,\infty) \times \partial^1(D^N)$.

\end{proof}

We close this section with an example which shows that in general, we do not have continuity on all of $\R_+ \times \ov{D}^N$.

\begin{example}
    We work in dimension $d=1$, with $D= (-1,1)$ and $N=2$. We look at the unique classical solution $u$ to 
$$
\left\{\begin{array}{l}
\ds \partial_t u - \frac{1}{2} \partial_{xx} u - \frac{1}{2} \partial_{yy} u = 0 \; {\rm in }\; (0,\infty)\times D^2, \vspace{.2cm}\\
\ds u(t,\pm 1,y) = u(t,y,y), \; u(t,x, \pm 1)= u(t,x,x) \qquad \forall x,y\in D, \; t>0, \vspace{.2cm} \\
\ds u(0,x,y)= x+y \;{\rm in}\; D^2.
\end{array}\right.
$$
We are going to show by contradiction that $u$ cannot be continuous on $(0,\infty)\times (\ov D)^2$. We first claim that, if $u$ is continuous on  $(0,\infty)\times (\ov D)^2$, then
$$
u(t,1,1)=0\qquad \forall t>0.
$$
Indeed, by uniqueness of the solution and symmetry, we have 
$$
u(t,x,y)= u(t,y,x), \qquad u(t,x,y)=-u(t,-x,-y)\qquad \forall x,y \in D, \; t>0.
$$
This implies that $u(t,x,-x)= 0$ for $x\in D$ and thus, by continuity, 
$$
u(t,x,-x)= 0 \qquad \forall x\in [-1,1]. 
$$
So we have, still by the assumed continuity,  
$$
u(t,1,1)= \lim_{x\to 1^-} u(t,x,1)= \lim_{x\to 1^-} u(t,x,x) = \lim_{x\to 1^-} u(t,x,-1)=u(t, 1,-1)=0,
$$
which proves the claim. 

We now build a subsolution $w$ of the equation satisfied by $u$, which is positive at $(t,1,1)$ for any  $t>0$ small, thus finding a contradiction. Let $\phi:[-2,2]\to \R$ be a smooth map such that 
$$
\phi(s)\leq s\; {\rm for }\; s\in [-2,2], \qquad \phi'(2)<0, \; \phi'(-2)>0, \; \phi(2)>0.
$$
For instance, we can choose $\phi$ as a smooth convolution of the map $\tilde \phi$ defined by 
$$
\tilde \phi(s)=s-1/2 \;{\rm for}\; s\in [-2, 3/2], \qquad \tilde \phi(s)= 1- (s-3/2) \; {\rm for}\; s\in [3/2,2].
$$
We set 
$$
w(t,x,y)= \phi(x+y) -A(x-y)^2 - Bt,
$$
for $A,B>0$ to be chosen below. We now check that $w$ is a subsolution to the equation satisfied by $u$. For the initial condition, we note that, for any $(x,y)\in (-1,1)^2$,  
\begin{align*}
w(0,x,y) & \leq  \phi(x+y) \leq x+y=u(0,x,y). 
\end{align*}
We now check the boundary condition. Let $\delta>0$ be such that $\phi'\leq 0$ on $[2-\delta,2]$ and $\phi'\geq 0$ on $[-2,-2+\delta]$. If $y\in[ 1-\delta/2,1]$, then $2\geq 1+y\geq 2y\geq 2-\delta$ and thus, as  $\phi'\leq 0$ on $[2-\delta,2]$, 
\begin{align*}
w(t,1,y) -w(t,y,y)  = \phi(1+y)-A(1-y)^2-Bt -(\phi(2y)-Bt) \leq \phi(1+y)-\phi(2y)  \leq0.
\end{align*}
In the same way, if $y\in [-1, -1+\delta/2]$, we have $-2+\delta \geq 2y \geq -1+y\geq -2$, so that, as  $\phi'\geq 0$ on $[-2,-2+\delta]$,  
$$
w(t,-1,y) -w(t,y,y) \leq \phi(-1+y)-\phi(2y) \leq 0.
$$
On the other hand, if $y\in [-1+\delta/2, 1-\delta/2]$, then (as $1-y\geq \delta/2$) 
$$
w(t,1,y) -w(t,y,y) = \phi(1+y)-\phi(2y) -A(1-y)^2 \leq 2\|\phi\|_\infty -\frac{A\delta^2}{4} \leq 0,
$$
if we choose $A\geq 8\|\phi\|_\infty \delta^{-2}$. In the same way  (as $-1-y\leq -\delta/2$),  
$$
w(t,-1, y) -w(t,y,y) =  \phi(-1+y)-\phi(2y) -A(-1-y)^2 \leq 2\|\phi\|_\infty -\frac{A\delta^2}{4} \leq 0. 
$$
So in any case, 
$$
w(t,\pm 1,y) \leq w(t,y,y) \;{\rm for}\; y\in [-1,1],
$$
which implies by symmetry that 
$$
w(t,x,\pm 1) \leq w(t,x,x) \;{\rm for}\; x\in [-1,1].
$$
Finally, 
$$
\partial_t w - \frac{1}{2} \partial_{xx} w - \frac{1}{2} \partial_{yy} w \leq  -B +\|\phi''\|_\infty +4A \leq 0, 
$$
if we can choose $B$ large enough (depending on  $A$ and $\|\phi''\|_\infty$). Thus  $w$ is a subsolution of the equation satisfied by $u$. By comparison, we obtain $w\leq u$. But $w(t,1,1)= \phi(2)-Bt$, which is positive for $t>0$ small: this contradicts the fact that $u(t,1,1)=0$. 
\end{example}

\section{The function $U$ and its projection}
\label{sec.Uprops}

Throughout this section, we fix a function 
\begin{align*}
    G \in \cC^{2+\alpha}\big( \cP(\ov D) \big),
\end{align*}
and study the function $U(t,m) = G(\ov p_t^m)$, with $\ov p_t$ the re-normalized heat flow. Our aim is to show that $U$ satisfies the equation \eqref{eqn.U}, and to study how the derivatives of $U$ blow up when the $m$ concentrates near the boundary.

We will assume that $G \in \cC^{2+\alpha}\big(\cP(\ov D)\big)$, even if it is not stated explicitly.
We use the notation
\begin{align*}
    \norm{ G }_{\cC^{2}} = \norm{ \frac{\delta G}{\delta m}}_{\infty}  + \norm{\frac{\delta^2 G}{\delta m^2}}_{\infty} + \norm{D_m G }_{\infty} + \norm{D_{mm} G}_{\infty}.
\end{align*}
Throughout this section, we view $G$ as fixed, and we make the following convention regarding constants:
\begin{align*}
    A \lesssim B \text{ means } A \leq CB, \text{ where $C$ depends only on $D$ and on $\norm{ G}_{\cC^{2}}$}.
\end{align*}
Similarly, $A \lesssim_E B$ means $A \leq CB$, where $C$ can depend on the additional parameter $E$, in addition to $D$ and $\norm{ G}_{\cC^{2}}$.

\subsection{Definition and basic properties}

For $m \in \cP(\ov D)$, we denote by $m\big|_D \in \sub(D)$ the restriction of $m$ to $D$, and we denote by $(p_t^m)_{t \geq 0}$ the unique solution in $C\big([0,\infty); \sub\big)$ of
\begin{align} \label{eqn.p}
    \partial_t p = \frac{1}{2} \Delta p \text{ in } \R_+ \times D, \quad p\big|_{(0,\infty) \times \partial D} = 0, \quad p_0 = m\big|_D.
\end{align}
This means that the map $[0,\infty) \ni t \mapsto p_t \in \big(\sub(D), \bd  \big)$ is continuous, $p_0 = m\big|_{D}$, and $p$ is a classical solution to \eqref{eqn.p} in $\R_+ \times D$. We note that the existence and uniqueness of a classical solution on $\R_+ \times D$ which is continuous at time $0$ in a distributional sense is standard; the continuity with respect to the metic $\bd$ is discussed in Appendix \ref{app.flow}. We also use the notation $p_t^x = p_t^{\delta_x}$. For notational simplicity, we set 
\begin{align*}
    \ov{p}^m_t = \frac{p^m_t}{p^m_t(D)} \in \cP(D).
\end{align*}
We define
\begin{align} \label{def.U}
    U : [0,\infty) \times \cP^+(\ov D) \to \R, \quad U(t,m) = G\big(\ov p_t^m\big) = G\Big( \frac{p_t^m}{p_t^m(D)}\Big)
\end{align}

We now explore the properties of the function $U$.

\begin{prop} \label{prop.Uprops}
 The function $U$ is continuous on $\R_+ \times \cP^+(\ov D)$. In addition, it is continuous on $[0,\infty) \times \cP(K)$, for any compact subset $K$ of $D$. Moreover, we have 
    \begin{align*}
         U(t,\cdot) \in \cC^{2}\big( \cP^+(\ov D) \big) \text{ for each $t > 0$,}
    \end{align*}
    with $\frac{\delta U}{\delta m}$ and $\frac{\delta^2 U}{\delta m^2}$ given explicitly by
     \begin{align} \label{U.firstderiv.formula}
        \frac{\delta U}{\delta m}(t,m,x) = \big(p_t^m(D)\big)^{-1} \int_D \frac{\delta G}{\delta m}\big( \ov p_t^m, y \big) p_t^x(dy). 
    \end{align}
    and 
    \begin{align} \label{U.secondderiv.formula}
        \frac{\delta^2 U}{\delta m^2}(t,m,x,x') &= \frac{p_t^m(D) - p_t^{x'}(D)}{(p_t^m(D))^2} \int_{D} \frac{\delta G}{\delta m}\big( \ov p_t^m, y \big) p_t^x(dy)
   \nonumber \\
    &\qquad \qquad  + \big(p_t^m(D)\big)^{-2} \int_{D} \int_{D} \frac{\delta^2 G}{\delta m^2} \big(\ov p_t^m, y,y'\big)  p_t^x(dy) p_t^{x'}(dy').
    \end{align}
    In particular, we have $\frac{\delta U}{\delta m}(t,m,x) = 0$ for $x \in \partial D$. 
\end{prop}

\begin{rmk} \label{rmk.finitedim.pde}
    By duality, the formulas \eqref{U.firstderiv.formula} and \eqref{U.secondderiv.formula} can alternatively be interpreted as 
    \begin{align} \label{dmU.alt}
        \frac{\delta U}{\delta m}(t,m,x) = (p_t^m(D))^{-1} u^1(t,x), 
    \end{align}
    and 
    \begin{align} \label{dmmU.alt}
        \frac{\delta^2 U}{\delta m^2}(t,m,x,x') = \frac{\big(p_t^m(D) - u^2(t,x')\big)}{(p_t^m(D))^2} u^1(t,x) + \big(p_t^m(D)\big)^{-2} v(t,x,x'), 
    \end{align}
    where $u^1$ and $u^2$ satisfy 
    \begin{align} \label{ui.eqn}
        \partial_s u^i(s,x) - \frac{1}{2} \Delta u^i(s,x) = 0, \quad (s,x) \in \R_+ \times D, \quad u^i\big|_{\R_+ \times \partial D} = 0, 
    \end{align}
    with initial conditions 
    \begin{align} \label{ui.term}
        u^1(0,x) = \frac{\delta G}{\delta m}\big(\ov p_t^m , x\big), \quad u^2(0,x) = 1, 
    \end{align}
    while $v$ satisfies 
    \begin{align} \label{v.eqn}
        \partial_s v(s,x,x') - \frac{1}{2} \Delta_x v(s,x,x') - \frac{1}{2} \Delta_{x'} v(s,x,x') = 0, \quad (s,x,x') \in \R_+ \times D \times D
    \end{align}
    with the initial condition 
    \begin{align}
        v\big|_{\R_+ \times \partial(D \times D)} = 0, \quad v(0,x,x') = \frac{\delta^2 G}{\delta m^2}(\ov{p}_t^m,x,x'). 
    \end{align}

    We emphasize that $u^1$, $u^2$, and $v$ depend implicitly on $t$ and $m$, through their initial conditions, so that the notation $u^{1,t,m}$, $u^{2,t,m}$, $v^{t,m}$ would be more precise, but we prefer to avoid this additional notation.
\end{rmk}

\begin{proof}[Proof of Proposition \ref{prop.Uprops}] 
First, the claimed continuity follows from the continuity properties of $(t,m) \mapsto p_t^m$ from Lemma \ref{lem.ptm}, together with the continuity of $G$.

We now turn to the differentiability of $U(t,\cdot)$. For notational simplicity, we define 
    \begin{align*}
        \Phi(t,m,x) &= \big(p_t^m(D)\big)^{-1} \int_D \frac{\delta G}{\delta m}\big( \ov p_t^m, y \big) p_t^x(dy), 
        \\
        \Psi(t,m,x,x')  &= \frac{p_t^m(D) - p_t^{x'}(D)}{(p_t^m(D))^2} \int_{D} \frac{\delta G}{\delta m}\big( \ov p_t^m, y \big) p_t^x(dy)
   \nonumber \\
    &\qquad \qquad  + \big(p_t^m(D)\big)^{-2} \int_{D} \int_{D} \frac{\delta^2 G}{\delta m^2} \big(\ov p_t^m, y,y'\big)  p_t^x(dy) p_t^{x'}(dy'), 
    \end{align*}
    so our goal is to verify that for $t > 0$ and $m \in \cP^+(\ov D)$, we have the formulas 
    \begin{align*}
        \frac{\delta U}{\delta m} = \Phi, \quad \frac{\delta^2 U}{\delta m^2} = \Psi. 
    \end{align*}
    Fix $t > 0$, $m,m' \in \cP^+(\ov D)$, and define $m^{\eps} = [m,m']_{\eps}$, for $\eps \in (0,1)$. 
    The first step is to recognize that, by the definition of the linear derivative of $G$ and the normalization convention $\int \frac{\delta G}{\delta m}(m,y) m(dy) = 0$, we have
    \begin{align} \label{firstderivcomp.1}
         G\big( \ov{p}_t^{m^{\eps}} \big) - G(\ov{p}_t^m) = \int_D \frac{\delta G}{\delta m}(\ov p_t^m,y) \ov p_t^{m^{\eps}}(dy) + o\big( \bd_1(\ov p_t^{m^{\eps}},\ov p_t^{m})\big)
    \end{align}
    The next step is to recognize that, by the linearity of the map $m \mapsto p_t^m$, we have 
    \begin{align} \label{meps.formula}
       p_t^{m^{\eps}} = (1-\eps) p_t^{m} + \eps p_t^{m'} \implies \ov{p}_t^{m^{\eps}} = \frac{ (1-\eps) p_t^{m} + \eps p_t^{m'}}{ (1-\eps) p_t^{m}(D) + \eps p_t^{m'}(D)}.
    \end{align}
    In light of the normalization convention for $\frac{\delta G}{\delta m}$, this yields
    \begin{align} \label{firstderiv.2}
        \int_D \frac{\delta G}{\delta m}(\ov p_t^m,y) \ov p_t^{m^{\eps}}(dy)
        &= \frac{\eps}{(1-\eps) p_t^m(D) + \eps p_t^{m'}(D)} \int_D \frac{\delta G}{\delta m}\big( \ov p_t^m, y \big) p_t^{m'}(dy)
      \nonumber  \\
        &= \frac{\eps}{p_t^m(D)}  \int_D \frac{\delta G}{\delta m}\big( \ov p_t^m, y \big) p_t^{m'}(dy) + o(\eps).
    \end{align}
    Moreover, one can check from the formula \eqref{meps.formula} that $\bd_1(\ov p_t^{m^{\eps}},\ov p_t^{m}) = O(\eps)$. In light of \eqref{firstderivcomp.1} and \eqref{firstderiv.2}, we use the fact that $p_t^{m'}(dy) = \int_{D} p_t^x(dy) m'(dx)$ to get
    \begin{align} \label{firstderiv.3}
        U(t, m^{\eps}) - U(t,m) = G\big( \ov{p}_t^{m^{\eps}} \big) - G(\ov{p}_t^m) &= \frac{\eps}{p_t^m(D)} \int_D \frac{\delta G}{\delta m}\big( \ov p_t^m, y \big) p_t^{m'}(dy) + o(\eps)
       \nonumber  \\
        &= \frac{\eps}{p_t^m(D)} \int_D \int_D \frac{\delta G}{\delta m}\big( \ov p_t^m, y \big) p_t^{x}(dy) m'(dx) + o(\eps)
       \nonumber  \\
        &= \eps \int_D \Phi(t,m,x) m'(dx) + o(\eps).
    \end{align}
    Moreover, by the normalization condition for $\frac{\delta G}{\delta m}$, we have 
    \begin{align} \label{firstderiv.4}
        \int_D \Phi(t,m,x) m(dx) = \frac{1}{p_t^m(D)} \int_D \int_D \frac{\delta G}{\delta m}(\ov p_t^m,y) p_t^x(dy) m(dx) =  \int_D \int_D \frac{\delta G}{\delta m}(\ov p_t^m,y) \ov{p}_t^m(dy) = 0.
    \end{align}
    Combining \eqref{firstderiv.3} and \eqref{firstderiv.4} shows that $\frac{\delta U}{\delta m} = \Phi$. 

    We now turn to the computation of the second linear derivative. We fix $t > 0$, $m,m' \in \cP^+(\ov D)$ and $x \in D$, and again set $m^{\eps} = [m,m']_{\eps}$, for $\eps \in (0,1)$. We start by computing 
    \begin{align*}
        \frac{\delta U}{\delta m}&(t,m^{\eps},x) - \frac{\delta U}{\delta m}(t,m,x) = \Phi(t,m^{\eps},x) - \Phi(t,m,x) 
        \\
        &= \big(p_t^{m^{\eps}}(D)\big)^{-1} \int_D \frac{\delta G}{\delta m}\big( \ov{p}_t^{m^{\eps}}, y\big) p_t^x(dy) - \big(p_t^{m}(D)\big)^{-1} \int_D \frac{\delta G}{\delta m}\big( \ov{p}_t^{m}, y\big) p_t^x(dy)
        \\
        &= \Big( \big(p_t^{m^{\eps}}(D)\big)^{-1} - \big(p_t^{m}(D) \big)^{-1} \Big) \int_D \frac{\delta G}{\delta m}\big( \ov{p}_t^{m^{\eps}}, y\big) p_t^x(dy)
        \\
        &\qquad \qquad + \big(p_t^m(D)\big)^{-1} \int_D \Big( \frac{\delta G}{\delta m}\big( \ov{p}_t^{m^{\eps}}, y\big) - \frac{\delta G}{\delta m} \big( \ov{p}_t^{m}, y\big)\Big) p_t^x(dy) =: I + II.
    \end{align*}
    Again using the expression $p_t^{m^{\eps}} = (1-\eps) p_t^m + \eps p_t^{m'}$, and the Lipschitz continuity of $\frac{\delta G}{\delta m}(\cdot,y)$, we deduce that 
    \begin{align} \label{I.comp}
        I &=  \eps \frac{ \big( p_t^{m}(D) - p_t^{m'}(D)\big)}{(p_t^m(D))^2}  \int_D \frac{\delta G}{\delta m}\big( \ov{p}_t^{m}, y\big) p_t^x(dy) + o\big( \bd_1(\ov p_t^{m^{\eps}}, \ov p_t^m) \big)
      \nonumber   \\
        &= \eps \frac{ \big( p_t^{m}(D) - p_t^{m'}(D)\big)}{(p_t^m(D))^2}  \int_D \frac{\delta G}{\delta m}\big( \ov{p}_t^{m}, y\big) p_t^x(dy) + o(\eps ). 
    \end{align}
    On the other hand, arguing exactly as in the computation of $\frac{\delta G}{\delta m}$, we have 
    \begin{align*}
        \frac{\delta G}{\delta m}\big( \ov p_t^{m^{\eps}},y \big) -  \frac{\delta G}{\delta m}\big( \ov p_t^{m},y \big) = \frac{\eps}{p_t^m(D)} \int_D \frac{\delta^2 G}{\delta m^2} (\ov p_t^m,y,y') p_t^{m'}(dy') + o(\eps), 
    \end{align*}
    where the constant hidden in the $o(\eps)$ term is uniform in $y$ (but depends on $m$ and $m'$). We deduce that 
    \begin{align} \label{II.comp}
        II = \big(p_t^m(D)\big)^{-2} \int_D \int_D \frac{\delta^2 G}{\delta m^2}\big( \ov p_t^m, y,y'\big) p_t^x(dy) p_t^{m'}(dy')+ o(\eps). 
    \end{align}
    Combining $I$ and $II$, and again using $p_t^{m'}(dy') = \int p_t^{x'}(dy') m(dx')$, we get 
    \begin{align} \label{secondderiv.1}
        \frac{\delta U}{\delta m}(t,m^{\eps},x) -  \frac{\delta U}{\delta m}(t,m,x) = \eps \int_D \Psi(t,m,x,x') m'(dx') + o(\eps). 
    \end{align}
    Finally, we notice that 
    \begin{align*}
        \int p_t^{x'}(D) m(dx') = p_t^m(D), \quad (p_t^m(D))^{-1} \int_D \int_D \frac{\delta^2 G}{\delta m^2} (\ov p_t^m, y,y') p_t^{x'}(dy') = \int_D \frac{\delta^2 G}{\delta m^2} (\ov p_t^m, y,y') \ov{p}_t^m(dy') = 0,  
    \end{align*}
    from which it follows that
    \begin{align} \label{seconderiv.2}
        \int_D \Psi(t,m,x,x') m(dx') = 0.
    \end{align}
    Combining \eqref{secondderiv.1} and \eqref{seconderiv.2} completes the proof that $\frac{\delta^2 U}{\delta m^2} = \Psi$.

    We have now shown that for each $t > 0$, $U(t,\cdot)$ admits two linear derivatives, and it is clear from the explicit expressions that they are continuous on their domains in $(t,m,x)$ and $(t,m,x,x')$. Finally, reinterpreting the expressions in Remark \eqref{rmk.finitedim.pde}, we see that $\frac{\delta U}{\delta m}(t,m,\cdot)$ is $C^2$ for each $(t,m)$, and $\frac{\delta^2 U}{\delta m^2}$ is $C^2$ for each fixed $(t,m)$, and we have the formulas
    \begin{align*}
        D_m U(t,m,x) = D_x \frac{\delta U}{\delta m}(t,m,x) = \big(p_t^m(D)\big)^{-1} Du^{1}(t,x), 
    \end{align*}
    and
    \begin{align*}
        D_{mm} U(t,m,x,x') = D_{x'} D_x \frac{\delta^2 U}{\delta m^2} (t,m,x,x') = \big(p_t^m(D)\big)^{-2} \Big( - Du^1(t,x) \big(Du^2(t,x')\big)^T + D_{x'} D_x v(t,x,x') \Big), 
    \end{align*}
    where $u^i$, $i = 1,2$, and $v$ solve \eqref{ui.eqn} and \eqref{v.eqn}, respectively.
    Using the regularity properties of $G$, one can check that $D_mU$ and $D_{mm} U$ are continuous in $(t,m,x)$ and $(t,m,x,x')$, respectively. In particular, this means that $U(t,\cdot)$ is $\cC^{2}$ for each fixed $t > 0$. 

    Finally, we note that $\frac{\delta U}{\delta m}(t,m,x) = 0$ for $x \in \partial D$ follows form the fact that $p_t^x = 0$ for $t > 0$ and $x \in \partial D$. 
\end{proof}

\subsection{Estimates on the derivatives} \label{subsec.Uderivest}

The next goal will be to estimate the derivatives of $U$. We start with the following.

\begin{prop} \label{prop.uderivs.largetime}
    We have the bounds
    \begin{align} \label{delta2U.mainest}
       \norm{ \frac{\delta^2 U}{\delta m^2}(t,m,\cdot,\cdot) }_{\infty} \lesssim \exp(-ct) \Big( \int_D \rho \, dm \Big)^{-2}, \quad t > 0, \, m \in \cP^+(\ov D)
    \end{align}
    and
    \begin{align} \label{DmmU.mainest}
     \| D_{mm} U(t,m,\cdot,\cdot)\|_{\infty} \lesssim {\exp(-ct)}\Big( \int_D \rho\, dm \Big)^{-2}, \quad t \geq 1, \, m \in \cP^+(\ov D), 
    \end{align}
    where $c = \lambda_2 - \lambda_1$.
\end{prop}

\begin{rmk}
    We emphasize that constants hidden by the notation $\lesssim$ can depend on $G$ (or more precisely $\|G\|_{\cC^{2}}$) in this section, as well as the domain $D$, but we do not allow them to depend on $t$. In particular, it is important that the estimates in Proposition \ref{prop.uderivs.largetime} are uniform in $t$. 
\end{rmk}

\begin{proof}
   We first estimate $\frac{\delta^2 U}{\delta m^2}$ for small times, i.e. for $0 < t < 1$. From \eqref{U.secondderiv.formula}, we have
    \begin{align*}
        \Big| \frac{\delta^2 U}{\delta m^2}(t,m,x,x') \Big| &\leq \big(p^m_t(D)\big)^{-2} \bigg( \big(p_t^m(D) + p_t^{x'}(D)\big) \norm{\frac{\delta G}{\delta m}}_{\infty} + p_t^x(D) p_t^{x'}(D) \norm{\frac{\delta^2 G}{\delta m^2}}_{\infty} \bigg).
    \end{align*}
    It remains only to apply Lemma \ref{lem.masslowerbound} to obtain the estimate 
    \begin{align*}
        \Big| \frac{\delta^2 U}{\delta m^2} (t,m,x,x') \Big| \lesssim \Big(\int_D \rho dm \Big)^{-2}, 
    \end{align*}
    which implies that \eqref{delta2U.mainest} holds for $0 < t < 1$.
    We now prove the bound on $\frac{\delta^2 U}{\delta m^2}$ for $t \geq 1$. We will make use of the formula \eqref{U.secondderiv.formula}. We first note that by Lemma \ref{lem.Ulongtime}, for $t \geq 1$ we have 
    \begin{align*}
        \int_D \frac{\delta G}{\delta m}\big(\ov p_t^m, y \big) p_t^x(dy) &= p_t^x(D) \int_D \frac{\delta G}{\delta m}\big(\ov p_t^m, y \big) \ov{p}_t^x(dy)
        \\
        &= p_t^x(D) \int_D \frac{\delta G}{\delta m}\big(\ov p_t^m, y \big) \rho(y) dy + p_t^x(D) O(e^{-(\lambda_2 - \lambda_1)t})
        \\
        &= p_t^x(D) \int_D \frac{\delta G}{\delta m}\big(\ov p_t^m, y \big)\ov p_t^m(dy) + p_t^x(D) O(e^{-(\lambda_2 - \lambda_1)t})
        \\
        &= p_t^x(D) O(e^{-(\lambda_2 - \lambda_1) t}) = O(e^{-\lambda_2 t}),
    \end{align*}
    where we used the convention that $\int_D \frac{\delta G}{\delta m}(\mu,y)\mu(dy) = 0$. 
    Applying Lemma \ref{lem.masslowerbound}, we see that 
    \begin{align} \label{delta2U.comp1}
      \Big|  \frac{p_t^m(D) - p_t^{x'}(D)}{(p_t^m(D))^2} \int_{D} \frac{\delta G}{\delta m}\big( \ov p_t^m, y \big) p_t^x(dy) \Big| \lesssim \Big(\int_D \rho dm \Big)^{-2} \exp\big(-(\lambda_2 - \lambda_1) t \big). 
    \end{align}
    Next, we again use Lemma \ref{lem.Ulongtime} to see that for $t \geq 1$, we have
     \begin{align*}
          \int_{D} \int_{D}& \frac{\delta^2 G}{\delta m^2} \big(\ov p_t^m, y,y'\big)  p_t^x(dy) p_t^{x'}(dy') = p_t^{x'}(D) \int_D \int_D \frac{\delta^2 G}{\delta m^2}\big(\ov p_t^m, y,y'\big)  \ov{p}_t^{x'}(dy') p_t^x(dy)
          \\
          &= p_t^{x'}(D) \int_D \int_D \frac{\delta^2 G}{\delta m^2}\big(\ov p_t^m, y,y'\big) \rho(y') dy' \,  p_t^x(dy)  + p_t^x(D) p_t^{x'}(D) O(e^{-(\lambda_2 - \lambda_1)t})
          \\
          &= p_t^{x'}(D) \int_D \int_D \frac{\delta^2 G}{\delta m^2}\big(\ov p_t^m, y,y'\big) \ov{p}_t^m(dy') p_t^x(dy)  +  p_t^x(D)  p_t^{x'}(D)O(e^{-(\lambda_2 - \lambda_1)t})
          \\
          &=  p_t^x(D)  p_t^{x'}(D) O(e^{-(\lambda_2 - \lambda_1)t}) = O(e^{-(\lambda_1 + \lambda_2)t}),
     \end{align*}
     where we again used the normalization convention for linear derivatives.
     Thus, by Lemma \ref{lem.masslowerbound}, we find that for $t \geq 1$,
     \begin{align} \label{delta2U.comp2}
          \Big| \big(p_t^m(D)\big)^{-2} \int_{D} \int_{D} \frac{\delta^2 G}{\delta m^2} \big(\ov p_t^m, y,y'\big)  p_t^x(dy) p_t^{x'}(dy') \Big| \lesssim \Big(\int \rho dm \Big)^{-2} \exp\big( - (\lambda_2 - \lambda_1) t\big). 
     \end{align}
     By combining \eqref{delta2U.comp1} and \eqref{delta2U.comp2} with \eqref{U.secondderiv.formula}, we obtain the estimate \eqref{delta2U.mainest} for $t \geq 1$. 
     
    For the bound on $D_{mm} U = D_x D_{x'} \frac{\delta^2 U}{\delta m^2}$, we use Remark \ref{rmk.finitedim.pde} to see that 
    \begin{align} \label{dmmu.firststep}
        \big| D_{mm} U(t,m,x,x') \big| \leq \big( p_t^m(D)\big)^{-2}\Big( |Du^1(t,x)||Du^2(t,x')|+ |D_{x'} D_x v(t,x,x')|\Big),
    \end{align}
    where $u^1$, $u^2$, and $v$ are as in Remark \ref{rmk.finitedim.pde}, and in particular satisfy the representation formulae
    \begin{align*}
        u^1(t,x) = \int_D \frac{\delta G}{\delta m}(\ov p_t^m, y) p_t^x(dy), \quad u^2(t,x') = p_t^{x'}(D), \quad v(t,x,x') = \int_D \int_D \frac{\delta^2 G}{\delta m^2}\big(\ov p_t^m,y,y') p_t^x(dy) p_t^{x'}(dy').
    \end{align*}
    The arguments above leading to \eqref{delta2U.mainest} show that
    \begin{align*}
        &\| u^1(t,\cdot)\|_{\infty} \lesssim e^{-\lambda_2 t}, \quad \|u^2(t,\cdot)\|_{\infty} \lesssim e^{-\lambda_1 t}, \quad \| v(t,\cdot,\cdot)\|_{\infty} \lesssim  e^{-(\lambda_1 + \lambda_2) t}, \quad t \in \R_+.
    \end{align*}
   From the smoothing effect of the heat equation, we find 
    \begin{align*}
        \|Du^1(t,\cdot)\|_{\infty} \lesssim &\|u^1(t-1,\cdot)\|_{\infty} \lesssim e^{-\lambda_2 t}, \quad \|Du^2(t,\cdot)\|_{\infty} \lesssim \|u^2(t-1,\cdot)\|_{\infty} \lesssim e^{-\lambda_1 t} 
        \\
        &\|D_{x} D_{x'} v(t,\cdot,\cdot)\|_{\infty} \lesssim \| v(t-1,\cdot)\|_{\infty} \lesssim e^{-(\lambda_1 + \lambda_2) t}, \quad t \geq 1.
    \end{align*}
    By combining this with \eqref{dmmu.firststep} and Lemma \ref{lem.masslowerbound}, we complete the proof of \eqref{DmmU.mainest}.
\end{proof}

In light of Proposition \ref{prop.uderivs.largetime}, to understand the behavior of the second derivatives of $U$, it remains to estimate $D_{mm} U(t,m,x,x')$ in the case that $t \leq 1$. The issue is that $D_{mm} U(t,m,x,x')$ can be large when $t \approx 0$ and either $x$ or $x'$ are near the boundary. To quantify this, it turns out that we need to introduce a barrier function which is built from the transition kernel $\Gamma_r(z,z')$ of the one-dimensional diffusion 
\begin{align*}
    dZ_t = - c_0 \text{sign}(Z_t) dt + dW_t, 
\end{align*}
with $W$ being a one-dimensional Brownian motion, and $c_0 = \|\Delta \dist\|_{\infty}$, $\dist$ being the smooth function on $D$ which agrees with $\dist\big(\cdot, \partial D\big)$ in a neighborhood of the boundary. This transition kernel was computed explicitly in \cite{KaratzasShreve} (see the "Special Case" following Proposition 5.1), and is described by the formula
\begin{align} \label{def.Gamma}
    \Gamma_r(z,z') = \begin{cases}
          \ds   \frac{1}{\sqrt{2\pi r}} \Big\{ \exp\Big[ - \frac{ (z-z' - c_0 r)^2}{2r} \Big] 
         \vspace{.2cm} \\
       \ds    \qquad \qquad  + c_0 \exp(-2c_0 z') \int_{z + z'}^{\infty} \exp\Big[ - \frac{(v - c_0r)^2}{2r}  \Big] dv \Big\}        & z \geq 0, \, z' \geq 0,
           \vspace{.2cm}   \\
          \ds \frac{1}{\sqrt{2\pi r}} \Big\{ \exp\Big[ 2c_0 z - \frac{ (z-z' + c_0 r)^2}{2r} \Big] 
             \vspace{.2cm} \\
         \ds  \qquad \qquad + c_0 \exp(2c_0 z') \int_{z - z'}^{\infty} \exp\Big[ - \frac{(v - c_0r)^2}{2r}  \Big] dv \Big\}        & z \geq 0, \, z' \leq 0,
                        \end{cases}
\end{align}
together with the symmetry 
\begin{align*}
    \Gamma_r(z,z') = \Gamma_r(-z,-z').
\end{align*}
We will primarily use $\Gamma$ through the function
\begin{align*}
    \gamma : \R_+ \times \R \to \R, \quad \gamma(t,z) = \Gamma_{t}(z,0), 
\end{align*}
which is given by the explicit formula 
\begin{align} \label{def.gamma}
    \gamma(t,z) = \frac{1}{\sqrt{2\pi t}} \bigg\{ \exp\Big[ - \frac{\big(|z| - c_0 t\big)^2}{2t} \Big] + c_0 \int_{|z|}^{\infty} \exp\Big[ - \frac{\big(v - c_0 t \big)^2}{2 t}   \Big] dv \bigg\}. 
\end{align}
In other words, $\gamma$ solves
\begin{align*}
    \partial_t \gamma - \frac{1}{2} \partial_{zz} \gamma + c_0 \text{sign}(z) \partial_z \gamma = 0 \text{ in } \R_+ \times \R, \quad \gamma(0,\cdot) = \delta_0.
\end{align*}
Finally, we will define a function 
\begin{align} \label{def.w}
    w :   \R_+ \times \ov D \to \R, \quad w(t,x) = \gamma\big(t,\dist(x) \big) + Ct,
\end{align}
where $C > 0$ is a constant to be determined. 



Our estimate on the derivatives of $U$ is as follows. 
\begin{prop} \label{prop.Uderiv.est.shorttime}
    Suppose that $G \in \cC^{2+\alpha}\big(\cP(\ov D) \big)$, and $w$ is defined by \eqref{def.w}, with $C > 0$ large enough. Then, we have the estimate
    \begin{align*}
        \big| D_{mm} U(t,m,x,x') \big| \lesssim 
            \Big( \int_D \rho \, dm \Big)^{-2}\big( 1 +  w(t,x) \big)\big(1 + w(t,x') \big).\, \quad 0 < t \leq 1, \,\, m \in \cP^+(\ov D), \,\, x,x' \in \ov D.
    \end{align*}
\end{prop}

\begin{proof} This is a matter of combining the representation of $\frac{\delta^2 U}{\delta m^2}$ from Remark \ref{rmk.finitedim.pde} with the estimates from Propositions \ref{prop.ueqn} and \ref{prop.veqn}.
\end{proof}

\subsection{The equation for $U$}

The goal of this subsection is to prove that $U$ satisfies an equation. 

\begin{prop} \label{prop.Ueqn}
   We have 
    $U \in \cC^{1,2}\big(\R_+ \times \cP^c(D) \big)$, and satisfies \eqref{eqn.U} pointwise. 
\end{prop}

The following lemma will be used when we compute $\partial_t U$.

\begin{lem} \label{lem.timederiv2}
For any fixed $t > 0$, $m \in \cP^c(D)$ and $h > 0$ small enough we have 
\begin{align*}
    U(t,\ov p_h^m) - U(t,m) = \int_D \frac{\delta U}{\delta m}(t,m,x) \big( \ov{p}_h^m(dx) - m(dx) \big) + o(h).
\end{align*}
\end{lem}

\begin{proof}
A second order Taylor expansion of $U$ gives 
\begin{align*}
    U(t,\ov p_h^m) - U(t,m) &- \int_D \frac{\delta U}{\delta m}(t,m,x) (\ov p_h^m - m)(dx) 
    \\
    &= \int_0^1 \int_0^1 \int_D \int_D s \bigg[ \frac{\delta^2 U}{\delta m^2}\Big(t, \big[ m, [m,\ov p_h^m]_s\big]_r,x,x' \Big) \bigg] (\ov p_h^m - m)(dx') (\ov p_h^m - m)(dx) ds dr.
\end{align*}
For all $h$ small enough, we will have $\int \rho \, d\ov{p}_h^m \geq \frac{1}{2} \int \rho dm$, and so by convexity  $ \int \rho \, d\big[ m, [m,\ov p_h^m]_s\big]_r \geq \frac{1}{2} \int \rho dm$ for each $s$ and $r$.
Thus, it suffices to prove the estimate
\begin{align*}
  \int_D \int_D \frac{\delta^2 U}{\delta m^2}\big(t,\mu,x,x'\big) (\ov p_h^m - m)(dx') (\ov p_h^m - m)(dx) = o(h), 
\end{align*}
uniformly over $\mu \in \cP(D)$ such that $\int \rho d\mu \geq \delta \coloneqq \frac{1}{2} \int \rho dm$.  
By Proposition \ref{prop.Uprops} and Remark \ref{rmk.finitedim.pde}, we have 
   \begin{align*}
       \frac{\delta^2 U}{\delta m^2}\big(t,\mu,x,x'\big) = K(x) + \big(p_t^{\mu}(D)\big)^{-2} \Big[ -u^1(t,x) u^2(t,x')+ v(t,x,x') \Big],
   \end{align*}
   where $u^1,u^2$, $v$ are as described in Remark \ref{rmk.finitedim.pde} (but with $\mu$ replacing $m$) and $K(x) = \frac{u^1(t,x)}{p_t^{\mu}(D)}$. In particular, since $\int K(x)(\ov p_h^m - m)(dx')=0 $, 
   \begin{align*}
       \int_{ D} &\int_D \frac{\delta^2 U}{\delta m^2}\big(t,\mu,x,x'\big)(\ov p_h^m - m)(dx') (\ov p_h^m - m)(dx)
       \\
       &= \big(p_t^{\mu}(D)\big)^{-2} \int_{D} \int_D \Big[ -u^1(t,x) u^2(t,x')+ v(t,x,x') \Big] (\ov p_h^m - m)(dx') (\ov p_h^m - m)(dx)
      \\
      &\lesssim_t  \Big( \int \rho \, d\mu \Big)^{-2} \bigg( \|u^1(t,\cdot)\|_{W^{2,\infty}} \|u^2(t,\cdot)\|_{W^{2,\infty}} \|\ov p_h^m - m\|_{(W_{0,\text{tr}}^{2,\infty})^*}^2 
      \\
      &\qquad \qquad \qquad \qquad \qquad   +\|\ov p_h^m -m \|_{(W_{0,\text{tr}}^{2,\infty})^*} \times  \norm{x \mapsto \int_D v(t,x,x') (\ov p_h^m  - m)(dx')}_{W^{2,\infty}} \bigg)
      \\
      &\lesssim_t  \Big( \int \rho \, d\mu \Big)^{-2} \Big( h^2 + h \sup_{x \in D} \Big| \int_D {D_x^2} v(t,x,x') (\ov p_h^m  - m)(dx') \Big| \Big)
      \\
      &\lesssim_t \Big( \int \rho \, d\mu \Big)^{-2} \Big( h^2 + h \bd_1(\ov p_h^m ,m) \|D_{x'} D_{xx} v(t,\cdot,\cdot)\|_{\infty} \Big) = o(h).
   \end{align*}
   In the above string of inequalities, we used the notation $(W_{0,\text{tr}}^{2,\infty})^*$ for the dual of the space 
   $$
   W_{0,\text{tr}}^{2,\infty} = \big\{ f \in W^{2,\infty} : f|_{\partial D} = 0, 
   $$
   so that 
   \begin{align*}
       \| \mu - \nu\|_{(W_{0,\text{tr}}^{2,\infty})^*} = \sup \Big\{ \int_D \phi \, d(\mu - \nu) \, : \, \|\phi\|_{W^{2,\infty}} \leq 1, \,\, \phi\big|_{\partial D} = 0 \Big\}. 
   \end{align*}
  We also used repeatedly the smoothing effects of the equations solved by $u^1$, $u^2$, and $v$, as well as the estimate $\| \ov p_h^m - m\|_{(W^{2,\infty})^*} = O(h)$ (which needs $m \in \cP^c(D)$), Lemma \ref{lem.masslowerbound} and the continuity of $h \mapsto \ov{p}_h^m$ from Lemma \ref{lem.ptm}.
\end{proof}

\begin{proof}[Proof of Proposition \ref{prop.Ueqn}]
Our goal is to show that for $t > 0$ and $m \in \cP^c(D)$, $\partial_t U(t,m)$ exists, and is given by the formula
\begin{align} \label{dtu.formula}
    \partial_t U(t,m) = \frac{1}{2} \int_D \Delta_x \frac{\delta U}{\delta m}(t,m,x) m(dx).
\end{align}
To prove this, we begin by noting that $\ov{p}^m_t$ satisfies the "flow property"
\begin{align*}
    \ov{p}^m_{t+h} = \frac{p_{t+h}^m}{p_{t+h}^m(D)} = \frac{p_h^{p_t^m}}{p_h^{p_t^m}(D)} = \frac{p_t^m(D)}{p_t^m(D)} \frac{p_h^{p_t^m}}{p_h^{p_t^m}(D)} = \frac{p_h^{\ov{p}_t^m}}{p_h^{\ov{p}_t^m}(D)} = \ov{p}^{\ov p_t^m}_h, 
\end{align*}
and hence 
\begin{align*}
    U(t,\ov{p}^m_h) = G\Big( \ov{p}^{\ov p_h^m}_t\Big) = G\big( \ov{p}^m_{t+h}\big) = U\big( t + h, m \big).
\end{align*}
Thus, 
\begin{align*}
    U(t + h, m) - U(t,m)= U\big(t, \ov p_h^m\big) - U(t, m).
\end{align*}
Now, using Lemma \ref{lem.timederiv2}, we have
\begin{align}
  \label{timederivcomp}  U\big(t, \ov p_h^m\big) - U(t, m) &=  \int_D \frac{\delta U}{\delta m}\big(t, m, x\big) \big(\ov{p}_h^m(dx) - m(dx) \big) + o(h)
  \nonumber   \\
    &= \big(p_h^m(D)\big)^{-1}  \int_D \frac{\delta U}{\delta m}\big(t, m, x \big) \big({p_h^m(dx) - m(dx)}\big) + o\big( h \big), 
\end{align}
where we used the normalization convention $\int \frac{\delta U}{\delta m}(t,m,x) m(dx) = 0$ to pull out the term $(p_h^m(D))^{-1}$. 
Now, by Lemma \ref{lem.ptm} we have
\begin{align} \label{timderiv3}
    p_h^m(D) = 1 - O(h) \implies \big(p_h^m(D)\big)^{-1} = 1 + O(h), 
\end{align}
Meanwhile, using the weak formulation of the equation satisfied by $p^m$, together with Lemma \ref{lem.ptm}, we get
\begin{align*}
    \int_D \frac{\delta U}{\delta m}(t,m,x) \big({p_h^m(dx) - m(dx)}\big) &= \frac{1}{2}\int_0^h \int_D \Delta_x \frac{\delta U}{\delta m}(t,m,x) p_s^m(dx) ds.
\end{align*}
We now recall that from Remark \ref{rmk.finitedim.pde}, we have
\begin{align*}
    \frac{\delta U}{\delta m}(t,m,x) = (p_t^m(D))^{-1} u(t,x), 
\end{align*}
where $u = u^{t,m}$ satisfies 
\begin{align*}
    \partial_s u - \frac{1}{2} \Delta u = 0 \text{ in } \R_+ \times D, \quad u\big|_{\R_+ \times \partial D} = 0, \quad u(0,x) = \frac{\delta G}{\delta m}(\ov p_t^m,\cdot). 
\end{align*}
Now, since the domain $D$ is smooth, $u$ is smooth for positive times, so that $\Delta_x \frac{\delta U}{\delta m}(t,m,\cdot)$ is smooth, and so
\begin{align*}
    \int_D \Delta_x \frac{\delta U}{\delta m}(t,m,x) p_s^m(dx) ds \xrightarrow{ s \to 0} \int_D \Delta_x \frac{\delta U}{\delta m}(t,m,x) m(dx), 
\end{align*}
by the continuity of $s \mapsto p_s^m$. We deduce that
\begin{align} \label{timederiv4}
    \int_0^h \int_D \Delta_x \frac{\delta U}{\delta m}(t,m,x) p_s^m(dx) ds = h\int_D \Delta_x \frac{\delta U}{\delta m}(t,m,x) m(dx) + o(h). 
\end{align}
  Plugging \eqref{timderiv3} and \eqref{timederiv4} into \eqref{timederivcomp}, we have 
   \begin{align*}
       U(t+h,m) - U(t,m) = U\big(t, \ov p_h^m\big) - U(t,m) = \frac{h}{2} \int_D \Delta_x \frac{\delta U}{\delta m}(t,m,x) m(dx) + o(h).
   \end{align*}
   We have now shown that for each $t > 0$ and $m \in \cP^c(D)$, $\partial_t U$ exists and satisfies the formula \eqref{dtu.formula}. From here, it is straightforward to check from the formula for $\frac{\delta U}{\delta m}$ and the continuity of $(t,m) \mapsto \ov p_t^m$ (from Lemma \ref{lem.ptm}) that $\partial_t U$ is continuous on $\R_+ \times \cP^c(D)$. Combined with Proposition \ref{prop.Uprops}, this completes the proof. 
\end{proof}

\subsection{The properties of $U^N$}

In Propositions \ref{prop.Uprops}, \ref{prop.Uderiv.est.shorttime}, and \ref{prop.uderivs.largetime} we proved various properties of the function $U$. Now we would like to translate these into properties of the functions 
\begin{align*}
    U^N : [0,\infty) \times \big( \ov{D}^N \setminus \partial^N(D^N)\big) \to \R, \quad U^N(t,\bx) = U\big(t,m_{\bx}^N\big).
\end{align*}

The main result of this section is as follows.

\begin{prop} \label{prop.UN.properties}
    The function $U^N$ is continuous on the set
    \begin{align}
       \Big( \R_+ \times \big( \ov{D}^N \setminus \partial^N(D^N) \big) \Big) \cup \Big( [0,\infty) \times D^N \Big). 
    \end{align}
    Moreover, it is $C^{1,2}$ on $\R_+ \times D^N$, and satisfies the estimates 
    \begin{align} \label{UN.eqnerror}
        \Big| \partial_t U^N(t,\bx) - \frac{1}{2} \sum_{i = 1}^N \Delta_{x^i} U^N(t,\bx) \Big| \lesssim e^{-ct} \frac{1}{N} \Big( \frac{1}{N} \sum_{i = 1}^N \rho(x^i) \Big)^{-2} \Big( 1 + \frac{1}{N} \sum_{i = 1}^N w^2(t,x^i) 1_{t \leq 1} \Big)
    \end{align}
    for all $(t,\bx) \in \R_+ \times D^N$, and 
    \begin{align} \label{UN.boundaryerror}
        \Big| U^N(t,\bx) - \frac{1}{N-1} \sum_{j \neq i} U^N\big(t,\bx^{i,j}\big) \Big| \lesssim  e^{-ct}\frac{1}{N^2} \Big( \frac{1}{N} \sum_{j = 1}^N \rho(x^j) \Big)^{-2}
    \end{align}
    for $t \in \R_+$ and $\bx \in \partial^{1,i}(D^N)$, with $c = \lambda_2 - \lambda_1 > 0$. 
\end{prop}

\begin{proof}
    The continuity properties of $U^N$ follow from Proposition \ref{prop.Uprops}. Next, we recall that the $\cC^{1,2}$ smoothness of $U$ on $\R_+ \times \cP^c(D)$ implies the $C^{1,2}$ smoothness of $U^N$ on $\R_+ \times D^N$, with the explicit relationships
    \begin{align*}
        & D_{x^i} U^N(t,\bx) = \frac{1}{N} D_m U\big(t,m_{\bx}^N,x^i\big) \\ & D_{x^jx^i} U^N(t,\bx) = \frac{1}{N} D_x D_m U \big(t,m_{\bx}^N,x^i\big)1_{i = j} + \frac{1}{N^2} D_{mm} U\big(t,m_{\bx}^N,x^i,x^j\big)
    \end{align*}
    between the derivatives of $U^N$ and the derivatives of $U$; see e.g. \cite[Proposition 5.91]{CarmonaDelarue_book_I}.
    Moreover, it follows from these formulas and the equation \eqref{eqn.U} that on $\R_+ \times D^N$, $U^N$ satisfies
    \begin{align*}
        \partial_t U^N(t,\bx) - \frac{1}{2} \sum_{i = 1}^N \Delta_{x^i} U^N(t,\bx) =  - \frac{1}{2N^2} \sum_{i = 1}^N \text{tr}\big(  D_{mm} U(t,m_{\bx}^N,x^i,x^i)\big), 
    \end{align*}
    and so the bound \eqref{UN.eqnerror} follows from Propositions \ref{prop.uderivs.largetime} and \ref{prop.Uderiv.est.shorttime}. 

    It remains to verify the estimate \eqref{UN.boundaryerror}. For this, we assume that $x^i \in \partial D$, $\bx \notin \partial^N(D^N)$. We use the $\cC^{2}$ regularity of $U(t,\cdot)$ on $\cP^+(\ov D)$ to compute
    \begin{align*}
        U^N(t, \bx) &- \frac{1}{N-1} \sum_{j \neq i} U^N\big(t,\bx^{i,j} \big) = \frac{1}{N-1} \sum_{j \neq i} \Big( U\big(t,m_{\bx}^N\big) - U\big(t,m_{\bx^{i,j}}^N) \Big)
        \\
        &= \frac{1}{N-1} \sum_{j \neq i} \int_0^1 \frac{\delta U}{\delta m}\Big(t, \big[m_{\bx^{i,j}}^N, m_{\bx}^N\big]_s, y \Big) \big(m_{\bx}^N - m_{\bx^{i,j}}^N\big)(dy) ds
        \\
        &= \frac{1}{N-1} \sum_{j \neq i} \int_0^1 \frac{\delta U}{\delta m}\Big(t, \big[m_{\bx^{i,j}}^N, m_{\bx}^N\big]_s, y \Big) \frac{1}{N}\big( \delta_{x^i} - \delta_{x^j}\big)(dy) ds 
        \\ 
        &= - \frac{1}{N(N-1)} \sum_{j \neq i} \int_0^1 \frac{\delta U}{\delta m} \Big(t, \big[ m_{\bx^{i,j}}^N,m_{\bx}^N\big]_s, x^j \Big) ds,
    \end{align*}
   where in the last line we used the fact that $x^i \in \partial D$ and $\frac{\delta U}{\delta m}(t,m,\cdot)$ vanishes on $\partial D$ for each $t > 0$. Now, we add and subtract to write 
    \begin{align*}
        \frac{1}{N(N-1)} & \sum_{j \neq i} \int_0^1 \frac{\delta U}{\delta m} \Big(t, \big[ m_{\bx^{i,j}}^N,m_{\bx}^N\big]_s, x^j \Big) ds =   \frac{1}{N(N-1)} \sum_{j \neq i} \frac{\delta U}{\delta m} \big(t, m_{\bx}^N, x^j \big) 
        \\
        &\qquad  +   \frac{1}{N(N-1)} \sum_{j \neq i} \int_0^1 \Big[ \frac{\delta U}{\delta m} \Big(t, \big[ m_{\bx^{i,j}}^N,m_{\bx}^N\big]_s, x^j \Big) -  \frac{\delta U}{\delta m} \Big(t, m_{\bx}^N, x^j \Big) \Big] ds
        \\
        &= \frac{1}{N(N-1)} \sum_{j \neq i} \int_0^1 \Big[ \frac{\delta U}{\delta m} \Big(t, \big[ m_{\bx^{i,j}}^N,m_{\bx}^N\big]_s, x^j \Big) -  \frac{\delta U}{\delta m} \big(t, m_{\bx}^N, x^j \big) \Big] ds, 
    \end{align*}
    where the last equality used the fact that 
    \begin{align*}
        \frac{1}{N-1} \sum_{j \neq i} \frac{\delta U}{\delta m}\big(t,m_{\bx}^N,x^j\big) = \frac{N}{N-1} \int_D \frac{\delta U}{\delta m}\big(t,m_{\bx}^N,y\big) m_{\bx}^N(dy)=0, 
    \end{align*}
    thanks to the fact that $\frac{\delta U}{\delta m}(t,m,x) = 0$ for $x \in \partial D$ and the normalization $\int \frac{\delta U}{\delta m}(t,m,x) m(dx) = 0$. Finally, we perform another expansion to get 
    \begin{align*}
        \frac{1}{N(N-1)} &\sum_{j \neq i} \int_0^1 \Big[ \frac{\delta U}{\delta m} \Big(t, \big[ m_{\bx^{i,j}}^N,m_{\bx}^N\big]_s, x^j \Big) -  \frac{\delta U}{\delta m} \big(t, m_{\bx}^N, x^j \big) \Big] ds
        \\
        &= \frac{1}{N(N-1)} \sum_{j \neq i} \int_0^1 \int_0^1 \frac{\delta^2 U}{\delta m^2} \Big(t, \Big[ m_{\bx}^N, \big[ m_{\bx^{i,j}}^N, m_{\bx}^N \big]_s \Big]_r, x^j, y \Big) d\big(\big[m_{\bx^{i,j}}^N, m_{\bx}^N \big]_s  - m_{\bx}^N \Big)(y) dr
        \\
        &= \frac{1}{N^2(N-1)} \sum_{j \neq i} \int_0^1 \int_0^1 (1-s)\frac{\delta^2 U}{\delta m^2} \Big(t, \Big[ m_{\bx}^N, \big[ m_{\bx^{i,j}}^N, m_{\bx}^N \big]_s \Big]_r, x^j, y \Big) d\big(\delta_{x^j} -  \delta_{x^i} \big)(y) dr, 
    \end{align*}
    and so combining the preceding computations gives 
    \begin{align*}
        \Big| U^N(t,\bx) - \frac{1}{N-1} \sum_{j \neq i} U^N(t,\bx^{i,j}) \Big| 
        \leq \frac{1}{N^2(N-1)} \sum_{i \neq j} \int_0^1 \int_0^1 \Big\| \frac{\delta^2 U}{\delta m^2} \Big(t, \Big[ m_{\bx}^N, \big[ m_{\bx^{i,j}}^N, m_{\bx}^N \big]_s \Big]_r, \cdot, \cdot \Big) \Big\|_{\infty} dr ds. 
    \end{align*}
    Next, noticing that $x^i \in \partial D$ implies that
    \begin{align*}
        \int \rho d\Big[ m_{\bx}^N, \big[ m_{\bx^{i,j}}^N, m_{\bx}^N \big]_s \Big]_r(y) \geq \int \rho dm_{\bx}^N, \quad \text{ for all } 0 \leq r,s \leq 1, 
    \end{align*}
    we see that by applying the upper bound on $\frac{\delta^2 U}{\delta m^2}$ from Proposition \ref{prop.uderivs.largetime}, we obtain
    \begin{align*}
          \Big| U^N(t,\bx) - \frac{1}{N-1} \sum_{j \neq i} U^N(t,\bx^{i,j}) \Big| \lesssim e^{-ct} \frac{1}{N^2} \Big(\frac{1}{N} \sum_k \rho(x^k) \Big)^{-2}, 
    \end{align*}
    as desired. 
\end{proof}

\section{Convergence}

\label{sec.conv}

This Section contains the proof of our main convergence result, Theorem \ref{thm.conv.finitetime}. The Section is split into three parts; first, we estimate $U^N - V^N$ on a unit time interval, then we obtain a uniform-in-time estimate on $U^N - V^N$, and finally we infer an estimate on $U - \int_{D^N} V^N dm^{\otimes N}$.

\subsection{Pointwise convergence on a unit time interval} \label{subsec.unittime}

The goal of this subsection is to prove the following (recall that $w(t,x) = \gamma\big(t,\dist(x) \big) + Ct$, $C>0$, and $\gamma$ is defined in \eqref{def.gamma}).

\begin{prop} \label{prop.unittime}
  For any $\theta > 0$, we have
    \begin{align}\label{eq.estunit}
        \Big| U\big(t,m_{\bx}^N\big) - V^N(t,\bx) \Big| \lesssim_{\theta} N^{-1}  \bigg[ \Big( \frac{1}{N} \sum_{i = 1}^N \rho(x^i) \Big)^{-(6 + \theta)} + \Big(\frac{1}{N} \sum_{i = 1}^N w(t,x^i)\Big)\bigg] 
    \end{align}
    for all $t \in (0,1]$, $\bx \in \ov D^N \setminus \partial^N(D^N)$. 
\end{prop}

In order to prove Proposition \ref{prop.unittime}, we are going to build appropriate penalizations. In particular, we will work with two functions
\begin{align*}
   \Phi_1^N, \, \Phi_2^N : \text{Dom}_N \to \R_+ 
\end{align*}
which will be used in the proof of convergence. These two functions, divided by $N$, behave as the two terms appearing on the right-hand side of \eqref{eq.estunit}. The strategy will be to show that $U-V^N - C(\Phi_1^N + \Phi_2^N)/N$ is a subsolution of \eqref{eqn.VN.subsol} for some large $C$, so that the comparison principle applies. Unfortunately, we cannot directly use Proposition \ref{prop.comparison}, since the subsolution property is false on the whole parabolic domain. In fact, we show that it holds on the set of interest $\{U-V^N - C(\Phi_1^N + \Phi_2^N)/N > 0\}$, by exploiting the boundedness of $U, V^N$.

Let $\theta >0$. We begin with $\Phi_1^N$, which we will define by
\begin{align} \label{def.Phi1}
    \Phi_1^N(t,\bx) = 
    \exp\big( (6 + 2\theta) \lambda_1 t \big) \Big(\frac{1}{N} \sum_{i = 1}^N \rho(x^i) \Big)^{-(6 + \theta)}.
\end{align}

The following lemma summarizes the key properties of $\Phi_1^N$. 

\begin{lem} \label{lem.Phi1N}
    For $(t,\bx) \in [0,\infty) \times D^N$, we have 
    \begin{align} \label{phi1.supersol}
        \partial_t \Phi_1^N - \frac{1}{2} \sum_{i = 1}^N \Delta_{x^i} \Phi_1^N &\geq \theta \lambda_1 \exp\big( (6 + 2\theta) \lambda_1 t\big) \Big(\frac{1}{N} \sum_{i = 1}^N \rho(x^i) \Big)^{-(6 + \theta)}
        \nonumber \\
        &\quad - \frac{\|D\rho\|_{\infty}^2 (6 + \theta)(7 + \theta)}{2N} \exp\big( (6 + 2\theta) \lambda_1 t \big) \Big(\frac{1}{N} \sum_{i = 1}^N \rho(x^i) \Big)^{-(8 + \theta)} 
    \end{align}
    In addition, for $t \in \R_+$ and $\bx \in \big(\ov D\big)^N \setminus \partial^N \big(D^N\big)$ with $x^i \in \partial D$, we have
    \begin{align} \label{phi1.boundary}
    \Phi_1^N(t,\bx) \geq \frac{1}{N-1} \sum_{j \neq i} \Phi_1^N\big(t,\bx^{i,j}\big) + \frac{(6 + \theta)}{2^{(7 + \theta)} (N-1)} \exp\big( (6 + 2\theta) \lambda_1 t \big)  \Big(\frac{1}{N} \sum_{j = 1}^N \rho(x^j) \Big)^{-(6 + \theta)}.
\end{align}
\end{lem}

\begin{proof}
    We start with \eqref{phi1.supersol}. We compute 
    \begin{align*}
        \partial_t \Phi_1^N(t,\bx) &= (6 + 2\theta) \lambda_1 \exp\big((6 + 2\theta) \lambda_1 t\big) \Big( \frac{1}{N} \sum_{i = 1}^N \rho(x^i) \Big)^{-(6 + \theta)}, 
        \\
         D_{x^i} \Phi_1^N(t,\bx) &= - \frac{(6 + \theta)}{N} \exp\big( (6 + 2\theta) \lambda_1t \big) \Big(\frac{1}{N} \sum_{j = 1}^N \rho(x^j) \Big)^{-(7 + \theta)} D\rho(x^i), 
        \\
       \Delta_{x^i} \Phi_1^N(t,\bx) &= - \frac{(6 + \theta)}{N} \exp\big( (6 + 2\theta) \lambda_1 t \big) \Big(\frac{1}{N} \sum_{j = 1}^N \rho(x^j) \Big)^{-(7 + \theta)} \Delta \rho(x^i) 
        \\
        &\qquad +\frac{(6 + \theta)(7 + \theta)}{N^2}\exp\big((6  + 2\theta) \lambda_1 t \big) \Big(\frac{1}{N} \sum_{j = 1}^N \rho(x^j) \Big)^{-(8 + \theta)} |D \rho(x^i)|^2
        \\
        &=  \frac{2(6 + \theta) \lambda_1}{N} \exp\big( (6 + 2\theta) \lambda_1 t \big) \Big(\frac{1}{N} \sum_{j = 1}^N \rho(x^j) \Big)^{-(7 + \theta)} \rho(x^i) 
        \\
        &\qquad +\frac{(6 + \theta)(7 + \theta)}{N^2}\exp\big((6  + 2\theta) \lambda_1 t \big) \Big(\frac{1}{N} \sum_{j = 1}^N \rho(x^j) \Big)^{-(8 + \theta)} |D \rho(x^i)|^2.
    \end{align*}
   Summing over $i$, we get
    \begin{align*}
        &\partial_t \Phi_1^N - \frac{1}{2} \sum_{i = 1}^N \Delta_{x^i} \Phi^N_1  = \theta \lambda_1 \exp\big( (6 + 2\theta) \lambda_1 t\big) \Big(\frac{1}{N} \sum_{i = 1}^N \rho(x^i) \Big)^{-(6 + \theta)} 
        \\
        &\qquad \qquad - \frac{(6 + \theta)(7 + \theta)}{2N} \exp\big( (6 + 2\theta) \lambda_1 t \big) \Big(\frac{1}{N} \sum_{i = 1}^N \rho(x^i) \Big)^{-(8 + \theta)} \frac{1}{N} \sum_{i = 1}^N |D\rho(x^i)|^2
    \end{align*}
    and \eqref{phi1.supersol} follows. 

    Now we suppose that $\bx \in \ov{D}^N \setminus \partial^N(D^N)$ and that $x^i \in \partial D$, and we aim to prove the estimate \eqref{phi1.boundary}. Suppose first that $\bm{z} = (z^1,...,z^N) \in [0,\infty)^N$, and that $z^i = 0$. Then, for each $j \neq i$, we use the convexity of $w \mapsto w^{-(6 + \theta)}$ to get
    \begin{align*}
        \big( \frac{1}{N} \sum_{k = 1}^N z^k \big)^{-(6 + \theta)}  &= \big(\frac{1}{N} \sum_{k \neq i} z^k \big)^{-(6 + \theta)} 
        \\
        &\geq \big( \frac{1}{N} \sum_{k \neq i} z^k + \frac{1}{N} z^j \big)^{-(6 + \theta)}
        + (6 + \theta) \big(\frac{1}{N} \sum_{k \neq i} z^k + \frac{1}{N} z^j \big)^{-(7 + \theta)} \frac{1}{N} z^j.
    \end{align*}
    Averaging over $j \neq i$, we obtain 
    \begin{align*}
        \big(\frac{1}{N} \sum_{k = 1}^N z^k\big)^{-(6 + \theta)} &\geq \frac{1}{N-1} \sum_{j \neq i} \big(\frac{1}{N} \sum_{k \neq i} z^k + \frac{1}{N} z^j \big)^{-(6 + \theta)}
        \\
        &\quad + \frac{(6 + \theta)}{N-1} \sum_{j \neq i} \big(\frac{1}{N} \sum_{k \neq i} z^k + \frac{1}{N} z^j \big)^{-(7 + \theta)} \frac{1}{N} z^j.
    \end{align*}
    Since 
    \begin{align*}
        \big(\frac{1}{N} \sum_{k \neq i} z^k + \frac{1}{N} z^j\big)^{-(7 + \theta)} \geq \big( \frac{2}{N} \sum_{k = 1}^N z^k \big)^{-(7 + \theta)} = 2^{-(7 + \theta)} \big(\sum_{k = 1}^N z^k \big)^{-(7 + \theta)}, 
    \end{align*}
    this yields
    \begin{align*}
        \big(\frac{1}{N} \sum_{k = 1}^N z^k\big)^{-(6 + \theta)} \geq \frac{1}{N-1} \sum_{j \neq i} \big(\frac{1}{N} \sum_{k \neq i} z^k + \frac{1}{N} z^j \big)^{-(6 + \theta)} +  \frac{(6 + \theta)}{2^{(7 + \theta)} (N-1)}\big( \frac{1}{N}\sum_{k = 1}^N z^k\big)^{-(6 + \theta)}.
    \end{align*}
    Applying this result with $z^k = \rho(x^k)$ gives \eqref{phi1.boundary}.
\end{proof}

To define $\Phi_2^N$, we are going to start by introducing the parameter $\beta > 0$, which is determined from the parameter $\theta \in (0,1)$ via 
\begin{align*}
    3 - 2 \beta = 2\Big(\frac{6 + \theta}{4+ \theta}\Big). 
\end{align*}
The choice of $\beta$ is made so that 
\begin{align*}
  p =  \frac{3 - 2 \beta}{2}, \quad q = \frac{6 + \theta}{2} \quad \text{  are conjugate exponents, i.e. } \quad  \frac{1}{p} + \frac{1}{q} = 1.
\end{align*}
We then define 
\begin{align} \label{def.Phi2}
    \Phi_2^N(t,\bx) = \begin{cases} \ds \frac{t^{\beta}}{N} \sum_{i = 1}^N w(t,x^i) & t > 0,
    \\
    0 & t = 0
    \end{cases},
\end{align}
where $w$ is as defined in \eqref{def.w}. We will need the following fact about $w$.

\begin{lem} \label{lem.w}
    The function $w$ is continuous on $\R_+ \times \ov D$, and $C^{1,2}$ on $\R_+ \times D$. Moreover, the following properties hold:
    \begin{enumerate}        \item If $C$ is large enough, then
        \begin{align} \label{w.supersol}
        \partial_t w - \frac{1}{2} \Delta w \geq 0 \text{ on } \R_+ \times D
    \end{align}
    holds in a classical sense.
        \item For each $t > 0$, $w(t,\cdot)$ is constant on $\partial D$ and achieves its global maximum there.
    \end{enumerate}
\end{lem}

\begin{proof}
  Recalling that $w(t,x) = \gamma\big(t,\dist(x)\big) + Ct$, we compute 
  \begin{align*}
      \partial_t w(t,x) - \frac{1}{2} \Delta w(t,x) = \partial_t \gamma\big(t,\dist(x)\big) - \frac{1}{2} \partial_{zz} \gamma\big(t,\dist(x)\big) |D\dist(x)|^2 - \frac{1}{2} \partial_z \gamma\big(t,\dist(x)\big) \Delta \dist(x) + C.
  \end{align*}
  For some $\eps > 0$, $|D\dist(x)|^2 = 1$ in $(\partial D)_{\eps}$, and so using the equation for $\gamma$,
  \begin{align*}
      \partial_t w - \frac{1}{2} \Delta w &= \partial_t \gamma\big(t,\dist(x)\big) - \frac{1}{2} \partial_{zz} \gamma \big(t,\dist(x)\big) - \frac{1}{2} \partial_z \gamma\big(t,\dist(x)\big) \Delta \dist(x) 
      \\
      &= - \|\Delta \dist\|_{\infty} \partial_z \gamma\big(t,\dist(x)\big) - \frac{1}{2} \partial_z \gamma\big(t,\dist(x)\big) \Delta \dist(x) 
      \\
      &= \|\Delta \dist\|_{\infty} |\partial_z \gamma\big(t,\dist(x)\big)|  - \frac{1}{2} \partial_z \gamma\big(t,\dist(x)\big) \Delta \dist(x) \geq 0 \text{ on } (\partial D)_{\eps},
  \end{align*}
  where we used the fact that $\partial_z \gamma < 0$ for $z > 0$, as is checked from the formula \eqref{def.gamma}. On the other hand, the derivatives of $\gamma$ are bounded on on $(0,1] \times (\eps, \infty)$, so if we choose $C$ large enough, we find that \eqref{w.supersol} also holds on $(0,1] \times \Big( D \setminus \big(\partial D\big)_{\eps} \Big)$.

  Property (2) follows from the fact that $\gamma(t,\cdot)$ is maximized at $0$. 
\end{proof}

The key properties of $\Phi_2^N$ are as follows. 
\begin{lem} \label{lem.Phi2N}
    The function $\Phi_2^N$ defined by \eqref{def.Phi2} has the following properties:
    \begin{enumerate}
        \item $\Phi_2^N$ is continuous on $\R_+ \times \ov{D}^N$, lower semi-continuous on $\text{Dom}_N$, and smooth on $\R_+ \times D^N$.
        \item We have 
        \begin{align} \label{Phi2N.supersol}
        \partial_t \Phi_2^N(t,\bx) - \frac{1}{2} \sum_{i = 1}^N \Delta_{x^i}\Phi_2^N(t,\bx) \gtrsim_{\beta} \frac{1}{N} \sum_{i = 1}^N w^{3 - 2 \beta} (t,x^i) \text{ on } (0,1) \times D^N, 
        \end{align}
        \item for $t \in [0,1]$ and $\bx \in \ov D^N$ with $x^i \in \partial D$, we have
   \begin{align} \label{Phi2N.boundary}
       \Phi_2^N(t,\bx) \geq \frac{1}{N-1} \sum_{j \neq i} \Phi_2^N\big(t,\bx^{i,j}\big).
    \end{align}
    \end{enumerate}
\end{lem}

\begin{proof}
Property (1) follows directly from the continuity of $w$.  

To verify Property (2), we compute 
\begin{align*}
   \Big( \partial_t - \frac{1}{2} \sum_{i = 1}^N \Delta_{x^i} \Big) \Phi_2^N &= \beta t^{\beta-1} \frac{1}{N} \sum_{i = 1}^N w(t,x^i) + \frac{t^{\beta}}{N} \sum_{i = 1}^N \Big(\partial_t w^i(t,x^i) - \frac{1}{2} \Delta w^i(t,x^i) \Big) 
   \\
   &\gtrsim \beta \| w(t,\cdot)\|_{\infty}^{2 - 2\beta} \frac{1}{N} \sum_{i = 1}^N w(t,x^i)
   \\
   &\gtrsim_{\beta} \frac{1}{N} \sum_{i = 1}^N w^{3-2\beta}(t,x^i), 
\end{align*}
where we used the fact that $w$ is a supersolution of the heat equation on $D$ (Lemma \ref{lem.w}) and the fact that $\| w(t,\cdot)\|_{\infty} = \gamma(t,0) + Ct = \frac{1}{\sqrt{2 \pi t}} + O(1)$ as $t \to 0$. 
    
    Finally, property (3) follows from the fact that $w(t,\cdot)$ is maximized on $\partial D$, as discussed in Lemma \ref{lem.w}.
\end{proof}

\begin{proof}[Proof of Proposition \ref{prop.unittime}]
For $\eps_1,\eps_2, \delta, \eta \in (0,1)$, consider the optimization problem 
\begin{align} \label{def.Meps}
    M_{\eps_1,\eps_2, \delta, \eta} &= \sup_{(t,\bx) \in \text{Dom}_N \cap \big([0,1) \times \ov{D}^N]\big)} \Big\{ U^N(t,\bx) - V^N(t,\bx) - \eps_1 \Phi_1^N(t,\bx) - \eps_2 \Phi_2^N(t,\bx) 
   \nonumber \\
    &\qquad \qquad \qquad \qquad \qquad \qquad \qquad \qquad \qquad   - \delta \Big( \Xi_1^N(t,\bx) + \Xi_{2,\eta}^N(t,\bx) + \frac{ 1 }{1-t}\Big) \Big\},
\end{align}
where $\Xi_{2,\eta}^N(t,\bx) = \Xi_2^N(t+ \eta, \bx)$, and $\Xi_1^N$, $\Xi_2^N$ are as defined in \eqref{def.Xi1N} and \eqref{def.Xi2N}. 

Arguing as in the proof of Proposition \ref{prop.comparison}, one can check that for each fixed $\eps_1,\eps_2,\delta$, and for any $\eta > 0$ small enough (depending on $\eps_1,\eps_2,\delta$), the problem admits at least one optimizer $(\hat t, \hat \bx)$ in $\text{Dom}_N \cap \Big([0,1) \times \ov D^N\Big)$. Throughout the following, we assume that $\eta$ is small enough and fix such an optimizer. 
   \newline \newline 
   \noindent \textit{Step 1: Avoiding the lateral boundary.} Now, we will argue that if $\eps_1$ is not too small, then in fact $\hat \bx \in D^N$. Suppose that for some $i \in \{1,...,N\}$, $\hat{x}^i \in \partial D$. Then, by the optimality of $(\hat{t},\hat{\bx})$, we know that for each $j \neq i$, we have 
   \begin{align*}
       &U^N(\hat t, \hat \bx) - V^N(\hat t,\hat \bx) - \eps_1 \Phi^N_1(\hat t,\hat \bx) - \eps_2 \Phi^N_2(\hat t,\hat \bx) - \delta \Big( \Xi_1^N(\hat t,\hat \bx) + \Xi_{2,\eta}^N(\hat t,\hat \bx)\Big)
       \\
       &\quad \geq U^N(\hat t, \hat \bx^{i,j}) - V^N(\hat t,\hat \bx^{i,j}) - \eps_1 \Phi^N_1(\hat t,\hat \bx^{i,j}) - \eps_2 \Phi^N_2(\hat t,\hat \bx^{i,j}) - \delta \Big( \Xi_1^N(\hat t,\hat \bx^{i,j}) + \Xi_{2,\eta}^N(\hat t,\hat \bx^{i,j})\Big) , 
   \end{align*}
   and hence 
   \begin{align*}
       &U^N(\hat t, \hat \bx) - V^N(\hat t,\hat \bx) - \eps_1 \Phi^N_1(\hat t,\hat \bx) - \eps_2 \Phi^N_2(\hat t,\hat \bx) - \delta \Big( \Xi_1^N(\hat t,\hat \bx) + \Xi_{2,\eta}^N(\hat t,\hat \bx)\Big) 
       \\
       &\quad \geq \frac{1}{N-1} \sum_{j \neq i} \Big\{ U^N(\hat t, \hat \bx^{i,j}) - V^N(\hat t,\hat \bx^{i,j}) - \eps_1 \Phi^N_1(\hat t,\hat \bx^{i,j}) - \eps_2 \Phi^N_2(\hat t,\hat \bx^{i,j}) - \delta \Big( \Xi_1^N(\hat t,\hat \bx^{i,j}) + \Xi_{2,\eta}^N(\hat t,\hat \bx^{i,j})\Big) \Big\}.
   \end{align*}
   Rearranging, and using the Fleming-Viot boundary condition for $V^N$, we find that 
   \begin{align} \label{boundary.cont0}
       0 &= V^N(\hat t, \hat \bx) - \frac{1}{N-1} \sum_{j \neq i} V^N(\hat{t}, \hat{\bx}^{i,j})
     \nonumber  \\
       &\leq U^N(\hat t, \hat \bx) - \frac{1}{N-1} \sum_{j \neq i} U^N(\hat{t}, \hat{\bx}^{i,j}) - \sum_{m = 1}^2 \eps_m \Big( \Phi_m^N(\hat t, \hat \bx) - \frac{1}{N-1} \sum_{j \neq i} \Phi_m^N(\hat{t}, \hat{\bx}^{i,j}) \Big)
       \\
    \nonumber    &\qquad \qquad - \delta \Big( \Xi_1^N(\hat t, \hat \bx) - \frac{1}{N-1} \sum_{j \neq i} \Xi_1^N (\hat{t}, \hat{\bx}^{i,j}) + \Xi_{2,\eta}^N(\hat t, \hat \bx) - \frac{1}{N-1} \sum_{j \neq i} \Xi_{2,\eta}^N (\hat{t}, \hat{\bx}^{i,j}) \Big)
   \end{align}
   By Lemmas \ref{lem.Phi1N}, \ref{lem.Phi2N}, \ref{lem.PhiN.comp} and \ref{lem.PsiN.comp} we have 
   \begin{align*}
        &\Phi_1^N(\hat t, \hat \bx) - \frac{1}{N-1} \sum_{j \neq i} \Phi_1^N(\hat{t}, \hat{\bx}^{i,j}) \geq   \frac{(6 + \theta)}{2^{(7 + \theta)}N} \exp\big((6 + 2\theta) \lambda_1 \hat t \big) \Big(\frac{1}{N} \sum_{j = 1}^N \rho(\hat x^j) \Big)^{-(6 + \theta)},
        \\
        & \Phi_2^N(\hat t, \hat \bx) - \frac{1}{N-1} \sum_{j \neq i} \Phi_2^N(\hat{t}, \hat{\bx}^{i,j}) \geq 0, 
        \\
        &\Xi_{1}^N(\hat t, \hat \bx) - \frac{1}{N-1} \sum_{j \neq i} \Xi_{1}^N(\hat{t}, \hat{\bx}^{i,j}) \geq 0,
        \\
        &\Xi_{2,\theta}^N(\hat t, \hat \bx) - \frac{1}{N-1} \sum_{j \neq i} \Xi_{2,\theta}^N(\hat{t}, \hat{\bx}^{i,j}) \geq 0.
   \end{align*}
   Thus, \eqref{boundary.cont0} implies that
   \begin{align*}
      \frac{\eps_1}{N} &\Big(\frac{1}{N} \sum_{j = 1}^N \rho(\hat x^j) \Big)^{-2} \lesssim \frac{\eps_1 (6+ \theta) }{N 2^{(7 + \theta)}} \exp\big((6 + 2\theta) \lambda_1 \hat t \big) \Big(\frac{1}{N} \sum_{j = 1}^N \rho(\hat x^j) \Big)^{-(6 + \theta)} 
      \\
      &\qquad \qquad \leq U^N(\hat t, \hat \bx) - \frac{1}{N-1} \sum_{j \neq i} U^N(\hat{t}, \hat{\bx}^{i,j})
      \lesssim \frac{1}{N^2} \Big(\frac{1}{N} \sum_{j = 1}^N \rho(\hat x^j) \Big)^{-2}, 
   \end{align*}
   where the last line uses Proposition \ref{prop.UN.properties}. 
   We reach a contradiction if $\eps_1 \geq \frac{C}{N}$ for some universal constant $C$. In other words, we have found a constant $C$ such that:
   \begin{align} \label{boundary.implication}
   \text{ if $\eps_1 \geq C/N$, and $(\hat t, \hat \bx)$ is any optimizer for \eqref{def.Meps} with $\hat t > 0$, then $\hat \bx \in D^N$}
   \end{align}
   \noindent 
    \textit{Step 2: Using the equation.} Suppose now that we have a maximum $(\hat t, \hat \bx) \in (0,1) \times  D^N$. From the first and second order conditions for optimality, we deduce that
    \begin{align*}
        \partial_t U^N(\hat t, \hat \bx) &\geq \partial_t V^N(\hat t, \hat \bx) + \sum_{m = 1}^2 \eps_m \partial_t \Phi_m^N(\hat t, \hat \bx) +  \delta \Big( \partial_t \Xi_1^N(\hat t, \hat \bx) + \partial_t \Xi_{2,\eta}^N(\hat t, \hat \bx) \Big), 
        \\
        \Delta_{x^i} U^N(\hat t, \hat \bx) &\leq \Delta_{x^i} V^N(\hat t, \hat \bx) + \sum_{m = 1}^2 \eps_m \Delta_{x^i}  \Phi_m^N(\hat t, \hat \bx) +  \delta \Big( \Delta_{x^i} \Xi_1^N(\hat t, \hat \bx) + \Delta_{x^i}  \Xi_{2,\eta}^N(\hat t, \hat \bx) \Big)
    \end{align*}
    and thus 
    \begin{align*}
        0 &= \partial_t V^N(\hat t, \hat \bx) - \frac{1}{2} \sum_{i = 1}^N \Delta_{x^i} V^N(\hat t, \hat \bx) 
        \\
        &\leq \partial_t U^N(\hat t, \hat \bx) - \frac{1}{2} \sum_{i = 1}^N \Delta_{x^i} U^N(\hat t, \hat \bx) - \sum_{m = 1}^2 \eps_m \Big( \partial_t \Phi_m^N(\hat t, \hat \bx) - \frac{1}{2} \sum_{i = 1}^N \Delta_{x^i} \Phi_m^N(\hat t, \hat \bx) \Big) 
        \\
        &\qquad \qquad - \delta \Big( \partial_t \Xi_1^N(\hat t, \hat \bx) - \frac{1}{2} \sum_{i = 1}^N \Delta_{x^i} \Xi_1^N(\hat t, \hat \bx) + \partial_t \Xi_2^N(\hat t + \eta, \hat \bx) - \frac{1}{2} \sum_{i = 1}^N \Delta_{x^i} \Xi_2^N(\hat t + \eta, \hat \bx) \Big).
    \end{align*}
    Applying Proposition \ref{prop.UN.properties} and Lemmas \ref{lem.Phi1N}, \ref{lem.Phi2N}, \ref{lem.PsiN.comp} and \ref{lem.PhiN.comp}, we find that there is a constant $C$ such that
    \begin{align} \label{contradiction.comp}
      \eps_1 \Big(\frac{1}{N}& \sum_{i = 1}^N \rho(\hat x^i) \Big)^{-(6 + \theta)} + \frac{\eps_2}{N} \sum_{i = 1}^N w^{3 - 2 \beta}(\hat t,\hat x^i)
   \nonumber    \\
      &\lesssim_{\theta} \bigg\{ \frac{1}{N} \Big( \frac{1}{N} \sum_{i = 1}^N \rho(\hat x^i) \Big)^{-2} \Big( 1 + \frac{1}{N} \sum_{i = 1}^N w^2(\hat t,\hat x^i) \Big) 
       + \frac{\eps_1}{N} \Big( \frac{1}{N} \sum_{i = 1}^N \rho(\hat x^i) \Big)^{-(8 + \theta)} \bigg\}.
    \end{align}
    Now notice that by the boundedness of $M_{\bm{\eps}}$, we have 
    \begin{align} \label{penaltybound}
        \eps_1 \Big(\frac{1}{N} \sum_{i = 1}^N \rho(\hat x^i) \Big)^{-(6 + \theta)} \lesssim 1.
    \end{align}
    In addition, applying Young's inequality to the first term with conjugate exponents $\frac{6 + \theta}{2}$ and $\frac{3 - 2\beta}{2}$ gives
    \begin{align} \label{youngs.ineq}
        \Big( \frac{1}{N} \sum_{i = 1}^N \rho(x^i) \Big)^{-2} \Big(\frac{1}{N} \sum_{i = 1}^N w^2(t,x^i) \Big) &\leq \Big(\frac{1}{N} \sum_{i = 1}^N \rho(x^i) \Big)^{-(6 + \theta)} + \Big(\frac{1}{N} \sum_{i = 1}^N w^{2}(t,x^i) \Big)^{\frac{3 - 2\beta}{2}}
    \nonumber     \\
        &\leq \Big(\frac{1}{N} \sum_{i = 1}^N \rho(x^i) \Big)^{-(6 + \theta)} + \frac{1}{N} \sum_{i = 1}^N w^{3 - 2 \beta}(t,x^i), 
    \end{align}
    with the last line coming from Jensen's inequality. Combining \eqref{contradiction.comp}, \eqref{penaltybound}, and \eqref{youngs.ineq}, we deduce that
    \begin{align*}
      \eps_1 \Big(\frac{1}{N}& \sum_{i = 1}^N \rho(\hat x^i) \Big)^{- (6 + \theta)} + \frac{\eps_2}{N} \sum_{i = 1}^N w^{3 - 2 \beta}(\hat t,\hat x^i)
      \\
      &\lesssim_{\theta} \frac{1}{N} \bigg\{ \Big(\frac{1}{N} \sum_{i = 1}^N \rho(\hat x^i) \Big)^{-(6 + \theta)} + \frac{1}{N} \sum_{i = 1}^N w^{3 - 2 \beta}(\hat t,\hat x^i)   \bigg\}.
    \end{align*}
    We deduce that there is a constant $C = C(\theta)$ such that if
    \begin{align*}
       \eps_1 \geq C/N, \quad \eps_2 \geq C/N, 
    \end{align*}
    then there is a contradiction and so we cannot have a maximum with $\hat t >0$ and $\hat \bx \in D^N$. 
    \newline \newline 
    \noindent \textit{Step 3: Conclusion.} Combining Steps 1 and 2, we find a $C$ depending on $\theta$ such that if
    \begin{align*}
        \eps_1 = \eps_2 = \frac{C}{N}, 
    \end{align*}
   then for any $\delta > 0$, and for any $\eta > 0$ small enough (depending on $N$ and $\delta$), there is at least one optimizer $(\hat t, \hat \bx)$ for \eqref{def.Meps}, and any such optimizer must satisfy $\hat t = 0$, $\bx \in D^N$. As a consequence, for every $\delta > 0$, and any $\eta$ small enough (depending on $N$ and $\delta$), we have $M_{CN^{-1},CN^{-1},\delta, \eta} \leq 0$, and hence the estimate
    \begin{align*}
        U^N(t,\bx) - V^N(t,\bx) \leq CN^{-1} \Big( \Phi_1^N(t,\bx) + \Phi_2^N(t,\bx) \Big) + \delta \Big( \Xi_1^N(t,\bx) + \Xi_{2,\eta}^N + \frac{1}{1-t} \Big) 
    \end{align*}
    holds. Sending first $\eta$ and then $\delta$ to zero, and using the form of $\Phi_1^N$ and $\Phi_2^N$, we obtain
    \begin{align*}
          U^N(t,\bx) - V^N(t,\bx) &\lesssim_{\theta} N^{-1} \bigg( \Big(\frac{1}{N} \sum_{i = 1}^N \rho(x^i) \Big)^{-(6 + \theta)} +  \frac{t^{\beta}}{N} \sum_{i = 1}^N w(t,x^i) \bigg) \quad (t,\bx) \in (0,1) \times ( D)^N.
    \end{align*}
    The proof of the matching bound on $V^N - U^N$ is almost identical and is omitted. 
    
\end{proof}

\subsection{Uniform in time pointwise convergence} \label{subsec.uniformintime}

In the previous subsection we showed (taking $\theta = 1$ for concreteness) the bound
\begin{align*}
    \Big| U\big(t,m_{\bx}^N\big) - V^N(t,\bx) \Big| \lesssim N^{-1} \Big[ \Big(\frac{1}{N} \sum_{i = 1}^N \rho(x^i) \Big)^{-7} + \Big(\frac{1}{N} \sum_{i = 1}^N w(t,x^i) \Big) \Big].
\end{align*}
The goal of this section is to prove an analogous bound for all $t \geq 1$. The weighting will be different, but the rate $N^{-1}$ will be the same. 

\begin{prop} \label{prop.uniformintime.pointwise}
    We have the estimate 
    \begin{align*}
        \Big| U\big(t,m_{\bx}^N\big) - V^N(t,\bx) \Big| \lesssim N^{-1} \Big[ \frac{1}{N(N-1)...(N-9)} \sum_{\substack{i_1,...,i_{10} = 1,...,N \\ i_1,...,i_{10} \text{distinct}}} \Big(\sum_{n = 1}^{10} \rho(x^{i_n}) \Big)^{-9} \Big], \quad t \geq 1, \quad \bx \in D^N
    \end{align*}
    for all $N \geq 10$.
\end{prop}

We again will work with two penalizations, which we will call $\Psi_1^N$, $\Psi_2^N$. The function $\Psi_1^N$ will be defined by 
\begin{align*}
    \Psi_1^N(t,\bx) = \exp\big( - (\lambda_2 - \lambda_1 ) t \big) \Big(\frac{1}{N} \sum_{i = 1}^N \rho(x^i) \Big)^{-7}
\end{align*}

The key properties of $\Psi_1^N$ are as follows. 

\begin{lem} \label{lem.Psi1N}
 For $(t,\bx) \in [0,\infty) \times D^N$, we have 
 \begin{align} \label{Psi1N.bound1}
     \partial_t \Psi_1^N - \frac{1}{2} \sum_{i = 1}^N \Delta_{x^i} \Psi_1^N \gtrsim - \exp\big( - (\lambda_2 - \lambda_1) t \big) \Big(\frac{1}{N} \sum_{i = 1}^N \rho(x^i) \Big)^{-9}. 
 \end{align}
 In addition, for $t \in \R_+$, $\bx \in \partial^{1,i}(D^N)$, we have 
 \begin{align} \label{Psi1N.bound2}
     \Psi_1^N(t,\bx) \geq \frac{1}{N-1} \sum_{j \neq i} \Psi_1^N\big(t,\bx^{i,j}\big) + \frac{7}{2^{8}(N-1)} \exp\big( - (\lambda_2 - \lambda_1) t \big) \Big(\frac{1}{N} \sum_{i = 1}^N \rho(x^i) \Big)^{-7}.
 \end{align}
\end{lem}

\begin{proof}
By explicit computation,  
\begin{align*}
     \partial_t \Psi_1^N  =  -(\lambda_2 - \lambda_1)  \exp\big( - (\lambda_2 - \lambda_1) t \big) \Big(\frac{1}{N} \sum_{i = 1}^N \rho(x^i) \Big)^{-7} \gtrsim -   \exp\big( - (\lambda_2 - \lambda_1) t \big)\Big(\frac{1}{N} \sum_{i = 1}^N \rho(x^i) \Big)^{-9}, 
\end{align*}
and
\begin{align*}
    - \frac{1}{2} \sum_{i = 1}^N \Delta_{x^i} \Psi_1^N &= - \exp\big(- (\lambda_2 - \lambda_1)t \big) \bigg[  7 \lambda_1 \Big(\frac{1}{N} \sum_{i = 1}^N \rho(x^i) \Big)^{-7} + \frac{28}{N} \Big(\frac{1}{N} \sum_{i = 1}^N \rho(x^i) \Big)^{-9} \Big(\frac{1}{N} \sum_{i = 1}^N |D\rho(x^i)|^2 \Big) \bigg]
    \\
    &\qquad \qquad \gtrsim -  \exp\big( - (\lambda_2 - \lambda_1) t \big)  \Big(\frac{1}{N} \sum_{i = 1}^N \rho(x^i) \Big)^{-9}.
\end{align*}
The bound \eqref{Psi1N.bound1} follows. The bound \eqref{Psi1N.bound2} follows exactly as in the proof of Lemma \ref{lem.Phi1N}, and so is omitted. 
\end{proof}

To define $\Psi_2^N$, we are going to apply Proposition \ref{prop.supersolexists} with parameters 
\begin{align*}
    k = 10, \quad c = \lambda_2 - \lambda_1, \quad M = \text{diam}(D),
\end{align*}
and with $\delta$ small enough that $|D\dist(x)| = 1$ on $\big(\partial D\big)_{\delta}$. Thus, we obtain a function 
\begin{align*}
    \psi : \R_+ \times \R_+^{10} \to \R 
\end{align*}
satisfying properties (1)-(5) of Proposition \ref{prop.supersolexists}. We then define a function 
\begin{align*}
    \Psi : \R_+ \times D^{10} \to \R
\end{align*}
via 
\begin{align*}
    \Psi(t,\bx) = \psi\big(t, \dist(x^1),...,\dist(x^{10}) \big). 
\end{align*}
Finally, for $N \geq 10$, we set
\begin{align*}
    \Psi_2^N(t,\bx) = \frac{1}{N(N-1)...(N-9)} \sum_{\substack{i_1,...,i_{10} = 1,...,N \\ i_1,...,i_{10} \text{ distinct}}} \Psi\big(t,x^{i_1},..x^{i_{10}}\big) + C \int_0^t \exp\big(- c'  s \big)ds , 
\end{align*}
where $c'$ is as in the statement of Proposition \ref{prop.supersolexists}, i.e. $0 < c' < \frac{c_0^2}{2} \wedge (\lambda_2 - \lambda_1)$. The key properties of $\Psi_2^N$ are as follows. 

\begin{lem} \label{lem.Psi2N}
    If $C$ is chosen large enough, then for $(t,\bx) \in [0,\infty) \times D^N$, we have 
    \begin{align} \label{psi2n.bound1}
        \partial_t \Psi_2^N - \frac{1}{2} \sum_{i = 1}^N \Delta_{x^i} \Psi_2^N \gtrsim \exp\big( - (\lambda_2 - \lambda_1) t \big) \Big(\frac{1}{N} \sum_{ i= 1}^N \rho(x^i) \Big)^{-9}. 
    \end{align}
    Moreover, for $t > 0$ and $\bx \in \partial^{1,i}(D^N)$, 
    \begin{align} \label{psi2n.bound2}
        \Psi_2^N(t,\bx) \geq \frac{1}{N-1} \sum_{j \neq i} \Psi_2^N\big(t,\bx^{i,j}\big). 
    \end{align}
\end{lem}

\begin{proof}
    We begin by establishing some properties of the function $\Psi : \R_+ \times D^{10} \to \R$ defined above. First, we compute
    \begin{align*}
        \partial_t \Psi(t,\bx) &= \partial_t \psi\big(t, \dist(x^1),...,\dist(x^{10}) \big), 
        \\
        \Delta_{x^i} \Psi(t,\bx) &= \partial_{z^iz^i}\psi\big(t, \dist(x^1), ... \dist(x^{10}) \big) |D\dist(x^i)|^2 + \partial_{z^i} \psi\big(t,\dist(x^1),...,\dist(x^{10})\big) \Delta \dist(x^i),
    \end{align*}
    and thus 
    \begin{align*}
        &\partial_t \Psi - \frac{1}{2} \sum_{i = 1}^{10} \Delta_{x^i} \Psi 
        \\
        &\quad = \partial_t \psi\big(t,\dist(x^1),...\dist(x^{10}) \big) - \frac{1}{2} \sum_{i = 1}^{10} \partial_{z^iz^i} \psi\big(t,\dist(x^1),...,\dist(x^{10})\big) |D\dist(x^i)|^2 
        \\
        &\qquad \qquad - \frac{1}{2} \sum_{i = 1}^{10} \partial_{z^i} \psi\big(t, \dist(x^1),...,\dist(x^{10})\big) \Delta \dist(x^{i}) 
        \\
        &\quad \geq  \partial_t \psi\big(t,\dist(x^1),...\dist(x^{10}) \big) 
        \\
        &\qquad \qquad - \frac{1}{2} \sum_{i = 1}^{10} \partial_{z^iz^i} \psi\big(t,\dist(x^1),...,\dist(x^{10})\big) |D\dist(x^i)|^2  - \frac{c_0}{2} \sum_{i = 1}^{10} \big| \partial_{z^i} \psi\big(t, \dist(x^1),...,\dist(x^{10})\big)\big|
    \end{align*}
    In particular, using property (2) from Proposition \ref{prop.supersolexists} $\psi$, the fact that $\partial_{z^i} \psi \leq 0$ for $\bz \in \R_+^{10}$, and the fact that $|D\dist(x^i)|^2 = 1$ for $x^i \in \big(\partial D\big)_{\delta}$, we find that
     \begin{align*}
        &\partial_t \Psi - \frac{1}{2}\sum_{i = 1}^{10} \Delta_{x^i} \Psi \gtrsim e^{-ct} \Big(\sum_{i = 1}^{10} \rho(x^i) \Big)^{-9} \text{  in  }  \R_+ \times \big(\partial D\big)_{\delta}^{10}.
    \end{align*}
    Meanwhile, by property (5) in Proposition \ref{prop.supersolexists}, we find that 
    \begin{align*}
       \partial_t \Psi - \frac{1}{2} \sum_{i = 1}^{10} \Delta_{x^i} \Psi \gtrsim - \exp(- c' t) \text{ in } \R_+ \times \big( D^{10} \setminus  \big(\partial D\big)_{\delta}^{10} \big). 
    \end{align*}
    As a consequence, for large enough $C$, the function 
    \begin{align*}
       \wt{\Psi} : \R_+ \times D^{10} \to \R, \quad  \wt{\Psi}(t,\bx) = \Psi(t,\bx) + C \int_0^t \exp( - c' s) ds
    \end{align*}
    satisfies 
         \begin{align*}
       \partial_t \wt{\Psi} - \frac{1}{2} \sum_{i = 1}^{10} \Delta_{x^i} \wt{\Psi} &\gtrsim \exp(-c' t) + \exp\big(- (\lambda_2 - \lambda_1) t \big) \Big( \sum_{i = 1}^{10} \rho(x^i) \Big)^{-9} 1_{\bx \in \big(\partial D\big)_{\delta}^{10}}
       \\
       & \gtrsim \exp\big(- (\lambda_2 - \lambda_1) t \big) \Big(\sum_{i = 1}^{10} \rho(x^i) \Big)^{-9} \text{ in } \R_+ \times D^{10}, 
    \end{align*}
    where in the last line we used the fact that $c' < \lambda_2 - \lambda_1$. It follows that 
    \begin{align*}
    \Psi_2^N(t,\bx) = \frac{1}{N(N-1)...(N-9)} \sum_{\substack{i_1,...,i_{10} = 1,...,N \\ i_1,...,i_{10} \text{ distinct}}} \wt{\Psi}\big(t,x^{i_1},..x^{i_{10}}\big)
\end{align*}
satisfies 
\begin{align*}
    \partial_t \Psi_2^N - \frac{1}{2} \sum_i \Delta_{x^i} \Psi_2^N &\gtrsim \exp\big( - (\lambda_2 - \lambda_1)t \big) \frac{1}{N(N-1)...(N-9)} \sum_{\substack{i_1,...,i_{10} = 1,...,N \\ i_1,...,i_{10} \text{ distinct}}} \Big( \rho(x^{i_1}) + ... + \rho_(x^{i_{10}}) \Big)^{-9}
    \\
    &\gtrsim \exp\big( - (\lambda_2 - \lambda_1)t \big) \Big( \frac{1}{N} \sum_{i = 1}^N \rho(x^i) \Big)^{-9},
\end{align*}
with the last line using the convexity of $z \mapsto z^{-9}$ and the fact that
\begin{align*}
    \frac{10}{N} \sum_{i = 1}^n \rho(x^i) = \frac{1}{N(N-1)...(N-9)} \sum_{\substack{i_1,...,i_{10} = 1,...,N \\ i_1,...,i_{10} \text{ distinct}}} \Big( \rho(x^{i_1}) + ... + \rho(x^{i_{10}}) \Big)
\end{align*}
This establishes \eqref{psi2n.bound1}. Finally, the fact that $\partial_{z^i} \psi \leq 0$ for $\bz \in \R_+^{10}$ implies that for $\bx \in \partial_{1,i}(D^N)$, we have $\Psi_2^N(t,\bx) \geq \Psi_2^N(t,\bx^{i,j})$ for \textit{each} $j \neq i$, which obviously implies \eqref{psi2n.bound2}. This completes the proof. 
\end{proof}

\begin{proof}[Proof of Proposition \ref{prop.uniformintime.pointwise}]
        For $\eps_1,\eps_2 \in (0,1)$, consider the optimization problem 
        \begin{align*}
            M_{\eps_1,\eps_2} = \sup_{(t,\bx) \in [1,\infty) \times \ov{D}^N} \Big\{ U^N(t,\bx) - V^N(t,\bx) - \eps_1 \Psi_1^N(t,\bx) - \eps_2 \Psi_2^N(t,\bx) \Big\}. 
        \end{align*}
       We are going to argue for simplicity as if we can find an optimizer $(\hat t, \hat \bx)$ for any $\eps_1,\eps_2 > 0$. To guarantee the existence of an optimizer, we can use the same strategy as in the proof of Proposition \ref{prop.unittime}, looking instead at
       \begin{align*}
            M_{\eps_1,\eps_2, \delta, \eta} = \sup_{(t,\bx) \in [1,\infty) \times \ov{D}^N} \Big\{ U^N(t,\bx) - V^N(t,\bx) - \eps_1 \Psi_1^N(t,\bx) - \eps_2 \Psi_2^N(t,\bx) - \delta \Big( \Xi_1^N(t,\bx) + \Xi_{2,\eta}^N(t,\bx) 
            +t \Big) \Big\}, 
        \end{align*}
        and sending first $\eta \to 0$ and then $\delta \to 0$ for each fixed $\eps_1,\eps_2$. Since this technical points is handled exactly as in the proof of Proposition \ref{prop.unittime}, we ignore it here, and assume we can find an optimizer for the problem defining $M_{\eps_1,\eps_2}$.        
        \newline \newline 
        \textit{Step 1: Avoiding the lateral boundary.} 
        Suppose first that $t > 1$, and that for some $i \in \{1,...,N\}$, $\hat{x}^i \in \partial D$. Then, arguing as in the proof of Proposition \ref{prop.unittime}, we find that 
   \begin{align} 
       0 &= V^N(\hat t, \hat \bx) - \frac{1}{N-1} \sum_{j \neq i} V^N(\hat{t}, \hat{\bx}^{i,j})
     \nonumber  \\
       &\leq U^N(\hat t, \hat \bx) - \frac{1}{N-1} \sum_{j \neq i} U^N(\hat{t}, \hat{\bx}^{i,j}) - \sum_{m = 1}^2 \eps_m \Big( \Psi_m^N(\hat t, \hat \bx) - \frac{1}{N-1} \sum_{j \neq i} \Psi_m^N(\hat{t}, \hat{\bx}^{i,j}) \Big)
   \end{align}
   By Lemmas \ref{lem.Psi1N}, \ref{lem.Psi2N}, we have 
   \begin{align*}
        &\Psi_1^N(\hat t, \hat \bx) - \frac{1}{N-1} \sum_{j \neq i} \Psi_1^N(\hat{t}, \hat{\bx}^{i,j}) \geq  \exp\big(-(\lambda_2 - \lambda_1) t \big) \frac{7}{2^{8} N}  \Big(\frac{1}{N} \sum_{j = 1}^N \rho(\hat x^j) \Big)^{-7},
        \\
        & \Psi_2^N(\hat t, \hat \bx) - \frac{1}{N-1} \sum_{j \neq i} \Psi_2^N(\hat{t}, \hat{\bx}^{i,j}) \geq 0.
   \end{align*}
   Thus, we obtain
   \begin{align*}
      \frac{\eps_1}{N}& \exp\big( - (\lambda_2 - \lambda_1) t \big) \Big(\frac{1}{N} \sum_{j = 1}^N \rho(\hat x^j) \Big)^{-2} \lesssim  \frac{\eps_1}{N} \exp\big( - (\lambda_2 - \lambda_1) t \big) \Big(\frac{1}{N} \sum_{j = 1}^N \rho(\hat x^j) \Big)^{-7}
      \\
      &\qquad \qquad \lesssim U^N(\hat t, \hat \bx) - \frac{1}{N-1} \sum_{j \neq i} U^N(\hat{t}, \hat{\bx}^{i,j})
      \lesssim \frac{1}{N^2} \exp\big(- (\lambda_2 - \lambda_1) t \big) \Big(\frac{1}{N} \sum_{j = 1}^N \rho(\hat x^j) \Big)^{-2}, 
   \end{align*}
   where the last line uses Proposition \ref{prop.UN.properties}. We deduce that if $\eps_1 = C/N$ for some large constant $C$, then any optimizer $(\hat t, \hat \bx)$ with $\hat t > 1$ must satisfy $\hat \bx \in D^N$. 
   \newline \newline 
   \textit{Step 2: Using the equation.} Now we suppose that we have an optimizer $(\hat t, \hat \bx) \in (1,\infty) \times D^N$. Arguing as in the proof of Proposition \ref{prop.unittime}, we get
    \begin{align*}
        0 &=  \partial_t V^N(\hat t, \hat \bx) - \frac{1}{2} \sum_{i = 1}^N \Delta_{x^i} V^N(\hat t, \hat \bx) 
        \\
        &\leq  \partial_t U^N(\hat t, \hat \bx) - \frac{1}{2} \sum_{i = 1}^N \Delta_{x^i} U^N(\hat t, \hat \bx) - \sum_{m = 1}^2 \eps_m \Big( \partial_t \Psi_m^N(\hat t, \hat \bx) - \frac{1}{2} \sum_{i = 1}^N \Delta_{x^i} \Psi_m^N(\hat t, \hat \bx) \Big).
    \end{align*}
    Applying Proposition \ref{prop.UN.properties} and Lemmas \ref{lem.Psi1N}, \ref{lem.Psi2N}, we find the estimate 
    \begin{align*}
     \eps_2 &\exp\big(- (\lambda_2 - \lambda_1)t \big) \Big(\frac{1}{N} \sum_{i = 1}^N \rho(x^i) \Big)^{-9} \lesssim 
        \eps_2 \Big( \partial_t \Psi_2^N(\hat t, \hat \bx) - \frac{1}{2} \sum_{i = 1}^N \Delta_{x^i} \Psi_2^N(\hat t, \hat \bx) \Big) 
        \\
        &\leq  \partial_t U^N(\hat t, \hat \bx) - \frac{1}{2} \sum_{i = 1}^N \Delta_{x^i} U^N(\hat t, \hat \bx) - \eps_1 \Big( \partial_t \Psi_1^N(\hat t, \hat \bx) - \frac{1}{2} \sum_{i = 1}^N \Delta_{x^i} \Psi_1^N(\hat t, \hat \bx) \Big)
        \\
        &\lesssim \frac{1}{N} \exp\big(- (\lambda_2 - \lambda_1) t \big) \Big(\frac{1}{N} \sum_{i = 1}^N \rho(x^i) \Big)^{-2} + \eps_1 \exp\big(- (\lambda_2 - \lambda_1) t \big) \Big(\frac{1}{N} \sum_{i = 1}^N \rho(x^i) \Big)^{-9}
        \\
        &\lesssim \Big(\frac{1}{N} + \eps_1 \Big)   \exp\big(- (\lambda_2 - \lambda_1) t \big) \Big(\frac{1}{N} \sum_{i = 1}^N \rho(x^i) \Big)^{-9}. 
    \end{align*}
    If we choose $\eps_1 = C/N$ and then $\eps_2 = C'\eps_1$ for $C$ and $C'$ large enough, then we reach a contradiction, which means that there cannot be an optimizer in $(1,\infty) \times D^N$. 
    \newline \newline 
    \noindent \textit{Step 3: Conclusion.} Combining Steps 1 and 2, we see that if $C$ and $C'$ are large enough, and we set
    \begin{align*}
        \eps_1 = \frac{C}{N}, \quad \eps_2 = C'\eps_1 = \frac{CC'}{N},
    \end{align*}
    then any optimizer $(\hat t, \hat \bx)$ for the problem defining $M_{\eps_1,\eps_2}$ must satisfy $\hat t = 1$. Moreover, by Proposition \ref{prop.unittime}, we see that by choosing $C$ larger if necessary, we can guarantee that 
    \begin{align*}
        \sup_{\bx \in D^N} \Big\{ U^N(1,\bx) - V^N(1,\bx) - \frac{C}{N} \Psi_1^N(1,\bx) \Big\} \leq 0.
    \end{align*}
    Thus, we find that
    \begin{align*}
        \sup_{(t,\bx) \in [1,\infty) \times D^N} & \Big\{  U^N - V^N - \frac{C}{N} \Psi_1^N - \frac{CC'}{N} \Psi_2^N \Big\} 
        \\
        &=  \sup_{\bx \in D^N} \Big\{ U^N(1,\bx) - V^N(1,\bx) - \frac{C}{N} \Psi_1^N(1,\bx) - \frac{CC'}{N} \Psi_2^N(1,\bx) \Big\}
        \\
        &\leq  \sup_{\bx \in D^N} \Big\{ U^N(1,\bx) - V^N(1,\bx) - \frac{C}{N} \Psi_1^N(1,\bx) \Big\} \leq 0, 
    \end{align*}
    or in other words
    \begin{align*}
        U^N(t,\bx) - V^N(t,\bx) \leq \frac{C}{N} \Psi_1^N + \frac{CC'}{N} \Psi_2^N, \quad \text{for all } (t,\bx) \in [1,\infty) \times D^N.
    \end{align*}
   Using the upper bounds on $\Psi_1^N$ and $\Psi_2^N$ completes the proof.  
\end{proof}

\subsection{Uniform in time weak propagation of chaos}

We now complete the proof of Theorem \ref{thm.conv.finitetime}. 

\begin{proof}[Proof of Theorem \ref{thm.conv.finitetime}]
    By the triangle inequality, 
    \begin{align*}
        \Big|U(t,m) & - \int_{D^N} V^N(t,\bx) m^{\otimes N}(d\bx) \Big| 
        \\
        & \leq  \Big| U(t,m) - \int_{D^N} U\big(t,m_{\bx}^N\big) m^{\otimes N}(d\bx) \Big| + \int_{D^N} \Big| U^N(t,\bx) - V^N(t,\bx) \Big| m^{\otimes N}(d\bx) \coloneqq I + II
    \end{align*}
    To estimate $I$, we mimic the arguments of \cite{ChassegneuxSzpruchTse}. In particular, for any $\bx \in D^N$, we write 
    \begin{align*}
         U\big(t,m_{\bx}^N \big) - U(t,m) = \int_0^1 \frac{\delta U}{\delta m} \Big(t, [m, m_{\bx}^N]_s ,y \Big) d\big(m_{\bx}^N - m\big), 
    \end{align*}
    so that 
    \begin{align*}
        &\int_{D^N} U^N(t,\bx) m^{\otimes N}(d\bx) - U(t,m) 
        \\
        &\quad = \E\big[ U(t, m_{\bm{\xi}}^N) - U(t,m) \big]
        \\
        &\quad = \int_0^1 \bigg( \E\Big[\frac{\delta U}{\delta m}\Big(t, [m, m_{\bm{\xi}}^N]_s, \xi^1 \Big) \Big] - \E\Big[\frac{\delta U}{\delta m}\Big(t, [m, m_{\bm{\xi}}^N]_s, \wt{\xi} \Big) \Big] \bigg)ds
        \\
        &\quad = \int_0^1 \bigg( \E\Big[\frac{\delta U}{\delta m}\Big(t, [m, m_{\bm{\xi}}^N]_s, \xi^1 \Big) \Big] - \E\Big[\frac{\delta U}{\delta m}\Big(t, [m, m_{\wt{\bm{\xi}}}]_s, \xi^1 \Big) \Big] \bigg)ds
        \\
        &\quad = \frac{1}{N} \int_0^1 \int_0^1 s \bigg( \E\Big[ \frac{\delta^2 U}{\delta m^2}\Big(t, \Big[ [m, m_{\wt{\bm{\xi}}}^N]_s, [m, m_{{\bm{\xi}}}^N]_s  \Big]_r, \xi^1, \xi^1\Big) \Big]
        \\
        &\qquad \qquad \qquad \qquad -  \E\Big[ \frac{\delta^2 U}{\delta^2 m}\Big(t, \Big[ [m, m_{\wt{\bm{\xi}}}^N]_s, [m, m_{{\bm{\xi}}}^N]_s  \Big]_r, \xi^1, \wt{\xi} \Big) \Big]\bigg) ds dr
    \end{align*}
    where $\wt{\xi}, \xi^1,...,\xi^N$ are independent with law $m$, $\bm{\xi} = (\xi^1,...,\xi^N)$ and $\wt{\bm{\xi}} = (\wt{\xi},\xi^2,...,\xi^N)$. Note that by convexity of $z \mapsto z^{-2}$, 
    \begin{align*}
          \bigg(  \int_D \rho d \Big[ [m, m_{\wt{\bm{\xi}}}^N]_s, [m, m_{{\bm{\xi}}}^N]_s  \Big]_r \bigg)^{-2} \leq \text{max} \bigg\{ \Big(\int_D \rho \, d m \Big)^{-2}, \Big(\int_D \rho \, dm_{{\bm{\xi}}}^N\Big)^{-2}, \Big(\int_D \rho \, dm_{{\wt{\bm{\xi}}}}^N\Big)^{-2}  \bigg\}, 
    \end{align*}
    so that 
    \begin{align*}
        \E\bigg[  \bigg(  \int_D \rho d \Big[ [m, m_{\wt{\bm{\xi}}}^N]_s, [m, m_{{\bm{\xi}}}^N]_s  \Big]_r \bigg)^{-2} \bigg] \lesssim \Big(\int_D \rho dm \Big)^{-2} + \E\bigg[ \Big(\int_D \rho \, dm_{{\bm{\xi}}}^N\Big)^{-2} \bigg].
    \end{align*}
    In light of the estimate
    \begin{align*}
        \Big|\frac{\delta^2 U}{\delta m^2} (t,m,x,x')\Big| \lesssim \Big( \int \rho dm\Big)^{-2}, 
    \end{align*}
    we deduce that 
    \begin{align*}
        &\Big| \int_{D^N} U^N(t,\bx) m^{\otimes N}(d\bx) - U(t,m) \Big| \lesssim \frac{1}{N} \bigg( 1 + \E\bigg[ \Big(\int_D \rho \, dm_{{\bm{\xi}}}^N\Big)^{-2} \bigg]\bigg) 
        \\
        &\qquad \qquad = \frac{1}{N} \bigg(1 + \int_{D^N} \Big(\frac{1}{N} \sum_{i = 1}^N \rho(x^i) \Big)^{-2} \bigg) dm^{\otimes N}(d\bx).
    \end{align*}
    The last integral is clearly bounded because $\supp(m) \Subset D$, so we get $I \lesssim 1/N$.

    We now turn to estimating $II$. We apply Propositions \ref{prop.unittime} and \ref{prop.uniformintime.pointwise} to get 
    \begin{align*}
        II \lesssim N^{-1} \bigg[  \int_{D^{10}} \Big( \sum_{i = 1}^{10} \rho(x^i) \Big)^{-9} dm^{\otimes 10} + 1_{t \leq 1} \int_D w(t,x) dm(x) \bigg].
    \end{align*}
    Again, the integrals are clearly bounded because $\supp(m) \Subset D$ by assumption, so this shows that $II \lesssim N^{-1}$ and completes the proof.
\end{proof}

We now explain some implications of Theorem \ref{thm.conv.finitetime} for the Fleming-Viot particle system explained in the introduction. To this end, we fix a measure $m \in \cP(D)$, and we denote by $\bX = (X^1,...,X^N)$ a solution to the Fleming-Viot particle system with $\cL(\bX_0) \sim m^{\otimes N}$.

\begin{lem} \label{lem.hminus}
    For any $s > d/2 + 2$, and any $n \in \cP(D)$, the function $G(m) = \frac{1}{2} \norm{ m - n }_{H^{-s}}^2$ lies in $\cC^{2 + \alpha}\big(\cP(\ov D) \big)$ for some $\alpha \in (0,1)$. The first and second linear derivatives are given explicitly by
    \begin{align}
       & \frac{\delta G}{\delta m}(m,x) = (m - n)^*(x) - \langle m, m - n \rangle_{H^{-s}},
      \nonumber  \vspace{.4cm}  \\
      &  \frac{\delta^2 G}{\delta m^2}(m,x,x') =(\delta_{x'})^*(x) - m^*(x) + n^*(x') - 2m^*(x') + 2 \|m\|_{H^{-s}}^2 - \langle m,n \rangle_{H^{-s}}.
    \end{align}
   Moreover, $\frac{\delta G}{\delta m}$, $\frac{\delta^2 G}{\delta m^2}$, $D_m G$, $D_{mm} G$ are bounded in all their arguments, independently of $n$. 
\end{lem}

\begin{proof}
    Obviously, $G$ extends in a natural way to all $q \in {H^{-s}}$, i.e. $G(q) = \frac{1}{2} \| q - n\|_{H^{-s}}^2$. Moreover, $G \in \cC^1(H^{-s})$, with the Fr\'echet derivative given by 
    \begin{align*}
        D_{H^{-s}} G(q)(\cdot) = (q - n)^* \in H^s(\ov D). 
    \end{align*}
    Therefore, for any $m,m' \in \cP(\ov D) \subset H^{-s}$, 
    \begin{align*}
        G(m') - G(m) &= \int_0^1 \langle D_{H^{-s}} G\big( [m',m]_s\big), m' - m \rangle_{s,-s} ds = \int_0^1 \langle \big( [m',m]_s - n\big)^*, (m' - m) \rangle_{s,-s}
         \\
         &= \int_0^1 \int_{\ov D} \big( [m,m']_s - n\big)^*(x) (m' - m)(dx). 
    \end{align*}
    Together with the fact that 
    \begin{align*}
        \int_{\ov D} (m - n)^*(x) m(dx) = \langle (m - n)^*, m \rangle_{s,-s} = \langle (m-n), m \rangle_{H^{-s}}, 
    \end{align*}
    this establishes the formula for $\frac{\delta G}{\delta m}$. Next, we fix $x$, and fix $m,m' \in \cP(\ov D)$, and set $m^{\eps} = [m,m']_{\eps}$. We write
    \begin{align*}
        (m^{\eps} - n)^*(x) - (m - n)^*(x) = \eps (m' - m)^*(x) = \eps \int_{\ov D} \big( \delta_{x'} - m\big)^*(x) m'(dx'),  
    \end{align*}
    and 
    \begin{align*}
        \langle m^{\eps}, m^{\eps}& - n \rangle_{H^{-s}} - \langle m, m - n \rangle_{H^{-s}} = 2 \eps \langle m' - m, m \rangle_{H^{-s}} - \eps \langle m' - m, n \rangle_{H^{-s}} + o(\eps)
        \\
        &= 2 \eps \langle m',m \rangle_{H^{-s}} -\eps  \langle m', n \rangle_{H^{-s}} - 2 \eps \|m\|_{H^{-s}}^2 + \eps \langle m,n \rangle_{H^{-s}} + o(\eps)
        \\
        &=  \eps \int_{\ov D} \Big(  2m^*(x') - n^*(x') - 2 \|m\|_{H^{-s}}^2 + \langle m,n \rangle_{H^{-s}} \Big)  m'(dx') + o(\eps).
    \end{align*}
    Combining the previous two computations, we find that 
    \begin{align*}
       & \frac{\delta G}{\delta m}(m^{\eps}, x) -   \frac{\delta G}{\delta m}(m, x) 
        \\
       &\quad  = \int_{\ov D} \Big[ \big(\delta_{x'} - m\big)^*(x) + n^*(x') - 2m^*(x') + 2 \|m\|_{H^{-s}}^2 - \langle m,n \rangle_{H^{-s}} \Big] m'(dx') + o(\eps).
    \end{align*}
    Together with the fact that 
    \begin{align*}
        \int_{\ov D} \Big[ \big( \delta_{x'} - m\big)^*(x') + n^*(x') - 2 m^*(x') \Big] m(dx') = \langle n, m \rangle_{H^{-s}} - 2 \|m\|_{H^{-s}}^2, 
    \end{align*}
    this establishes the formula for the second derivative.

    Now, to complete the proof, we must show that $D_x D_m G$ and $D_{mm} G$ exist and are H\"older continuous, and $\frac{\delta G}{\delta m}$, $\frac{\delta^2 G}{\delta m^2}$, $D_m G$, $D_x D_m G$ and $D_{mm} G$ are each bounded, independently of $n$. This is a straightforward but tedious consequence of Sobolev embedding, so we explain only the boundedness and H\"older continuity of $D_x D_m G = D_{xx} \frac{\delta^2 G}{\delta m^2}$, and omit the rest of the argument, which follows along similar lines. First, note that because $H^s$ embeds into $L^{\infty}$ and $m$ and $n$ are probability measures, $m-n$ is bounded in $H^{-s}$ by a constant independent of $m$ and $n$. As a consequence, $(m-n)^*$ is bounded in $H^s$ by a constant independent of $m$ and $n$. Since $s > d/2 + 2$, $H^s$ embeds into $C^{2+\alpha}(\ov D)$ for some $\alpha \in (0,1)$, and so
    \begin{align*}
        D_m G (m,x) = D_x \frac{\delta G}{\delta m}(m,x) = D (m -n)^*(x), \quad D_x D_m G(m,x) = D_{xx} \frac{\delta G}{\delta m}(m,x) = D^2 (m-n)^*(x)
    \end{align*}
    exist, and $\| D_m G(m,\cdot)\|_{C^{\alpha}(\ov D)} + \|D_x D_m G(m,\cdot)\|_{C^{\alpha}(\ov D)}$ is bounded by a constant independent of $m$ and $n$. Moreover, since $s > d/2 + 2$,  
    \begin{align*}
        \big\| D_x D_m G(m,\cdot) & - D_x D_m G(m',\cdot) \big\|_{\infty} = \norm{ D^2 (m - n)^*(\cdot) - D^2 (m' - n)^*(\cdot) }_{\infty} \\
        &\leq \norm{ (m-m')^*}_{H^s} = \norm{m-m'}_{H^{-s}}
        \\
        &\lesssim \| m - m' \|_{H^{-s}}
        \lesssim \bd(m,m'), 
    \end{align*}
    so $D_x D_m G(\cdot, x)$ is Lipschitz, uniformly in $n$ and $x$. 
\end{proof}

\begin{proof}[Proof of Corollary \ref{cor.particlesystem.hminuss}]
  To prove \eqref{hminuss.est}, fix $m \in \cP(D)$ with $\supp(m) \Subset D$ and $t \geq 0$. Let $U$ be as defined in \eqref{def.U} with
  \begin{align*}
      G(n) = \frac{1}{2} \|n - \ov p_t^m\|_{H^{-s}}^2, 
  \end{align*}
  and $V^N$ be the unique solution to \eqref{eqn.VN} with $G^N(\bx) = G(m_{\bx}^N)$. 
  
 By Lemma \ref{lem.hminus}, $G$ is in $\cC^{2 + \alpha}\big(\cP(\ov D)\big)$, and $\| G\|_{\cC^2}$ is bounded independently of $t$ and $m$. Moreover, by design, $U(t,m) = 0$, so the result follows from Theorem \ref{thm.conv.finitetime}.

  To prove \eqref{d1.est}, notice that for any $s > 1$, we have the bound 
  \begin{align} \label{comp.d1hminuss}
      \bd_1(m,n) \lesssim \| m - n\|_{H^{-s}}^{1/s}. 
  \end{align}
  To see this, fix a $1$-Lipschitz function $\phi$, set $\phi_{\eps} = \rho_{\eps} * \phi$, where $(\rho_{\eps})_{\eps > 0}$ is a standard approximation to the identity with $\text{diam}(\text{supp}(\rho_{\eps})) \lesssim \eps$. Then 
  \begin{align*}
     \Big| \int_D \phi d(m-n) \Big| &\leq \Big|\int_D \phi_{\eps} d(m-n) \Big| + \Big| \int_{D} (\phi_{\eps} - \phi) d(m - n) \Big| 
     \\
     &\leq \|\phi_{\eps}\|_{H^s} \|m - n\|_{H^{-s}} + 2 \|\phi_{\eps} - \phi\|_{\infty}
     \\
     &\lesssim \eps^{-(s-1)} \|m - n\|_{H^{-s}} + \eps, 
  \end{align*}
  so choosing $\eps = \|m - n\|_{H^{-s}}^{1/s}$ and taking a supremum over 1-Lipschitz functions $\phi$ gives \eqref{comp.d1hminuss}. 
  
  We can thus take $s = d/2 + 2 + \theta/2$, and apply \eqref{hminuss.est} to estimate 
    \begin{align*}
        \E \Big[\bd_1\big(m_{t}^N, \ov{p}_t^{m}\big)\Big] &\lesssim \E\Big[\| m_{t}^N - \ov{p}_t^{m} \|_{H^{-s}}^{1/s}\Big] \leq \Big( \E\Big[\| m_{t}^N - \ov p_t^{m} \|_{H^{-s}}^2\Big] \Big)^{1/(2s)} \lesssim N^{-1/(2s)}.
    \end{align*}
    this completes the proof.
\end{proof}



\begin{proof}[Proof of Corollary \ref{cor.invariantdist}]
    We send $t \to \infty$ in Theorem \ref{thm.conv.finitetime} and use the fact that $\cL(\bX_t) \to P^N$ weakly and $\ov{p}_t^{m} \to \rho \, dx$ weakly, together with the continuity of $\bx \mapsto G(m_{\bx}^N)$.
     This proves \eqref{invariantest1}. The estimates \eqref{invariantest2} and \eqref{invariantest3} follow from \eqref{invariantest1} as in the proof of Corollary \ref{cor.particlesystem.hminuss}.
\end{proof}

\appendix 

\section{Continuity and long time behavior of (normalized) heat flow}\label{app.flow}

We discuss in this appendix some properties of $(p_t^m)_{t \geq 0}$, which is the unique solution of
\begin{align} \label{eqn.pappe}
    \partial_t p = \frac{1}{2} \Delta p \text{ in } \R_+ \times D, \quad p\big|_{(0,\infty) \times \partial D} = 0, \quad p_0 = m\big|_D,
\end{align}
and of the normalized flow $\ov{p}^m_t = (p^m_t(D))^{-1} p^m_t$. It is standard to represent $p_t^m$ as follows :
\begin{equation}\label{eqn.repr}
p_t^m(x) = \sum_{n =1}^{\infty} \exp(- \lambda_n t) \langle m\big|_D, \rho_n \rangle  \rho_n(x) = \sum_{n =1}^{\infty} \exp(- \lambda_n t) \langle m, \rho_n \rangle  \rho_n(x) ,
\end{equation}
that $p_t^m$ is smooth for $t > 0$, and that $p_t^m$ converges to $m$ as $t \to 0$ in the sense of distributions (that is, against smooth test functions with compact support). By a density argument, such convergence holds actually against $W^1_0$ test functions, and since $m(K^c)$ can be made arbitrarily small by a suitable choice of $K \Subset D$, a localization procedure shows that continuity at $t=0$ holds with respect to the metric $\bd$.

\begin{lem}
    \label{lem.masslowerbound} 
    There are constants $\eps_0, C_0, C_1, C_2 > 0$ with the following property: for any $m \in \cP^+(\ov D)$, we have 
    \begin{align}
        C_0e^{-\lambda_1 t} \Big(\int_D \rho dm\Big) \le p_t^m(D) \le C_1   e^{-\lambda_1 t} \quad \forall t>0,
    \end{align}
    and
     \begin{align*}
        \int_D \rho \, d \ov{p}_1^m \geq \eps_0, \quad \| \ov{p}_1^m \|_{L^2(D)} \leq C_2.
    \end{align*}
\end{lem}

\begin{proof}
    Notice that, by duality, 
    \begin{align*}
        p^m_t(D) = \int_D u(t,x) m(dx), 
    \end{align*}
    where $u$ solves 
    \begin{align*}
        \partial_t u = \frac{1}{2} \Delta u \text{ in } \R_+ \times D, \quad u\big|_{\R_+ \times \partial D} = 0, \quad u(0,x) = 1.
    \end{align*}
    Meanwhile, $\frac{1}{\|\rho\|_{\infty}} e^{-\lambda_1 t} \rho(x)$ satisfies the same equation with a smaller initial condition, hence the exponential lower bound on $p_t^m(D)$ follows by the comparison principle.
    
    For the upper bound, note first that by \eqref{eqn.repr} we have
   \begin{align*}
       \| p_t^m \|_{L^2}^2 = \sum_{n = 1}^{\infty} e^{-2\lambda_n t} |\langle \rho_n, m \rangle |^2. 
   \end{align*}
   Recall the estimate $\| \rho_n \|_{\infty} \lesssim \lambda_n^k$, where $k = (d+1)/4$ (see e.g. \cite{Grieser07012002, Arnaudon2018GradientEO} and the references therein), which extends to $\| D \rho_n \|_{\infty} \lesssim \lambda_n^k$ by standard elliptic regularity. Since $\partial_\nu \rho$ is negative, and bounded away from zero on $\partial D$, the pointwise estimate $| \rho_n(x)| \lesssim \lambda_n^k \rho(x)$ holds. Therefore, for $t \ge 1$,
   \begin{align}\label{eqn.l2b}
       \| p_t^m \|_{L^2}^2 \lesssim \sum_{n = 1}^{\infty} e^{-2\lambda_n t} \lambda_n^{2k}\left(\int_D \rho dm\right)^2 \le e^{-2\lambda_1 t} \left(\int_D \rho dm\right)^2 \left (1 + \sum_{n = 1}^{\infty} e^{-2(\lambda_n-\lambda_1)}\lambda_n^{2k}  \right). 
   \end{align}
   Note that the series appearing in the last line converges ($\lambda_n$ has polynomial growth by Weyl's law), hence by H\"older's inequality we get the exponential upper bound on $p_t^m(D)$ for all $t \ge 1$ (it also holds $p_t^m(D) \le 1$ for all $t$).

   To get the bound from below for $\int_D \rho \, d \ov{p}_1^m$, we first have by \eqref{eqn.repr} that
   \[
   \langle p_1^m,\rho\rangle = e^{-\lambda_1} \langle m,\rho\rangle,
   \]
   and by \eqref{eqn.l2b} evaluated at $t=1$ we obtain
   \begin{align*}
       \langle \ov p_1^m, \rho \rangle = \frac{\langle p_1^m, \rho \rangle}{p_1^m(D)}  \ge \frac{e^{-\lambda_1} \langle m,\rho\rangle}{\| p_t^m \|_{L^2} |D|^{1/2}} \gtrsim 1. 
   \end{align*}
   Furthermore, $\|\ov{p}_1^m \|_{L^2}$ is estimated again by \eqref{eqn.l2b} evaluated at $t=1$ and the lower bound on $p_1^m(D)$.
\end{proof}

\begin{lem} \label{lem.ptm}
  The maps
 \begin{align} \label{continuity1}
     (t,m) \mapsto p_t^m(D) \in \R_+, \qquad (t,m) \mapsto \ov{p}_t^m \in \cP(D)
 \end{align}
 are continuous on $(0,\infty) \times \cP^+(\ov D)$. Moreover, for any $K \Subset D$, they are continuous on $[0, \infty) \times \cP(K)$. 
 
 Finally, for any $m \in \cP^c(D)$, there is a constant $C$ such that 
 \begin{align} \label{continuity3}
     0 \leq 1 - p_h^m(D) \leq Ch, \quad \|m - \ov p_h^m\|_{(W^{2,\infty}_{0,\text{tr}})^*} \leq Ch, \quad \forall \,\, 0 \leq h \leq 1.
 \end{align}

\end{lem}

\begin{proof}
 
   Let us begin by verifying the continuity of the maps in \eqref{continuity1} on $(0,\infty) \times \cP^+(\ov D)$. First, note that for any $t > 0$, $m,m' \in \cP(\ov D)$, and any $f$ with $\|f\|_{1,\infty} < \infty$, by duality we have
   \begin{align*}
       \int_D f \, d\big(p_t^{m'} - p_t^{m}\big) = \int_D u(t,\cdot) \, d(m'-m), 
   \end{align*}
   where $u$ solves 
   \begin{align*}
       \partial_t u - \frac{1}{2} \Delta u = 0, \quad u\big|_{\R_+ \times \partial D} = 0, \quad u(0,\cdot) = f. 
   \end{align*}
   By the smoothing properties of the heat equation, we have 
   $\| u(t, \cdot)\|_{1,\infty} \lesssim \frac{1}{\sqrt{t}} \|f\|_{\infty} \leq \frac{1}{\sqrt{t}} \|f\|_{1,\infty}$, and so we obtain the (uniform continuity) estimate 
   \begin{align*}
       \bd(p_t^{m'}, p_t^{m}) \lesssim \frac{1}{\sqrt{t}} \bd(m,m'). 
   \end{align*}
   Together with the continuity of $\R_+ \ni t \mapsto \ov{p}_t^m \in \sub(D)$ for each fixed $m$, this implies the continuity of $(t,m) \mapsto p_t^m$ on $\R_+ \times \cP^+(\ov D)$. Since the function $f(x) = 1$ is bounded and Lipschitz, we deduce that the function $(t,m) \mapsto p_t^m(D) \in \R$ is continuous on $\R_+ \times \cP^+(\ov D)$. Since for all $a, b > 0$ and $m,n \in \cP_{\rm sub}(\overline D)$ one has $\bd(am, bn) \leq a \bd(m, n) + |a-b|$,
   we deduce that 
   \begin{align*}
       \bd\big(\ov p_{t'}^{m'}, \ov p_t^m \big) \leq   \frac{1}{p_{t'}^{m'}(D)} \bd(p_{t'}^{m'}, p_t^m) + \Big| \frac{1}{p_{t'}^{m'}(D)} - \frac{1}{p_t^{m}(D)} \Big|, 
   \end{align*}
   and so the continuity of $(t,m) \mapsto \ov p_t^m$ on $\R_+ \times \cP^+(\ov D)$ follows from the continuity of $(t,m) \mapsto p_t^m$ and $(t,m) \mapsto p_t^m(D)$ on the same domain, together with Lemma \ref{lem.masslowerbound} above, which ensures that if $m' \to m$ in $\cP^+(\ov D)$ and $t' \to t$, then $p_{t'}^{m'}(D)$ is bounded from below. 

   We now fix a compact set $K$, and address the continuity of the maps on $[0,\infty) \times \cP(K)$. Along the way, we will also obtain the estimate $1 - p_h^m(D) \leq Ch$. First, let us fix a compact set $K'$ such that $K \Subset (K')^o \Subset D$, where $(K')^o$ denotes the interior of $K'$. Then, we can find a smooth function $\kappa \in C_c^{\infty}(D)$ such that 
   \begin{align*}
       0 \leq \kappa(x) \leq 1 \text{ for $x \in D$,} \quad \kappa(x) = 1 \text{ for } x \in K, \quad \kappa(x) = 0 \text{ for } x \in D \setminus K'. 
   \end{align*}
   Then, by the definition of weak solution, for any $m \in \cP(K)$ we have
   \begin{align} \label{phmDest}
       1 - p_h^m(D) \leq \int_D \kappa \, d(m - p_h^m) = - \int_0^h \int_D\frac{1}{2} \Delta \kappa \, dp_t^m dt \leq h \|\kappa\|_{W^{2,\infty}}, 
   \end{align}
   which proves the first estimate in \eqref{continuity3}. Now, for a test function $f \in W^{2,\infty}$, and any $m \in \cP(K)$, we can estimate
   \begin{align*}
       \int_D f \, d(p_h^m - m) &= \int_D f \kappa d\big(p_h^m - m\big) + \int_D f \big(1 - \kappa\big) \, d\big(p_h^m - m\big)
       \\
       &= \frac{1}{2} \int_0^h \int_D \Delta(f \kappa) dp_t^m dt + \int_D f\big(1 - \kappa\big) dp_h^m
       \\
       &\lesssim_K \|f\|_{W^{2,\infty}} h + \|f\|_{\infty} \big(1 - \int_D \kappa dp_h^m\big) \lesssim_K \|f\|_{W^{2,\infty}} h.
   \end{align*}
   In other words, we have proven that for any compact set $K \Subset D$, there is a constant $C > 0$ such that, for any $m \in \cP(K)$ and any $h > 0$, 
   \begin{align*}
       \sup_{\|f \|_{W^{2,\infty}} \leq 1} \int_{D} f \, d(p_h^m - m) \leq Ch. 
   \end{align*}
   Note that in the above estimate, we do not require $f$ to vanish on the boundary. 
   
   Moreover, one can check that $$\bd(m,n) \lesssim  \Big(\sup_{\|f \|_{W^{2,\infty}} \leq 1} \int_{D} f \, d(m - n) \Big)^{1/2}.$$
   Thus, 
   \begin{align*}
       \bd\big(p_{t'}^{m'}, m\big) \leq \bd\big( p_{t'}^{m'}, m' \big) + \bd(m',m\big) \lesssim \sqrt{t'} + \bd(m',m) \to 0, 
   \end{align*}
   as $t'\to 0$ and $\bd(m',m) \to 0$, i.e. $(t,m) \mapsto p_t^m$ is continuous at $(0,m)$. Together with the already established continuity on $\R_+ \times \cP^+(\ov D)$, this implies continuity on $[0,\infty) \times \cP(K)$. The continuity of $p_t^m(D)$ and $\ov{p}_t^m(D)$ on $[0,\infty) \times \cP(K)$ is proved as before, and we omit the details. 

   Finally, to prove the estimate on $\| m - \ov p_h^m\|_{(W^{2,\infty}_{0, \text{tr}})^*}$ we start by observing that
   \begin{align*}
       \| m - p_h^m \|_{(W^{2,\infty}_{0,\text{tr}})^*} \leq Ch.
   \end{align*}
   Indeed, this follows from the formula (which has already been used above)
   \begin{align*}
       \int_D f \, d(p_h^m - p) = \frac{1}{2} \int_0^h \int_{D} \Delta f \, dp_t^m dt 
   \end{align*}
   for any $f \in W_{0,\text{tr}}^{2,\infty}$. 
   Thus, for any $f \in W_{0,\text{tr}}^{2,\infty}$ and $m \in \cP^c(D)$, we have
   \begin{align*}
       \int_D f \, d\big( \ov p_h^m - m \big) &= \Big(\frac{1}{p_h^m(D)} - 1 \Big) \int_D f dp_h^m + \int_D f \, d\big( p_h^m - m \big) 
       \\
       &\leq \|f\|_{\infty} \Big|1 - \frac{1}{p_h^m(D)}\Big| + \|f\|_{W^{2,\infty}} \|p_h^m - m\|_{(W^{2,\infty}_0)^*} = O(h), 
   \end{align*}
   where to conclude we use the estimate \eqref{phmDest}.
\end{proof}

\begin{lem} \label{lem.Ulongtime}
    There is a constant $C$ such that for all $t \geq 1$ and all $m \in \cP^+(\ov D)$, we have
    \begin{align}
       \label{ptm.conv}
       \norm{ p_t^m - e^{-\lambda_1 t} \langle m, \rho_1 \rangle \rho_1}_{L^2(D)} \leq C \exp\big(- \lambda_2 t \big), 
    \end{align}
    as well as
    \begin{align} \label{ovptm.conv}
        \norm{ \ov{p}_t^m - \rho }_{L^2(D)} \leq C \exp\big( - (\lambda_2 - \lambda_1) t \big).
    \end{align}
\end{lem}

\begin{proof} In light of Lemma \ref{lem.masslowerbound}, we may assume that 
\begin{align*}
    \int_D \rho \, dm \geq \eps_0, \quad \|m\|_{L^2} \leq C_0,
\end{align*}    
for some universal constants $\eps_0$ and $C_0$. We thus have
\begin{align*}
    \norm{p_t^m - e^{-\lambda_1 t} \langle m, \rho_1 \rangle \rho_1 }_{L^2}^2 = \sum_{n = 2}^{\infty} e^{-2 \lambda_n t} |\langle m, \rho_n \rangle|^2 \lesssim e^{-2 \lambda_2 t} \|m\|_{L^2}^2 \lesssim e^{-2 \lambda_2 t}, 
\end{align*}
which proves \eqref{ptm.conv}.
Next, we note that
\begin{align*}
    \big| p_t^m(D) - e^{-\lambda_1 t} \langle m, \rho_1 \rangle \rho_1(D) \big| \lesssim e^{- \lambda_2 t}.
\end{align*}
We now write 
\begin{align*}
    \ov{p}_t^m - \rho = \frac{p_t^m}{p_t^m(D)} - \rho = p_t^m \Big( \frac{1}{p_t^m(D)} -\frac{1}{ e^{-\lambda_1 t} \langle m, \rho_1 \rangle \rho_1(D) } \Big) + \frac{p_t^m - e^{-\lambda_1 t} \langle m, \rho_1 \rangle \rho_1}{e^{-\lambda_1 t} \langle m, \rho_1 \rangle \rho_1(D)} = : I + II.
\end{align*}
We estimate 
\begin{align*}
    \|I\|_{L^2} &= \|p_t^m\|_{L^2} \Big| \frac{1}{p_t^m(D)} - \frac{1}{ e^{-\lambda_1 t} \langle m, \rho_1 \rangle \rho_1(D) } \Big| 
    \\
    &= \|p_t^m\|_{L^2} \Big| \frac{e^{-\lambda_1 t} \langle m, \rho_1 \rangle \rho_1(D) - p_t^m(D)}{p_t^m(D) e^{-\lambda_1 t} \langle m, \rho_1 \rangle \rho_1(D)} \Big|
    \\
    &\lesssim e^{- \lambda_1 t} \frac{e^{- \lambda_2 t}}{e^{-2\lambda_1 t}} = e^{-(\lambda_2 - \lambda_1)t}, 
\end{align*}
where the last upper bound used the fact that we assume $\int \rho dm \geq \eps_0$ and $\| m \|_{L^2} \leq C_0$ for some universal constants $\eps_0$ and $C_0$, and the lower bound on $p_t^m(D)$ comes from Lemma \ref{lem.masslowerbound}. 

Meanwhile, similar reasoning also shows that $\|II\|_{L^2} \lesssim e^{- (\lambda_2 - \lambda_1)t}$, completing the proof of \eqref{ovptm.conv}.
\end{proof}

\section{Some estimates for the heat equation on $D$ and $D \times D$}

In this section, we prove some auxiliary estimates on the equations 
\begin{align} \label{eqn.u}
    \partial_t u - \frac{1}{2} \Delta u = 0 \text{ in } \R_+ \times D \quad u\big|_{\R_+ \times \partial D} = 0, \quad u(0,x) = f(x), 
\end{align}
and 
\begin{align} \label{eqn.v}
    \partial_t v - \frac{1}{2} \Delta_x v - \frac{1}{2} \Delta_{x'} v = 0 \text{ in } \R_+ \times D^2, \quad v\big|_{\R_+ \times \partial(D^2)} = 0, \quad v(0,x,x') = g(x,x').
\end{align}
The initial conditions $f : \ov D \to \R$, $g : (\ov D)^2 \to \R$ are assumed to be smooth, but \textit{not} compatible with the lateral boundary condition, in the sense that we do not have $f = 0$ on $\partial D$ or $g = 0$ on $\partial (D^2)$. By parabolic regularity, $u$ and $v$ are smooth on $\text{Dom}_1$ and $\text{Dom}_2$ respectively. 

\begin{prop} \label{prop.ueqn}
Suppose that $f \in C^1(\ov D)$. Then \eqref{eqn.u} has a unique bounded weak solution which is smooth on $\R_+ \times D$, and satisfies the estimate 
\begin{align*}
    |Du(t,x)| \lesssim (\|f \|_{\infty} + \| Df\|_{\infty}\big) \big(1 + w(t,x)\big), \quad 0 < t \leq 1, \quad x \in D,
\end{align*}
where $w$ is as defined in \eqref{def.w}, with $C$ large enough that the conclusions of Lemma \ref{lem.w} hold.
\end{prop}

Before we begin the proof, we introduce some notation which will be used throughout this section. First, we let $(\kappa^{\eps})_{0 < \eps < 1}$ be a family of functions such that 
\begin{align*}
    \kappa^{\eps} : D \to [0,1], \quad \kappa^{\eps}(x) = 1 \text{ if } \dist(x) \geq \sqrt{\eps}, \quad \kappa^{\eps}(x) = 0 \text{ if } \dist(x) \leq \frac{\sqrt{\eps}}{2}, \quad |D \kappa^{\eps}(x)| \lesssim \big( 1 + \frac{1}{\sqrt{\eps}} 1_{\dist(x) \leq \sqrt{\eps}} \big).
\end{align*}
We also define 
 \begin{align*}
     w^{\eps} : [0,\infty) \times \ov D \to \R, \quad    w^{\eps}(t,x) = w(t+ \eps, x) = \gamma\big(t + \eps, \dist(x) \big) + C(t + \eps).
\end{align*}
We note for later use that for $x \in \partial D$ and $ 0 \leq t \leq 1$, 
\begin{align} \label{weps.lateral}
        w^{\eps}(t,x) \geq \gamma(t + \eps, 0) \geq \frac{1}{\sqrt{2\pi(t+ \eps)}} \exp\Big[ - \frac{c_0^2}{2} t \Big] \gtrsim \frac{1}{\sqrt{t+ \eps}} \gtrsim \frac{1}{\sqrt{t}} \wedge \frac{1}{\sqrt{\eps}}, 
    \end{align}
while at time $t = 0$, one easily checks that 
\begin{align*}
    \dist(x) \leq \sqrt{\eps} \implies w^{\eps}(0,x) = \gamma(\eps, x) \geq \frac{1}{\sqrt{2\pi \eps}} \exp\Big[ - \frac{\dist(x)^2}{2\eps} \Big] \gtrsim \frac{1}{\sqrt{\eps}}.
\end{align*}
As a consequence, 
\begin{align} \label{weps.time0}
    1 + w^{\eps}(0,x) \gtrsim 1 +  \frac{1}{\sqrt{\eps}} 1_{\dist(x) \leq \sqrt{\eps}}.
\end{align}
We now turn to the proof of Proposition \ref{prop.ueqn}. 

\begin{proof}
    Without loss of generality we assume that $\|f\|_{\infty} + \|Df\|_{\infty} \leq 1$.  For $0 < \eps < 1$, define 
    \begin{align*}
        f^{\eps}(x) = f(x) \kappa^{\eps}(x).
    \end{align*}
    We note that
    \begin{align*}
        f^{\eps}(x) = f(x) \text{ for } \dist(x) \geq \sqrt{\eps}, \quad f^{\eps}(x) = 0 \text{ for } x \in \partial D, \quad |D f^{\eps}(x)| \lesssim \big(1 + \frac{1}{\sqrt{\eps}} 1_{\dist(x) \leq \sqrt{\eps}} \big).
    \end{align*}
    Then let $u^{\eps}$ denote the solution to \eqref{eqn.u}, but with $f$ replaced with $f^{\eps}$. Then $u^{\eps}$ is globally smooth, and satisfies 
    \begin{align*}
        \| D u^{\eps} \|_{\infty} \lesssim  \| f^{\eps}\|_{\infty} + \|Df^{\eps}\|_{\infty} \lesssim \frac{1}{\sqrt{\eps}}, 
    \end{align*}
    and in addition by the smoothing effect of the heat equation, 
    \begin{align*}
        \| Du^{\eps}(t,\cdot)\|_{\infty} \lesssim \frac{1}{\sqrt{t}} \|f^{\eps}\|_{\infty} \lesssim \frac{1}{\sqrt{t}}. 
    \end{align*}
    For $\eps > 0$, the function $w^{\eps}$ introduced above is continuous on all of $[0,\infty) \times \ov D$, and by Lemma \ref{lem.w}, satisfies 
    \begin{align*}
        \partial_t w^{\eps} - \frac{1}{2} \Delta w^{\eps} \geq 0
    \end{align*}
    on $\R_+ \times D$.
    Putting these facts together, and also making use of the initial condition $u^{\eps}(0,x) = f^{\eps}(x)$, we see that for any unit vector $\nu$, the function $D_{\nu} u^{\eps} = \langle \nu, Du^{\eps} \rangle$ is a solution to the heat equation on $D$, which satisfies
    \begin{align*}
         |D_{\nu} u^{\eps}(t,x)|  \lesssim \frac{1}{\sqrt{\eps}} \wedge \frac{1}{\sqrt{t}} \lesssim 1 + w^{\eps}(t,x) \quad t > 0, \, x \in \partial D, 
        \\
        |D_{\nu} u^{\eps}(0,x)| = |D_{\nu} f^{\eps}(x)| \lesssim 1 + \frac{1}{\sqrt{\eps}} 1_{\dist(x) \leq \sqrt{\eps}} \lesssim 1 + w^{\eps}(t,x), \quad x \in D.
    \end{align*}
    We deduce that for a large enough constant $C$ (independent of $\eps$ and $\nu$), the function $C( 1 + w^{\eps})$ is a supersolution of the heat equation and $C(1 + w^{\eps}) \geq D_{\nu} u^{\eps}$ on both the lateral and initial boundaries. By the comparison principle, we get 
    \begin{align*}
        D_{\nu} u^{\eps}(t,x) \leq C(1 + w^{\eps}(t,x)). 
    \end{align*}
    Sending $\eps \to 0$ gives the result, since $D_{\nu} u^{\eps} \to D_{\nu} u$ pointwise on $\R_+ \times D$. 
\end{proof}

We now turn our attention to the equation \eqref{eqn.v}. We need the following technical lemma, which will allow us to justify certain computations when we prove Proposition \ref{prop.veqn} below.

\begin{lem} \label{lem.veqn}
    Suppose that $g \in C_c^{\infty}(D \times D)$. Then the unique weak solution of \eqref{eqn.v} lies in $C^{\infty}\big([0,\infty) \times \ov{D \times D}\big)$, i.e. it is smooth up to the entire parabolic boundary. 
\end{lem}

\begin{proof}
    For $j,k \in \N$, define
    \begin{align*}
        \lambda_{j,k} = \lambda_j + \lambda_k, \quad \rho_{j,k}(x,x') = \rho_j(x) \rho_k(x'), \quad c_{j,k} = \big \langle \rho_{j,k}, g \big \rangle_{L^2(D \times D)}.
    \end{align*}
    Then we have the representation (note that $\{\rho_{j,k}\}$ is an orthonormal basis of $L^2(D\times D)$)
    \begin{align} \label{rep.v}
        v(t,x,x') = \sum_{j,k = 1}^{\infty} c_{j,k}  e^{-\lambda_{j,k} t} \rho_{j,k}(x,x'). 
    \end{align}
    Since $D$ is smooth, we have $\rho_j \in C^{\infty}(\ov D)$ for each $j$, and moreover for any multi-index $\alpha$, the $\| D^{\alpha} \rho_{j} \|_{\infty}$ grows at most polynomially in $\lambda_j$ (this can be seen from starting with the estimate $\| \rho_{j} \|_{L^{\infty}(D)} \lesssim \lambda_j^{(d-1)/2}$ from \cite{Grieser07012002}, and then bootstrapping). As a consequence, for any multi-indices $\alpha, \alpha'$, the quantity $\| D_x^{\alpha} D_{x'}^{\alpha'} \rho_{j,k}\|_{L^{\infty}(D \times D)}$ grows at most polynomially in $\lambda_{j,k}$. It follows easily that $v \in C^{\infty}\big((0,\infty) \times (\ov D \times \ov D) \big)$. 

    It remains to check smoothness at time $0$. For this, we use the following spectral argument. First, notice that because $g \in C_c^{\infty}\big( D \times D\big)$, 
    \begin{align*}
        \sum_{j,k = 1}^{\infty} c_{j,k} \lambda_{j,k}^M \rho_{j,k} = \sum_{j,k = 1}^{\infty} \langle \rho_{j,k}, \Big(-\frac{1}{2}\Delta\Big)^M g \big \rangle_{L^2(D \times D)} \rho_{j,k} = \Big(-\frac{1}{2}\Delta\Big)^M g \in L^2(D \times D), 
    \end{align*}
    and so 
    \begin{align*}
        \sum_{j,k = 1}^{\infty} c_{j,k}^2 \lambda_{j,k}^{2M} = \norm{ \Big(- \frac{1}{2} \Delta \Big)^M g }_{L^2} < \infty. 
    \end{align*}
    In particular, we have 
    \begin{align*}
        c_{j,k} \lesssim_M \lambda_{j,k}^{-M}, \quad \forall \,\, M \in \N. 
    \end{align*}
    That is, $c_{j,k}$ decays faster than any polynomial in $\lambda_{j,k}$. Combined with the polynomial growth of the derivatives of $\rho_{j,k}$ discussed above and the representation formula \eqref{rep.v}, this implies that for any multi-indices $\alpha, \alpha'$, 
    \begin{align*}
        D_{x}^{\alpha} D_{x'}^{\alpha'} v(t,\cdot,\cdot) \xrightarrow{ t \downarrow 0} D_x^{\alpha} D_{x'}^{\alpha'} g \quad \text{in } L^{\infty}(\ov D \times \ov D). 
    \end{align*}
    It follows that $v \in C^{\infty}\big([0,\infty) \times (\ov D \times \ov D) \big)$, as claimed. 
\end{proof}

\begin{prop} \label{prop.veqn}
Suppose that $g \in C^2(\ov D \times \ov D)$. Then \eqref{eqn.v} has a unique bounded weak solution which is smooth on $\R_+ \times D \times D$, and satisfies the estimate 
\begin{align*}
    |D_x D_{x'} v(t,x,x')| \lesssim \|g\|_{C^2} \big(1 + w(t,x)\big)\big(1 + w(t,x')\big), \quad 0 < t \leq 1, \quad x,x' \in D.
\end{align*}
\end{prop}

\begin{proof}
  We assume without loss of generality that $g \in C^{\infty}(\ov D \times \ov D)$. We begin by defining 
  \begin{align*}
      g^{\eps}(x,x') = \kappa^{\eps}(x) \kappa^{\eps}(x') g(x,x'), 
  \end{align*}
  and let $v^{\eps}$ solve the same equation as $v$, but with $g^{\eps}$ replacing $g$. By Lemma \ref{lem.veqn}, $v^{\eps} \in C^{\infty}\big([0,\infty) \times \ov{D \times D} \big)$, which allows us to justify the computations below.
  
  Notice that, from the Dirichlet boundary conditions for $v^{\eps}$, we get
\begin{align} \label{Dxv.lateral1}
   x' \in \partial D \implies  D_x v^{\eps}(t,x,x') = 0.
\end{align}
Now, let $u^{\eps}$ denote the solution of 
\begin{align*}
    \partial_t u^{\eps} - \frac{1}{2} \Delta u^{\eps} = 0 \text{ in } \R_+ \times D, \quad u^{\eps}\big|_{\R_+ \times \partial D} = 0, \quad u^{\eps}(0,x) = \kappa^{\eps}(x).
\end{align*}
Then, the comparison principle implies that there is a constant $C$ independent of $\eps$ such that
\begin{align*}
 - C u^{\eps}(t,x) u^{\eps}(t,x') \leq  v^{\eps}(t,x,x') \leq Cu^{\eps}(t,x) u^{\eps}(t,x').
\end{align*}
As a consequence, and recalling that the proof of Proposition \ref{prop.ueqn} showed that $|Du^{\eps}(t,x)| \lesssim 1 + w^{\eps}(t,x)$, we have
\begin{align} \label{Dxv.lateral2}
   | D_x v^{\eps}(t,x,x') | \leq u^{\eps}(t,x') \big(1 + w^{\eps}(t,x)\big) \text{ for } x \in \partial D.
\end{align}
In addition, we know that 
\begin{align} \label{Dxv.terminal1}
  \nonumber   D_x v^{\eps}(0,x,x')  &= D \kappa^{\eps}(x) \kappa^{\eps}(x') g(x,x') + \kappa^{\eps}(x) \kappa^{\eps}(x') D_x g(x,x') 
  \\
  &\lesssim \kappa^{\eps}(x') \big(1 + \frac{1}{\sqrt{\eps}} 1_{\dist(x) \leq \sqrt{\eps}} \big) \lesssim \kappa^{\eps}(x') \big(1 + w^{\eps}(0,x)\big).
\end{align}
Combining \eqref{Dxv.lateral1}, \eqref{Dxv.lateral2}, \eqref{Dxv.terminal1}, we see that for any unit vector $\nu$, $D_{x;\nu} v^{\eps} = \langle \nu, D_x v^{\eps} \rangle$ is a bounded function which is smooth on the domain 
\begin{align*}
     \Big( [0,\infty) \times \ov D \Big) \setminus \Big(\{0\} \times \partial D \Big),
\end{align*}
satisfies
\begin{align*}
    \partial_t D_{x; \nu} v^{\eps} - \frac{1}{2} \Delta_x  D_{x; \nu} v^{\eps} - \frac{1}{2} \Delta_{x'}  D_{x; \nu} v^{\eps} = 0 \text{ in } \R_+ \times D^2, 
\end{align*}
as well as 
\begin{align*}
D_{x; \nu} v^{\eps} \lesssim u^{\eps}(t,x') \big(1 + w^{\eps}(t,x) \big) \text{ on } \Big(\R_+ \times \partial(D^2) \Big) \cup \Big(\{0\} \times D^2 \Big).
\end{align*}
Meanwhile, one can check that if $u^1$ and $u^2$ are non-negative supersolutions to the heat equation on $D$, then 
\begin{align*}
    (t,x,x') \mapsto u^1(t,x) u^2(t,x')
\end{align*}
is a supersolution to the heat equation on $D \times D$, so that in particular $u^{\eps}(t,x') \big(1 + w^{\eps}(t,x) \big)$ is a supersolution. By the comparison principle, we thus obtain
\begin{align*}
    D_{x,\nu} v^{\eps}(t,x,x') \lesssim u^{\eps}(t,x') \big(1 + w^{\eps}(t,x) \big).
\end{align*}
As a consequence, we find that
\begin{align*}
    |D_{x'} D_x v^{\eps}(t,x,x')| \lesssim |D u^{\eps}(t,x')|  \big(1 + w(t,x)\big) \lesssim \big(1 + w^{\eps}(t,x) \big)\big(  1 +  w^{\eps}(t,x' \big) \text{ if } x' \in \partial D.
\end{align*}
Reversing the roles of $x$ and $x'$ gives the same bound if $x \in \partial D$. Finally, one can easily check that
\begin{align*}
   | D_{x'} D_x v^{\eps}(0,x,x') | &= | D_{x'} D_x \big ( \kappa^{\eps}(x) \kappa^{\eps}(x') g(x,x') \big) |
    \\
    &\lesssim \big(1 + \frac{1}{\sqrt{\eps}} 1_{\dist(x) \leq \sqrt{\eps}} \big) \big(1 + \frac{1}{\sqrt{\eps}} 1_{\dist(x') \leq \sqrt{\eps}} \big) \lesssim \big(1 + w^{\eps}(0,x)\big)\big(1 + w^{\eps}(0,x')\big).
\end{align*}
In particular, for any pair of unit vectors $\nu,\nu'$, the scalar function $D_{x'; \nu'} D_{x; \nu} v$ is smooth on
\begin{align*}
   \Big( [0,\infty) \times \ov D \times \ov D \Big) \setminus \Big(\{0\} \times \partial ( D \times D) \Big),
\end{align*}
which satisfies 
\begin{align*}
    \partial_t \big( D_{x'; \nu'} D_{x; \nu} v^{\eps}\big) - \frac{1}{2}  \Delta_x\big( D_{x'; \nu'}  D_{x; \nu} v^{\eps} \big) - \frac{1}{2} \Delta_{x'}  \big( D_{x'; \nu'}D_{x; \nu}  v^{\eps}\big) = 0 \text{ in } \R_+ \times D^2, 
\end{align*}
as well as 
\begin{align*}
D_{x'; \nu'} D_{x; \nu} v^{\eps} \lesssim \big(1 + w^{\eps}(t,x')\big) \big(1 + w^{\eps}(t,x) \big) \text{ on } \Big(\R_+ \times \partial(D^2) \Big) \cup \{0\} \times D^2.
\end{align*}
Arguing as above, we propagate this bound to the interior via the comparison principle to find that 
\begin{align*}
    |D_x D_{x'} v^{\eps}| \lesssim \big(1 + w^{\eps}(t,x)\big) \big(1 + w^{\eps}(t,x') \big), 
\end{align*}
and then send $\eps \to 0$ to complete the proof.
\end{proof}

\section{Building the supersolution $\phi$}

In this section, we fix $k \in \N$ and $c,c_0, \delta, M > 0$. We are going to build the barrier function which played a key role in the uniform-in-time convergence result in Subsection \ref{subsec.uniformintime}. In particular, our goal is to prove the following statement.

\begin{prop} \label{prop.supersolexists}
    There exists a function 
    \begin{align*}
        \psi : \R_+ \times \R_+^k \to \R
    \end{align*}
    such that 
    \begin{enumerate}
        \item $\psi$ is smooth on $\R_+ \times \R_+^k$, and extends continuously to the domain 
        \begin{align*}
           \bigg( \R_+ \times \Big( [0,\infty)^k \setminus \{\bm 0\}\Big) \bigg) \bigcup \Big( \{0\} \times \R_+^k \Big), 
        \end{align*}
        \item $\psi(t,\bz)$ satisfies 
        \begin{align*}
            \partial_t \psi - \frac{1}{2} \sum_{i = 1}^{k} \partial_{z^iz^i} \psi + c_0 \sum_{i = 1}^k \partial_{z^i} \psi \gtrsim e^{-ct} |\bz|^{-(k-1)} \text{ on } \R_+ \times \R_+^k
        \end{align*}
        \item $\psi$ satisfies the bounds
        \begin{align} \label{psi.upperbound}
            0 \leq \psi(t,\bz) \lesssim |\bz|^{-(k-1)}  + 1, 
        \end{align}
        \item for each $i = 1,...,k$, we have 
        \begin{align*}
            \partial_{z^i} \psi \leq 0 \text{ on } \R_+ \times \R_+^k, 
        \end{align*}
        \item for any $c' < \frac{c_0^2}{2} \wedge c$, and any $i = 1,...,k$, we have the estimate
        \begin{align*}
            | \partial_t \psi| + |\partial_{z^i} \psi| + |\partial_{z^iz^i} \psi | \lesssim e^{- c' t} \text{ on } U \coloneqq \Big\{ (t,\bz) \in \R_+ \times (0,M)^k : z^j > \delta \text{ for some } j \Big\}.
        \end{align*}
    \end{enumerate}

\begin{proof}
The argument will be broken up into steps. First we will define a candidate $\psi$, then we will verify each of the properties (1) - (5) in turn, with some of the more tedious computations deferred to a sequence of lemmas which are stated and proved after the main line of the proof.
\newline \newline 
\textit{Construction of $\psi$.} We denote by $f$ a function
\begin{align*}
    f : \R^k \setminus \{\bm 0\} \to \R 
\end{align*}
such that 
\begin{itemize}
    \item $f(\bz) = |\bz|^{-(k-1)}$ for $\bz \in (-\delta/4, \delta/4)^{k}$, 
    \item $f(\bz) = 0$ if $|z^i| \geq \delta/2$ for some $i = 1,....,k$,
    \item $f$ is non-negative, smooth on $\R^k \setminus \{\bm{0}\}$, and satisfies the symmetry condition 
    \begin{align*}
        f(\bz) = f\big(|z^1|,...,|z^k|\big), 
    \end{align*}
    \item for each $i = 1,...,k$, 
    \begin{align*}
        \partial_{z^i} f \leq 0 \text{ on } {\{z^i > 0\}}.
    \end{align*}
\end{itemize}
To construct $f$, one can take $|\bz|^{-(k-1)} \prod_{i = 1}^k \xi( z^i)$, where $\xi$ is an appropriate cut-off function. Now, let $P_t$ denote the semi-group associated with the transition kernel $\Gamma$ defined in \eqref{def.Gamma}, and let $P_t^{\otimes k}$ denote its $k$-fold tensor product, so that 
\begin{align*}
    P_t^{\otimes k} g (\bz) = \int_{\R^k} \Gamma_t^{\otimes k}(\bz,\ov \bz) g(\ov \bz) d\ov{\bz}, \quad \Gamma_t^{\otimes k}(\bz,\ov \bz) = \prod_{i = 1}^k \Gamma_t(z^i, \ov z^i). 
\end{align*}
Now define $\psi : \R_+ \times \R^k \to \R$ via 
\begin{align*}
    \psi(t,\bz) = \int_0^t e^{-c(t-s)} \Big( P_s^{\otimes k} f \Big)(\bz) ds + \int_0^t e^{-c(t-s)} ds
\end{align*}
Note that this is the Duhamel representation for the PDE 
\begin{align} \label{psi.duhamel}
 \ds  \partial_t \psi - \frac{1}{2} \sum_{i = 1}^k \partial_{z^iz^i} \psi + c_0 \sum_{i = 1}^k \text{sign}(z^i) \partial_{z^i} \psi = e^{-ct} \big(1 + f(\bz) \big) \text{ in } \R_+ \times \R^k, \quad \psi(0,\bz) = 0.
\end{align}
Property (2) is clear from the PDE \eqref{psi.duhamel}, and the simple observation that $|\bz|^{-(k-1)} \lesssim 1 + f(\bz)$. Property (1) follows from local regularity for the PDE \eqref{psi.duhamel}. We next verify Properties (3), (4), and (5). 
\newline \newline 
\textit{Property (3).} Since $f$ is non-negative, it is clear from the definition that $\psi$ is also non-negative. For $t \geq 1$, we can estimate 
\begin{align*}
    \big|P_t^{\otimes k} f(\bz) \big| = \bigg| \int_{\R^k} \Gamma_t^{\otimes k}(\bz,\bz') f(\bz') \bigg| \leq \norm{\Gamma_t}_{L^{\infty}(\R \times \R)}^k \| f \|_{L^1(\R^k)} \lesssim 1. 
\end{align*}
Meanwhile, one can check using the explicit formula for $\Gamma$ (or alternatively appeal to the well-known result of Aronson \cite{Aronson1968}), we see that there are positive constants $C, \sigma$ such that 
\begin{align*}
    \Gamma_t(z,z') \leq C \ov{\Gamma}^{\sigma}_t(z,z'), \quad 0 < t < 1, \quad z,z' \in \R
\end{align*}
where 
\begin{align*}
    \ov{\Gamma}^{\sigma}_t(z,z') = \frac{1}{\sqrt{2 \pi \sigma^2 t}} \exp\Big( \frac{-|z-z'|^2}{2 \sigma^2 t} \Big)
\end{align*}
is the transition kernel associated to the process $\sigma B$, with $B$ a standard Brownian motion. Thus, we have
\begin{align*}
  |P_t^{\otimes k}f(\bz)| &\lesssim \int_{\R^k} \frac{1}{t^{k/2}} \exp\Big( \frac{-|\bz-\bz'|^2}{2 \sigma^2 t} \Big)|\bz'|^{-(k-1)} d \bz' \\
  & = \int_{|\bz'| \le |\bz|/2} \frac{1}{t^{k/2}} \exp\Big( \frac{-|\bz-\bz'|^2}{2 \sigma^2 t} \Big)|\bz'|^{-(k-1)}  + \int_{|\bz'| \ge |\bz|/2} \frac{1}{t^{k/2}} \exp\Big( \frac{-|\bz-\bz'|^2}{2 \sigma^2 t} \Big)|\bz'|^{-(k-1)}  \\
  & \lesssim \frac{1}{t^{k/2}} \exp\Big( \frac{-|\bz|^2}{8 \sigma^2 t} \Big) \int_{|\bz'| \le |\bz|/2}  |\bz'|^{-(k-1)} d \bz' + |\bz|^{-(k-1)} \\
  & \lesssim \left(\frac{|\bz|}{t^{1/2}}\right)^k \exp\Big( \frac{-|\bz|^2}{2 \sigma^2 t} \Big) |\bz|^{-(k-1)} + |\bz|^{-(k-1)}  \lesssim |\bz|^{-(k-1)}, \quad 0 < t < 1, \quad \bz \in \R^k.
\end{align*}
We have thus verified that 
\begin{align*}
    P_tf(\bz) \lesssim |\bz|^{-(k-1)} + 1, \quad 0< t < \infty, \quad \bz \in \R^k , 
\end{align*}
and so we find that 
\begin{align*}
    \psi(t,\bz) = \int_0^t \exp\big(-c(t-s)\big) \big(P_s f(\bz) + 1\big) ds \lesssim  (|\bz|^{-(k-1)} +1) \int_0^t \exp\big(-c(t-s)\big) ds \lesssim |\bz|^{-(k-1)} + 1.
\end{align*}
\newline \newline 
\textit{Property (4).} We consider smooth sequences $f_n$ and $s_n$ that pointwise converge to $f$ and $\text{sign}(\cdot)$ respectively. We take $f_n$ to satisfy the same symmetry and monotonicity properties of $f$, and we take $s_n$ to be odd, smooth, and increasing. Let $\psi_n$ be the (smooth) solution to the Cauchy problem
\begin{align} \label{psi.duhamel.approx}
 \ds  \partial_t \psi_n - \frac{1}{2} \sum_{i = 1}^k \partial_{z^iz^i} \psi_n + c_0 \sum_{i = 1}^k s_n(z^i) \partial_{z^i} \psi_n = e^{-ct} \big(1 + f_n(\bz) \big) \text{ in } \R_+ \times \R^k, \quad \psi_n(0,\bz) = 0.
\end{align}

By the symmetry of $\psi_n$, we see that for each $i = 1,...,k$, the partial derivative $\psi_i = \psi_{n,i} \coloneqq \partial_{z^i} \psi_n$ satisfies the Dirichlet problem 
\begin{align*}
    \begin{cases} 
   \ds  \partial_t \psi_i - \frac{1}{2} \sum_{j = 1}^k \partial_{z^jz^j} \psi_i + c_0 \sum_{j = 1}^k s_n(z^j) \partial_{z^j} \psi_i + c_0  s'_n(z^i) \psi_i = e^{-ct} \partial_{z^i} f_n(\bz) \quad (t,\bz) \in \R_+ \times \{z^i > 0\}, 
    \\
   \ds  \psi_i(0,\bz) = 0, \quad \bz \in \{z^i \ge 0\}, \quad \psi_i(t,\bz) = 0, \quad \bz \in \{z^i = 0\}. 
    \end{cases}
\end{align*}
Since $\partial_{z^i} f_n(\bz) \leq 0$ for $z^i > 0$, it follows from the maximum principle that $\partial_{z^i} \psi_n \leq 0$ for $z^i > 0$. Passing to the limit $n \to \infty$, Property (4) follows. 
\newline \newline 
\textit{Property (5).} By combining lemmas \ref{lem.Ptf} and \ref{lem.Ptf2}, we find the estimate
\begin{align*}
    \big|\partial_{z^i} \big( P_t^{\otimes k} f \big) \big| + \big|\partial_{z^iz^i} \big( P_t^{\otimes k} f \big) \big| \lesssim \exp\big( - \frac{c_0^2 t}{2} \big), \quad (t,\bz) \in U. 
\end{align*}
In particular, the bound for $0 < t < 1$ comes from Lemma \ref{lem.Ptf2} and the bound for $t \geq 1$ comes from Lemma \ref{lem.Ptf}. 
Thus, for $(t,\bz) \in U$, 
\begin{align*}
   \big| \partial_{z^i} \psi(t,\bz) \big| &=  \bigg| \partial_{z^i} \bigg( \int_0^t e^{-c (t-s)} P_s f(\bz) ds \bigg) \bigg| = \bigg| \int_0^t e^{-c(t-s)} \partial_{z^i} P_s f(\bz) ds \bigg| 
   \\
   &\lesssim \int_0^t e^{-c(t-s)} e^{-\frac{c_0^2 s}{2}} ds  \leq \int_0^t \exp \Big( -(c \wedge c_0^2/2) t \Big) ds = t \exp \Big( -(c \wedge c_0^2/2) t \Big) \lesssim \exp(-c' t). 
\end{align*}
The bound on $\partial_{z^iz^i} \psi$ is similar, and then the bound on $\partial_t \psi$ comes from the equation for $\psi$. 
\end{proof}
\end{prop}

\begin{lem} \label{lem.Gammaderivs}
    Let $\Gamma$ be as in \eqref{def.Gamma}. For each fixed $z' \neq 0$, $(t,\bz) \mapsto \Gamma_t(\bz,\bz')$ is smooth on $\R_+ \times \R_+$. Moreover, for any $M > 0$, we have derivatives satisfying 
    \begin{align*}
     \sup_{ z \in (0,M), \, z' \in (-M,M) \setminus \{0\}}   \big| \partial_t \Gamma_t(z,z') \big| + \big| \partial_z \Gamma_t(z,z') \big| + \big| \partial_{zz} \Gamma_t(z,z') \big| \lesssim \exp\big(- \frac{c_0^2}{2} t\big), \quad t \geq 1.
    \end{align*}
\end{lem}

\begin{proof}
    We find it useful to write 
    \begin{align*}
        \Gamma_t(z,z') = \begin{cases}
            \ds \ov{\Gamma}_t(z,z' + c_0 t) + c_0 \exp\big( - 2c_0 z' \big) \int_{z + z'}^{\infty} \ov{\Gamma}_t(v, c_0 t) dv, & z > 0, \, z' > 0,
          \vspace{.2cm}  \\
            \ds \exp(2 c_0 z) \ov{\Gamma}_t(z,z' - c_0 t) + c_0 \exp(2c_0z') \int_{z-z'}^{\infty} \ov{\Gamma}_t(v,c_0 t) dv & z > 0, \, z' < 0, 
        \end{cases}
    \end{align*}
    where 
    \begin{align*}
        \ov{\Gamma}_t(z,z') = \frac{1}{\sqrt{2 \pi t}} \exp\big(- \frac{|z-z'|^2}{2t} \big)
    \end{align*}
    is the standard heat kernel on $\R$. Now, by explicit computation, 
    \begin{align*}
        \partial_z \Gamma_t(z,z') =  \partial_z \ov{\Gamma}_t(z,z' + c_0 t) - c_0 \exp(- 2c_0z') \ov{\Gamma}_t(z, - z' + c_0 t), \quad  z > 0, \, z' > 0, 
    \end{align*}
    and 
    \begin{align*}
          \partial_z \Gamma_t(z,z') &=    \ds 2c_0 \exp(2c_0 z) \ov{\Gamma}_t(z,z' -c_0 t) + \exp(2c_0 z  ) \partial_z \ov{\Gamma}_t(z,z' - c_0t) 
          \\
          &\qquad \qquad - c_0 \exp(2c_0 z') \ov{\Gamma}_t(z, z' + c_0 t), \quad  z > 0, \, z' < 0. 
    \end{align*}
    Thus, we see that for $t \geq 1$, $z > 0$, $z' > 0$ and $|z|, |z'| \leq M$, we have 
    \begin{align*}
        |\partial_z \Gamma_t(z,z')| &\leq |\partial_z \ov{\Gamma}_t(z,z' + c_0 t)| + c_0 \ov{\Gamma}_t(z,-z' + c_0 t)
        \\
        &= (2 \pi)^{-1/2} t^{-3/2} \big|z-z'-c_0 t\big| \exp\Big( - \frac{|z-z'- c_0 t|^2}{2t} \Big) + c_0 (2 \pi t)^{-1/2} \exp\Big( - \frac{|z + z'- c_0 t|^2}{2t} \Big)
        \\
        &\lesssim  \exp \Big( - \frac{|c_0 t - M|^2}{2t} \Big) \lesssim \exp\big(- (c_0^2/2) t\big). 
    \end{align*}
    The analogous bound for $z' < 0$ is similar and is omitted. 

    We now differentiate again, to find that 
    \begin{align*}
        \partial_{zz} \Gamma_t(z,z') = \partial_{zz} \ov{\Gamma}_t(z,z'  +c_0 t) - c_0 \exp( - 2c_0 z') \partial_z \ov{\Gamma}_t(z,-z'+ c_0 t), \quad z > 0, \, z' > 0, 
    \end{align*}
    and 
    \begin{align*}
        \partial_{zz} \Gamma_t(z,z') &= 4c_0^2 \exp(2c_0z) \ov{\Gamma}_t(z,z' - c_0 t) + 4 c_0 \exp(2c_0 z) \partial_z \ov{\Gamma}_t(z,z' - c_0 t)
        \\
        &\quad + \exp(2c_0z) \partial_{zz} \ov{\Gamma}_t(z,z' - c_0 t) - c_0 \exp(2c_0 z') \partial_z \ov{\Gamma}_t(z,z' + c_0 t), \quad z > 0, \, z' < 0. 
    \end{align*}
    In particular, uniformly over $t \geq 1$, $z > 0$, $z' > 0$ with $|z|,|z'| \leq M$, we argue as in the bound on $\partial_z \Gamma$ to obtain
    \begin{align*}
        \big|    \partial_{zz} \Gamma_t(z,z') \big| &\lesssim |\partial_{zz} \ov{\Gamma}_t(z,z' + c_0 t) | + |\partial_z \ov{\Gamma}_t(z,- z' + c_0 t)| \lesssim \exp\big(- (c_0^2/2) t\big).
    \end{align*}
    The corresponding bound for $z' < 0$ is similar, and finally the bound on $\partial_t \Gamma$ follows from the fact that
    \begin{align*}
        \partial_t \Gamma_t(z,z') = \frac{1}{2} \partial_{zz} \Gamma_t(z,z') - c_0  \text{sign}(z) \partial_z \Gamma_t(z,z'). 
    \end{align*}
    This completes the proof. 
\end{proof}

\begin{lem}\label{lem.Ptf}
    With the function $f$ as defined in the proof of Proposition \ref{prop.supersolexists}, for any $M > 0$, we have the bound
    \begin{align*}
     \sup_{\bz \in (0,M)^k} \Big\{ \big|  \partial_{z^i} \big( P_t^{\otimes k} f \big) \big| + \big| \partial_{z^iz^i} \big(P_t^{\otimes k} f \big) \big| \Big\} \lesssim \exp\big(- (c_0^2/2) t\big), \quad i = 1,...,k, \quad t \geq 1.
    \end{align*}
\end{lem}

\begin{proof}
    For $t \geq 1$ and $|\bz| \leq M$, we have
    \begin{align*}
        \partial_{z^i} P_t^{\otimes k} f &= \partial_{z^i} \int_{\R^k} f(\ov\bz) \prod_{j = 1}^k \Gamma_t(z^j,\ov{z}^j) d\ov{\bz} = \int_{\R^k} f(\ov \bz) \partial_z \Gamma_t(z^i,\ov{z}^i) \prod_{j \neq i} \Gamma_t(z^j,\ov{z}^j) d\ov{\bz} 
        \\
        & \leq \|f\|_{L^1(\R^k)} \norm{ \bz' \mapsto  \partial_z \Gamma_t(z^i,\ov{z}^i) \prod_{j \neq i} \Gamma_t(z^j,\ov{z}^j) }_{L^{\infty}\big( (-\delta/2,\delta/2)^k \big)} \lesssim \exp\big(\frac{-c_0^2 t}{2} \big), 
    \end{align*}
    where we used the fact that $f(\bz) = 0$ if $|z^j| > \delta/2$ for some $j$, and Lemma \ref{lem.Gammaderivs}. The estimate for $\partial_{z^iz^i} \big(P_t^{\otimes k} f \big)$ is similar. 
\end{proof}

\begin{lem}\label{lem.Ptf2}
    With the function $f$ and the domain $U$ as defined in the proof of Proposition \ref{prop.supersolexists}, there is a constant $C$ such that
    \begin{align*}
      \big|  \partial_{z^i} \big( P_t^{\otimes k} f \big) \big| + \big| \partial_{z^iz^i} \big(P_t^{\otimes k} f \big) \big| \lesssim C \text{ for } (t,\bz) \in U.
    \end{align*}
\end{lem}

\begin{proof}
    If $(t,\bz) \in U$, then by definition, there is at least one index $i' \in \{1,...,k\}$ such that $z^{i'} > \delta$. Let us first suppose that $i' \neq i$, and estimate $\partial_{z^i}\big(P_t^{\otimes k} f \big)$. As in the proof of Lemma \ref{lem.Ptf}, we start with the expression 
    \begin{align*}
        \partial_{z^i} P_t^{\otimes k} f(\bz) = \int_{\R^k} f(\ov \bz) \partial_z \Gamma_t(z^i, \ov z^i) \prod_{j \neq i} \Gamma_t\big(z^j, \ov z^j\big) d\ov{\bz}.
    \end{align*}
    Now, because $f \in L^1$ and $f(\ov{\bz}) = 0$ if $|\ov{z}^j| > \delta/2$ for any $j$, we see that for any $(t,\bz) \in U$, we have
    \begin{align*}
        \Big| \partial_{z^i} P_t^{\otimes k} f(\bz) \Big| \lesssim \norm{\partial_z \Gamma_t(z^i,\cdot)}_{L^{\infty}(\R)} \times  \prod_{j \neq i,i'} \norm{\Gamma_t(z^j,\cdot)}_{L^{\infty}(\R)} \times  \sup_{M> z > \delta, \, |z'| < \delta/2} \big| \Gamma_t(z,z') \big|. 
    \end{align*}
    Now, it is clear from the expressions for $\Gamma$ and $\partial_z \Gamma$ appearing in the proof of Lemma \ref{lem.Gammaderivs} that
    \begin{align*}
        &\norm{\Gamma_t(z^j, \cdot)}_{\infty} \lesssim t^{-1/2}, \quad 
        \norm{\partial_z \Gamma_t(z^i,\cdot)}_{L^{\infty}(\R)} \lesssim t^{-3/2}, \quad 0 < t < 1, 
    \end{align*}
    while 
    \begin{align*}
         \sup_{M> z > \delta, \, |z'| < \delta/2} \big| \Gamma_t(z,z') \big| &\lesssim  \frac{1}{\sqrt{t}} \exp\Big( - \frac{(z - z' - c_0 t)^2}{2t} \Big) \leq \frac{1}{\sqrt{t}} \exp\Big(- \frac{(\frac{\delta}{2} - c_0 t)^2}{2t} \Big) 
         \\
         & \leq \frac{1}{\sqrt{t}} \exp\Big( - \frac{\delta^2}{32t}\Big) \quad 0 < t < t_0 \coloneqq \frac{\delta}{4c_0}. 
    \end{align*}
    Thus, if $(t,\bz) \in U$ and $z^{i'} > \delta$ for some $i' \neq i$, then for $0 < t < t_0$, we have
    \begin{align} \label{partialziest}
        \Big| \partial_{z^i} P_t^{\otimes k} f(\bz) \Big| \lesssim t^{-k/2 - 1} \exp\Big( - \frac{\delta^2}{32t} \Big)\lesssim 1. 
    \end{align}
    Now, we suppose that $i = i'$. This time, we find that 
     \begin{align} \label{ieqi'comp}
        \Big| \partial_{z^i} P_t^{\otimes k} f(\bz) \Big| \lesssim  \prod_{j \neq i} \norm{\Gamma_t(z^j,\cdot)}_{\infty} \times  \sup_{M> z > \delta, \, |z'| < \delta/2} \big| \partial_{z} \Gamma_t(z,z') \big|. 
    \end{align}
    From the explicit expression for $\partial_z \Gamma_t$ appearing in the proof of Lemma \ref{lem.Ptf}, we see that for $M > z > \delta, 0 < z' < \delta/2$, and $0 < t < 1$, we have
    \begin{align*}
       \big|\partial_z \Gamma_t(z,z') \big| &\lesssim 
       t^{-3/2} \exp\Big( - \frac{|z - z' - c_0 t|^2}{2t} \Big) + t^{-1/2} \exp\Big(- \frac{|z + z' - c_0 t|^2}{2t} \Big)
       \\
       &\lesssim t^{-3/2} \exp\Big( - \frac{|\delta/2 - c_0 t|^2}{2t} \Big), 
    \end{align*}
    so that
    \begin{align} \label{shorttimedecay}
         \big|\partial_z \Gamma_t(z,z') \big| \lesssim t^{-3/2} \exp\Big( - \frac{|\delta/2 - c_0 t|^2}{2t} \Big) \leq t^{-3/2} \exp\Big( - \frac{\delta^2}{32t}\Big), \quad 0 < t < t_0 \coloneqq \frac{\delta}{4c_0}
    \end{align}
    holds for $M > z > \delta, 0 < z' < \delta/2$.
    Meanwhile, for $M > z > \delta, -\delta/2 < z' < 0$, and $0 < t < 1$, we instead have 
     \begin{align*}
       \big|\partial_z \Gamma_t(z,z') \big| &\lesssim 
       t^{-3/2} \exp\Big( - \frac{|z - z' + c_0 t|^2}{2t} \Big) + t^{-1/2} \exp\Big(- \frac{|z - z' + c_0 t|^2}{2t} \Big) 
       \\
       &\qquad \qquad \qquad \qquad + t^{-1/2} \exp\Big(- \frac{|z - z' - c_0 t|^2}{2t} \Big)
       \\
       &\lesssim t^{-3/2} \exp\Big( - \frac{|\delta - c_0 t|^2}{2t} \Big). 
    \end{align*}
    We thus deduce that \eqref{shorttimedecay} holds for any $M > z > \delta, |z'| < \delta/2$. Coming back to \eqref{ieqi'comp}, we see that in fact \eqref{partialziest} also holds in the case $i = i'$. 
    
    In other words, we have verified that for $0 < t < t_0 \coloneqq \delta/(4c_0)$, and any $(t,\bz) \in U$, the estimate \eqref{partialziest} holds. For $t > t_0$, the estimate follows easily from the boundedness of $\Gamma_t$ and $\partial_z \Gamma$. This completes the proof of the upper bound on $\partial_{z^i}\big(P_t^{\otimes k} f\big)$. The corresponding bound for $\partial_{z^iz^i} \big(P_t^{\otimes k} f \big)$ follows along similar lines, and is omitted. 
\end{proof}

\bibliographystyle{alpha}
\bibliography{fv}

\end{document}